\documentclass[10pt,a4paper]{amsart}
\usepackage{verbatim}
\usepackage{hyperref}
\usepackage[toc]{appendix}
\usepackage[T1]{fontenc}
\usepackage{graphicx}
\usepackage{enumerate}
\usepackage{amsmath,amsfonts,amssymb}
\usepackage{color}
\def\loc{\operatorname{loc}}
\usepackage{cite}
\definecolor{citation}{rgb}{0.11,0.67,0.84}
\definecolor{formula}{rgb}{0.1,0.2,0.6}
\definecolor{url}{rgb}{0.11,0.67,0.84}
\usepackage{pgf,tikz}
\usepackage{mathrsfs}
\usepackage{fancyhdr}
\usepackage{dutchcal}

\usepackage{dsfont}

\newcommand{\reqnomode}{\tagsleft@false}

\def\dx{\,{\rm d}x}

\def\ds{\,{\rm d}s}
\def\dt{\,{\rm d}t}
\def\dy{\,{\rm d}y}

\def \d{\,{\rm d}}
\def \diver{\,{\rm div}}
\def\dist{\,{\rm dist}}

\def\supp{\,{\rm supp}}
\def\diam{\,{\rm diam}}

\newcommand{\mf}[1]{\mathfrak{#1}}

\allowdisplaybreaks
\makeatletter
\DeclareRobustCommand*{\bfseries}{%
	\not@math@alphabet\bfseries\mathbf
	\fontseries\bfdefault\selectfont
	\boldmath
}

\makeatother

\newlength{\defbaselineskip}
\newcommand{\setlinespacing}[1]
{\setlength{\baselineskip}{#1 \defbaselineskip}}
\newcommand{\mint}{\mathop{\int\hskip -1,05em -\, \!\!\!}\nolimits}

\def\XXint#1#2#3{{\setbox0=\hbox{$#1{#2#3}{\int}$}
		\vcenter{\hbox{$#2#3$}}\kern-.5\wd0}}

\newtheorem{theorem}{Theorem}[section]

\newtheorem{remark}{Remark}[section]
\newtheorem{lemma}{Lemma}[section]
\newtheorem{proposition}{Proposition}[section]
\numberwithin{equation}{section}
\allowdisplaybreaks[4]

\newcommand{\kk}{\kappa}

\def\er{\mathbb R}

\newcommand{\ti}[1]{\tilde{#1}}

\newcommand\eps\varepsilon

\def\eqn#1$$#2$${\begin{equation}\label#1#2\end{equation}}

\newcommand{\be}{\begin{equation}}

	\newcommand{\ee}{\end{equation}}

\newcommand{\rr}{\varrho}

\newcommand{\snr}[1]{\lvert #1\rvert}

\newcommand{\nr}[1]{\lVert #1 \rVert}

\newcommand{\N}{\mathbb{N}}

\def\name[#1, #2]{#1 #2}

\title[Borderline global regularity for nonuniformly elliptic systems]{Borderline global regularity for nonuniformly elliptic systems}

\author[De Filippis]{Cristiana De Filippis}  \address{Cristiana De Filippis\\Dipartimento SMFI, Universit\`a di Parma, Viale delle Scienze 53/a, Campus, 43124 Parma, Italy} \email{\url{cristiana.defilippis@unipr.it}}
\author[Piccinini]{Mirco Piccinini}  \address{Mirco Piccinini\\Dipartimento SMFI, Universit\`a di Parma, Viale delle Scienze 53/a, Campus, 43124 Parma, Italy} \email{\url{mirco.piccinini@unipr.it}}
\begin{document}
	
	\subjclass[2020]{35J60, 49J45, 49N60\vspace{1mm}} 
	
	\keywords{Regularity, vectorial problems, nonuniform ellipticity, $(p,q)$-growth\vspace{1mm}}

	\thanks{{\it Acknowledgements.}\ The authors are supported by the INdAM GNAMPA project "Fenomeni non locali in problemi locali", CUP\_E55F22000270001.
		\vspace{1mm}}
	
	\maketitle

	\begin{abstract}
		We establish sharp global regularity results for solutions to nonhomogeneous, nonunifomrly elliptic systems with zero boundary conditions. In particular, we obtain everywhere Lipschitz continuity under borderline Lorentz assumptions on the forcing term, thus positively settling the optimality issue raised in \cite{bdms}.
		
	\end{abstract}
	\vspace{3mm}
	\setcounter{tocdepth}{1}
	
	\setlinespacing{1.08}
 \section{Introduction}\label{si}
 In this paper we prove global Lipschitz continuity and gradient higher differentiability up to the boundary for solutions to possibly degenerate/singular nonuniformly elliptic systems of the form
 \eqn{i.0}
 $$
 -\diver\ a(Du)=f\qquad  \mbox{in} \ \ \Omega,
 $$
 where $\Omega\subset \mathbb{R}^{n}$, $n\ge 3$, is an open, convex domain, the forcing term $f\colon \Omega\to \mathbb{R}^{N}$ satisfies suitable borderline regularity conditions and the (sufficiently smooth) nonlinear tensor $a\colon \mathbb{R}^{N\times n}\to \mathbb{R}^{N\times n}$ is nonuniformly elliptic in the sense that the associated ellipticity ratio
 $$
 \mathcal{R}(z):=\frac{\mbox{highest eigenvalue of} \ \partial a(z)}{\mbox{lowest eigenvalue of} \ \partial a(z)}
 $$
 blows up in correspondence of large values of the gradient variable. Moreover, $a(\cdot)$ features so-called radial "Uhlenbeck" structure in the sense that
 \eqn{i.1}
 $$
 a(z)=\ti{a}(\snr{z})z
 $$
 for some scalar weight function $\ti{a}(\cdot)$. This gives variational structure to \eqref{i.0} in the sense that it coincides with the Euler-Lagrange system of functional
 \begin{flalign}\label{fun}
 \mathcal{F}(w;\Omega):=\int_{\Omega}\left[F(Dw)-f\cdot w\right] \dx,
 \end{flalign}
 with $F(z)\equiv \ti{F}(\snr{z})$ and $\ti{F}(\cdot)$ being a primitive of $t\mapsto \ti{a}(t)t$. Let us enter into the details of the rich literature available on the regularity theory for vectorial problems as \eqref{i.0}. The first interior regularity result for systems the $p$-Laplacian type, i.e. $\ti{a}(t)=t^{p-2}$ in \eqref{i.1}, can be found in Uhlenbeck's seminal paper \cite{uh}, that lays down the foundations of the vectorial counterpart of Ladyzhenskaya \& Ural'tseva landmarking program on quasilinear elliptic equations \cite{LU,Ur}. In this respect, it is worth mentioning that in the genuine vectorial setting there is no hope of getting full regularity, at least for general systems not verifying \eqref{i.1}. In this respect, the first counterexamples were obtained by De Giorgi \cite{dg1} and by Giusti \& Miranda \cite{gimi}, that exhibited strongly elliptic, quadratic systems with discontinuous coefficients and non-smooth solutions. Those findings could lead to the erroneous conclusion that the existence of irregular solutions is due to the way the singularities of coefficients mix up with the components of the gradient variable. Anyway, such irregularity is a purely vectorial phenomenon and appears even when the nonlinear tensor $a(\cdot)$ is smooth, strongly elliptic and depends only on the gradient variable. Starting with these first fundamental results, the matter of regularity for systems driven by the $p$-Laplace operator, or by more general uniformly elliptic structures - meaning that their ellipticity ratio $\mathcal{R}(z)$ stays bounded uniformly with respect to all $z\in \mathbb{R}^{N\times n}$ - received lots of important contributions over the years, see \cite{kumig} for an overview of the most significant advances in the field. In particular, for nonhomogeneous systems it is crucial to establish precise criteria linking the regularity of the datum on the right-hand side to the regularity of solutions in an optimal fashion. For better understanding this issue, let us introduce the Lorentz space $L(n,1)$, that can be characterized as:
$$
w\in L(n,1) \ \Longleftrightarrow \ \nr{w}_{L(n,1)}:=\int_{0}^{\infty}\snr{\{x\in \mathbb{R}^{n}\colon \snr{w(x)}>t\}}^{1/n} \ \dt<\infty.
$$
Stein's related deep outcome \cite{stein} states that 
\eqn{xx.4}
$$
w\in W^{1,n}\ \ \mbox{and} \ \ Dw\in L(n,1) \ \Longrightarrow \ w \ \ \mbox{is continuous},
$$
so combining \eqref{xx.4} and the immersions $L^{n+\varepsilon}\hookrightarrow L(n,1) \hookrightarrow L^{n}$ for all $\varepsilon>0$, lead to the borderline interpretation of $L(n,1)$ as the limiting space with respect to the Sobolev embedding theorem. Implication \eqref{xx.4} has immediately a significant impact on the Poisson equation in connection with standard Calder\'on-Zygmund theory. In fact, it holds that
$$
-\Delta u=f\in L(n,1) \ \Longrightarrow \ Du \ \ \mbox{is continuous},
$$
which is sharp, according to Cianchi's counterexample \cite{CiGA}. It turns out that this result is not related to the linearity of the Poisson equation, nor in general to the specifics of the differential operator governing the system; it is rather a universal feature of elliptic structures. In fact, starting with the groundbreaking results of Duzaar \& Mingione and Kuusi \& Mingione, we see that the same result is true for uniformly elliptic PDE \cite{akm,ba2,dz,dumi2,dumi3,kumi1,kumig,kumi2,tn}, systems of differential forms \cite{sil}, fully nonlinear elliptic equations \cite{bm,dkm} and general systems without Uhlenbeck's structure \cite{by,kumi}. All the aforementioned results hold in the interior of the ambient space domain. Concerning global regularity, sharp results in terms of $L(n,1)$-data have been obtained by Cianchi \& Maz'ya in \cite{cm,cm1} for uniformly elliptic systems and by Kumar \cite{k} for equations in nondivergence form. We further refer to \cite{bl,beck2,bcds,dugrkr,jm,shouhl1} on global everywhere and partial regularity results for systems of the $p$-Laplacian type, to \cite{bcdm,c1,c2,cm2} about gradient higher differentiability in the uniformly elliptic setting and to \cite{dukrmi,km} on the existence of regular boundary points. So far we sketched quite a complete picture of the state of the art for local and global regularity in the uniformly elliptic framework, highlighting in particular that the limiting cases leading to global optimal regularity results are by now well understood. In sharp contrast, the regularity theory for nonuniformly elliptic problems is in large part still open and its treatment requires a different perspective relying on the identification of suitable constraints aimed at slowing down the rate of blow up of the ellipticity ratio to make it quantitatively comparable to the underlying energy. This is the viewpoint originally adopted by Marcellini in a series of seminal papers \cite{ma2,ma4,ma1,ma5}, where he reinterpreted nonuniform ellipticity in terms of so-called $(p,q)$-growth conditions. Precisely, Marcellini considers systems as in \eqref{i.0}-\eqref{i.1} driven by nonlinear tensors $a(\cdot)$ such that
\eqn{i.pq}
$$
\snr{z}^{p-2}\mathds{I}_{N\times n}\lesssim \partial a(z)\lesssim \snr{z}^{q-2}\mathds{I}_{N\times n} \quad \mbox{for} \ \ \snr{z}\ge 1  \ \ \mbox{and} \ \ 1<p\le q.
$$
The above inequality means that the only available bound on the ellipticity ratio is given by $\mathcal{R}(z)\lesssim 1+\snr{z}^{q-p}$, thus naturally suggesting that solutions to system \eqref{i.0} governed by $a(\cdot)$ as in \eqref{i.1} and \eqref{i.pq} are regular only for small values of the difference $q-p$. This is indeed the case: in \cite{ma4,ma1} was proven that necessary and sufficient condition for regular solutions to \eqref{i.0} under \eqref{i.1} and \eqref{i.pq} is that the size of the difference $q-p$ satisfies a restriction of the type
$q-p<\texttt{o}(n)$, where $\texttt{o}(n)\to 0$ as $n\to \infty$. Although Marcellini obtained Lipschitz continuity provided that
\eqn{i.pq1}
$$
0\le q-p<\frac{2p}{n},
$$
constraint later on relaxed by Bella \& Sch\"affner \cite{BS1,BS,s} as
\eqn{i.pq2}
$$
0\le q-p<\frac{2p}{n-1},
$$
the the optimal bound on exponents $(p,q)$ guaranteeing gradient boundedness for solutions to autonomous system as in \eqref{i.0}, \eqref{i.1} and \eqref{i.pq} is still unknown. Since Marcellini's pioneering papers, the regularity theory for functionals with $(p,q)$ growth has been object of intensive investigation, cf. \cite{BS,demi3,sharp,hisc,ko1,s} for an (incomplete) account of recent contributions and \cite{masu2} for a reasonable survey. Concerning nonhomogeneous problems of type \eqref{i.0}-\eqref{fun} or nonuniformly elliptic Stein type theorems, the first interior regularity result has been obtained by Beck \& Mingione \cite{bemi} in the case of autonomous functionals, see also \cite{BS1,bdms,ko} in this respect, \cite{demicz,demi1} for nonautonomous problems and \cite{deqc,ds,ts1} on nonconvex integrals. Let us point out that, apart from scalar cases or very special structures \cite{bo,cdk,djrr,rata,tdp}, global regularity for systems with $(p,q)$-growth is a recent remarkable achievement: in \cite{ko,ko1} Koch obtained higher integrability up to the boundary for free or constrained minima of \eqref{fun}, while B\"ogelein \& Duzaar \& Marcellini \& Scheven \cite{bdms} proved global gradient boundedness for solutions to systems with $(p,q)$-growth and (locally) homogeneous boundary conditions. Let us be more precise: the Lipschitz continuity result in \cite{bdms} holds for solutions to nondegenerate/nonsingular systems as in \eqref{i.0} with $L^{d}$-integrable right-hand side datum $f$ for some $d>n$, under \eqref{i.1}, \eqref{i.pq} and \eqref{i.pq1}. The approach developed in \cite{bdms} relies on Banerjee \& Lewis's technique \cite{bl} that allows dealing with homogeneous boundary conditions prescribed only on a part of the boundary and on Moser iteration that, as pointed out by the authors in \cite[Section 1.2]{bdms}, does not catch the borderline case $f\in L(n,1)$, which is nonetheless the expected one in the light of the results obtained by Cianchi \& Maz'ya \cite{cm1,cm} in the uniformly elliptic setting. There is also room for improving the bound on exponents $(p,q)$: in fact, Koch \cite{ko1} proved global gradient higher integrability for minima of autonomous functionals with \eqref{i.pq2} in force. The contributions to the boundary regularity theory for nonuniformly elliptic problems of our paper are threefold: first, we positively settle the sharpness issue raised in \cite{bdms} on the forcing term $f$ by proving global Lipschitz continuity for solutions to \eqref{i.0} with $f\in L(n,1)$; second, we improve the bound from \eqref{i.pq1} to \eqref{i.pq2}; third, our results cover possibly degenerate or singular systems as well. In fact, the main outcome of this paper is the following
\begin{theorem}
 \label{t1}
 Under assumptions \eqref{om} and \eqref{f.0}-\eqref{ff}, let $u\in W^{1,p}(\Omega,\mathbb{R}^{N})$ be a minimizer of functional \eqref{fun} in its own Dirichlet class. There is a threshold radius $\rr_{*}\equiv \rr_{*}(n,\Omega)>0$ such that if for some $x_{0}\in \partial \Omega$ and $0<\rr\le \rr_{*}$ it is
 $$
 u=0 \ \ \mbox{on} \ \ \partial \Omega\cap B_{\rr}(x_{0}) \ \ \mbox{in the sense of traces},
 $$
 then, whenever $B_{\sigma}(x_{0})\subset B_{\rr}(x_{0})$ are concentric balls with $0<\sigma<\rr$ it is $u\in W^{1,\infty}(\Omega\cap B_{\sigma}(x_{0}),\mathbb{R}^{N})$ with
 \begin{flalign}\label{lipfin}
 \nr{Du}_{L^{\infty}(\Omega\cap B_{\sigma}(x_{0}))}\le& c(\rr-\sigma)^{-a_{1}}\left(\mathcal{F}(u;\Omega\cap B_{\rr}(x_{0}))+\nr{u}_{W^{1,p}(\Omega\cap B_{\rr}(x_{0}))}^{p}+1\right)^{d_{1}}\nonumber \\
 &+c(\rr-\sigma)^{-a_{2}}[f]_{n,1;\Omega\cap B_{\rr}(x_{0})}^{d_{2}},
 \end{flalign}
 where $c\equiv c(n,N,\Omega,\Lambda,p,q)$, $a_{1},a_{2}\equiv a_{1},a_{2}(n,p,q)$ and $d_{1},d_{2}\equiv d_{1},d_{2}(n,p,q)$.
 \end{theorem}
The proof of Theorem \ref{t1} shares the same premises of \cite[Theorem 1.1]{bl} and \cite[Theorem 1.1]{bdms}, but instead of deriving gradient boundedness via Moser iteration, we adopt a global De Giorgi type iterative technique in the spirit of those employed in \cite{bemi,BS1,demi1} for local regularity results. The advantage in using De Giorgi's technique instead of Moser iteration lies in the possibility of running the iterative process via dyadic sums rather than using geometric progressions. In the summation procedure the nonhomogeneity $f$ contributes to the construction of a Wolff type potential that embeds in $L(n,1)$, \cite{hm1,kumi1}. As a byproduct of Theorem \ref{t1} we also obtain higher differentiability for suitable nonlinear functions of the gradient. This is the content of
 \begin{theorem}\label{t2}  Under the same assumptions of Theorem \ref{t1}, the following higher differentiability results hold true. If $1<p\le 2$ and $B_{\sigma}(x_{0})\subset B_{\rr}(x_{0})$ are concentric balls with $0<\sigma<\rr\le \rr_{*}$, then $u\in W^{2,2}(\Omega\cap B_{\sigma}(x_{0}),\mathbb{R}^{N})$, $V_{\mu}(Du)\in W^{1,2}(\Omega\cap B_{\sigma}(x_{0}),\mathbb{R}^{N\times n})$ with
     \begin{flalign}\label{t2.1}
     &\nr{DV_{\mu}(Du)}_{L^{2}(\Omega\cap B_{\sigma}(x_{0}))}+ \nr{D^{2}u}_{L^{2}(\Omega\cap B_{\sigma}(x_{0}))}\nonumber\\
     &\qquad \qquad \le c(\rr-\sigma)^{-a}\left(\mathcal{F}(u;\Omega\cap B_{\rr}(x_{0}))+\nr{u}_{W^{1,p}(\Omega\cap B_{\rr}(x_{0}))}^{p}+1\right)^{d_{1}}\left([f]_{n,1;\Omega\cap B_{\rr}(x_{0})}^{d_{2}}+1\right),
     \end{flalign}
     for $c\equiv c(n,N,\Omega,\Lambda,p,q)$, $a\equiv a(n,p,q)$ and $d_{1},d_{2}\equiv d_{1},d_{2}(n,p,q)$.\\
     If $p>2$, then $W_{\mu}(Du)\in W^{1,2}(\Omega\cap B_{\sigma}(x_{0}),\mathbb{R}^{N\times n})$ with
     \begin{flalign}\label{t2.2}
     \nr{DW_{\mu}(Du)}_{L^{2}(\Omega\cap B_{\sigma}(x_{0}))}\le& c(\rr-\sigma)^{-a_{1}}\left(\mathcal{F}(u;\Omega\cap B_{\rr}(x_{0}))+\nr{u}_{W^{1,p}(\Omega\cap B_{\rr}(x_{0}))}^{p}+1\right)^{d_{1}}\nonumber \\
     &+c(\rr-\sigma)^{-a_{2}}[f]_{n,1;\Omega\cap B_{\rr}(x_{0})}^{d_{2}},
     \end{flalign}
     for $c\equiv c(n,N,\Omega,\Lambda,p,q)$, $a_{1},a_{2}\equiv a_{1},a_{2}(n,p,q)$ and $d_{1},d_{2}\equiv d_{1},d_{2}(n,p,q)$ - in particular, if $\mu>0$ in \eqref{f.1}, then $u\in W^{2,2}(\Omega\cap B_{\sigma}(x_{0}),\mathbb{R}^{N})$ and
     \begin{flalign}\label{t2.3}
     &\nr{D^{2}u}_{L^{2}(\Omega\cap B_{\sigma}(x_{0}))}+\nr{DW_{\mu}(u)}_{L^{2}(\Omega\cap B_{\sigma}(x_{0}))} \nonumber\\
     &\qquad \qquad \le c(\rr-\sigma)^{-a_{1}}\left(\mathcal{F}(u;\Omega\cap B_{\rr}(x_{0}))+\nr{u}_{W^{1,p}(\Omega\cap B_{\rr}(x_{0}))}^{p}+1\right)^{d_{1}}+c(\rr-\sigma)^{-a_{2}}[f]_{n,1;\Omega\cap B_{\rr}(x_{0})}^{d_{2}},
     \end{flalign}
     is verified with $c\equiv c(n,N,\Omega,\Lambda,p,q,\mu)$, $a_{1},a_{2}\equiv a_{1},a_{2}(n,p,q)$ and $d_{1},d_{2}\equiv d_{1},d_{2}(n,p,q)$.
 \end{theorem}
 We refer to Sections \ref{bt} and \ref{sa} for more details on the various quantities appearing in the statement of Theorems \ref{t1}-\ref{t2}. Let us point out that while in the singular case $1<p\le 2$ Theorem \ref{t2} yields higher differentability for the usual nonlinear vector field $V_{\mu}(Du)$ incorporating the scaling features of the $p$-Laplacian; in the degenerate one $p>2$ we need to consider a different quantity $W_{\mu}(Du)$, still related to the $p$-Laplace operator that seems to better fits degenerate problems, cf. \cite{bcdm,cm2}. We conclude by pointing out that the approach presented in this paper allows extending the results of Theorems \ref{t1}-\ref{t2} to a more general class of nonuniformly elliptic systems than those with $(p,q)$-growth. In fact, with minor modification to the approximation scheme designed in Section \ref{as}\footnote{The approximation procedure described in Section \ref{as} is based on the mollification of the integrand in order to make it smooth and on the rebalancing of its growth by adding a term proportional to the $L^{q}$-energy via an infinitesimal parameter. In alternative, one could correct the growth of the integrand as done in \cite{bemi,demi1} by connecting it at infinity to the $L^{p}$-energy with a $W^{1,\infty}$-junction and then regularize via convolution the truncated integrand. This produces a family of smooth, uniformly elliptic problems, to whom the estimates derived in Section \ref{apr} apply, and that suitably converge to the original problem even if $\Delta_{2}$ or $\nabla_{2}$ conditions fail.} it is possible to apply the a priori estimates obtained in Section \ref{apr} to systems with fast exponential growth or to second order elliptic differential operators governed by Orlicz functions not satisfying $\Delta_{2}$ or $\nabla_{2}$ conditions, see \cite{bemi,demi1,mm2} for more details.
 \section{Preliminaries}
 \subsection{Notation} 
 In this paper, $\Omega\subset \mathbb{R}^{n}$ will always be (at least) an open, bounded, convex set and $n\ge 3$. We denote by $c$ a general constant larger than one possibly changing from line to line and depending on the data of the problem. Special occurrences will be highlighted as $\ti{c}$, $c_{*}$, $c'$ or the like. Specific dependencies on certain parameters will be outlined by putting them in parentheses, i.e. $c\equiv c(n,p)$ means that $c$ depends on $n$ and $p$. By $ B_r(x_0):= \{x \in \er^n  :   |x-x_0|< r\}$ we indicate the open ball with center in $x_0$ and radius $r>0$; we shall avoid denoting the center when this is clear from the context, i.e., $B \equiv B_r \equiv B_r(x_0)$; this happens in particular with concentric balls. The Lebesgue measure of the $n$-dimensional unitary ball is denoted by $\omega_{n}$, $\texttt{tr}(\cdot)$ refers to the usual trace operator and with $\Phi$ being a Lipschitz-regular map, $\mathcal{J}_{\Phi}$ denotes its Jacobian. We further introduce symbols
 \begin{flalign}\label{*}
 p^{*}_{\mf{s}}:=\begin{cases}
 \ p(n-1)/(n-1-p)\quad &\mbox{if} \ \ n-1>p \\
 \ p+p\nu^{-1}\quad &\mbox{if} \ \ n-1=p,
 \end{cases}\qquad \quad p^{*}:=\begin{cases}
 \ pn/(n-p)\quad &\mbox{if} \ \ n>p\\
 \ p+p\nu^{-1}\quad &\mbox{if} \ \ n=p,
 \end{cases}
 \end{flalign}
 for arbitrary $\nu\in (0,1/2)$, representing the $(n-1)$-dimensional and the $n$-dimensional critical $p$-Sobolev exponent respectively. Moreover, given parameters $\delta,\varepsilon\in (0,1]$, saying that a sequence $\{a_{\varepsilon}\}_{\varepsilon\in (0,1]}$ is $\texttt{o}(\varepsilon)$ means that $a_{\varepsilon}\to_{\varepsilon\to 0}0$, while a sequence $\{b_{\delta\varepsilon}\}_{\delta,\varepsilon\in (0,1]}$ is such that $b_{\delta\varepsilon}=\texttt{o}_{\varepsilon}(\delta)$ when $b_{\delta\varepsilon}\to_{\delta\to 0}0$ for fixed $\varepsilon$, whilst we say that $b_{\delta\varepsilon}=\texttt{o}(\delta,\varepsilon)$ if $b_{\delta\varepsilon}\to_{\delta\to 0}b_{\varepsilon}$ for fixed $\varepsilon$ and $b_{\varepsilon}=\texttt{o}(\varepsilon)$. 
 
 \subsection{Basic tools for the $p$-Laplacian}\label{bt}
 With $s\ge 0$ and $z\in \mathbb{R}^{N\times n}$ we shall always set $\ell_{s}(z):=\sqrt{s^{2}+\snr{z}^{2}}$. Moreover, we will often employ the vector fields
 $$
 V_{s}(z):=(s^{2}+\snr{z}^{2})^{\frac{p-2}{4}}z\qquad \mbox{and}\qquad W_{s}(z):=(s^{2}+\snr{z}^{2})^{\frac{p-2}{2}}z,
 $$
 that encode the scaling features of the $p$-Laplacian. It is well-known that if $w$ is a $W^{2,2}$-regular map, the equivalences
\begin{flalign}\label{dvp}
 \ell_{s}(Dw)^{\frac{p-2}{2}}\snr{D^{2}w}\stackrel{p>1}{\approx} \snr{DV_{s}(Dw)},\qquad \qquad \ell_{s}(Dw)^{p-2}\snr{D^{2}w}\stackrel{p>3/2}{\approx} \snr{DW_{s}(Dw)}
\end{flalign}
and 
\begin{flalign}\label{dvp.1}
\begin{cases}
\ \snr{V_{s}(z_{1})-V_{s}(z_{2})}\stackrel{p>1}{\approx}\left(s^{2}+\snr{z_{1}}^{2}+\snr{z_{2}}^{2}\right)^{\frac{p-2}{4}}\snr{z_{1}-z_{2}}\\
\ \snr{W_{s}(z_{1})-W_{s}(z_{2})}\stackrel{p>3/2}{\approx}\left(s^{2}+\snr{z_{1}}^{2}+\snr{z_{2}}^{2}\right)^{\frac{p-2}{2}}\snr{z_{1}-z_{2}}
\end{cases}
\end{flalign}
hold for all $z,z_{1},z_{2}\in \mathbb{R}^{N\times n}$, with constants implicit in "$\approx$" depending only on $(n,N,p)$, cf. \cite[Section 2]{ham}.
We conclude this section by recalling a classical iteration lemma.
\begin{lemma} \label{l5}
Let $h\colon [t,s]\to \mathbb{R}$ be a non-negative and bounded function, and let $a,b, \gamma$ be non-negative numbers. Assume that the inequality 
$ 
h(\tau_1)\le  (1/2) h(\tau_2)+(\tau_2-\tau_1)^{-\gamma}a+b,
$
holds whenever $t\le \tau_1<\tau_2\le s$. Then $
h(t)\le c( \gamma)[a(s-t)^{-\gamma}+b]
$, holds too. 
\end{lemma}
 \subsection{Convex domains}\label{cd}
We assume that
\begin{flalign}\label{om}
\Omega\subset \mathbb{R}^{n}, \ \ n\ge 3, \ \ \mbox{is an open, bounded, convex domain}. 
\end{flalign}
For $x_{0}\in \mathbb{R}^{n}$, $\sigma\in (0,\diam(\Omega)]$, let us introduce the density function  $\Gamma_{\Omega}(x_{0};\sigma):=\sigma^{n}\snr{\Omega\cap B_{\sigma}(x_{0})}^{-1}$, defined whenever $\snr{\Omega\cap B_{\sigma}(x_{0})}>0$. If $\Omega_{1}$, $\Omega_{2}$ are two sets and $B_{\sigma_{1}}(x_{1})$, $B_{\sigma_{2}}(x_{2})$ two balls such that $0<\theta\sigma_{1}\le \sigma_{2}\le \sigma_{1}\le \min\{\diam(\Omega_{1}),\diam(\Omega_{2})\}$ for some $\theta\in (0,1)$, $\Omega_{2}\cap B_{\sigma_{2}}(x_{2})\subset \Omega_{1}\cap B_{\sigma_{1}}(x_{1})$ and $\snr{\Omega_{2}\cap B_{\sigma_{2}}(x_{2})}>0$, then
\eqn{2.2.2}
 $$
 \Gamma_{\Omega_{1}}(x_{1};\sigma_{1})=\frac{\sigma_{1}^{n}}{\snr{\Omega_{1}\cap B_{\sigma_{1}}(x_{1})}}\le \frac{\sigma_{2}^{n}}{\theta^{n}\snr{\Omega_{2}\cap B_{\sigma_{2}}(x_{2})}}=\theta^{-n}\Gamma_{\Omega_{2}}(x_{2};\sigma_{2}).
 $$
Convex domains feature reasonable uniformity properies in terms of the density function $\Gamma_{\Omega}$, as the next lemma shows.
\begin{lemma}
Let $\Omega\subset \mathbb{R}^{n}$ be a convex domain. There exists a positive constant $\Upsilon_{\Omega}\equiv \Upsilon_{\Omega}(n,\Omega)$ such that
\begin{flalign}\label{2.2}
 \Gamma_{\Omega}(x_{0};\sigma):=\frac{\sigma^{n}}{\snr{\Omega\cap B_{\sigma}(x_{0})}}\le \Upsilon_{\Omega}\qquad \mbox{for all} \ \ x_{0}\in \bar{\Omega}, \ \ \sigma\in (0,\diam(\Omega)].
 \end{flalign}
\end{lemma}
\begin{proof}
Since a convex domain satisfies the interior cone condition, \eqref{2.2} holds for all $x_{0}\in \partial \Omega$ and any $\sigma\in (0,\diam(\Omega)]$ for some $\tilde{\Upsilon}_{\Omega}\equiv \ti{\Upsilon}_{\Omega}(n,\Omega)>0$, see \cite[Section 2.1]{bdms}. Next, if $x_{0}\in \Omega$ with $\dist(x_{0},\partial\Omega)>0$, we take any $\sigma\in (0,\diam(\Omega)]$ and distinguish two cases: $\dist(x_{0},\partial\Omega)>\sigma/4$ and $0<\dist(x_{0},\partial\Omega)\le \sigma/4$. If $\dist(x_{0},\partial\Omega)>\sigma/4$, then $B_{\sigma/8}(x_{0})\Subset \Omega$ and obviously $\Omega\cap B_{\sigma/8}(x_{0})\subset \Omega\cap B_{\sigma}(x_{0})$, so $\snr{\Omega\cap B_{\sigma/8}(x_{0})}=\snr{B_{\sigma/8}(x_{0})}$ and by \eqref{2.2.2} inequality $\Gamma_{\Omega}(x_{0};\sigma)\le 2^{3n}\Gamma_{\Omega}(x_{0};\sigma/8)=2^{3n}\omega_{n}^{-1}$ holds true. In case $0<\dist(x_{0},\partial \Omega)\le \sigma/4$, we let $\ti{x}_{0}\in \partial \Omega$ verifying $\snr{x_{0}-\ti{x}_{0}}=\dist(x_{0},\partial \Omega)$ and observe that if $x\in \Omega\cap B_{\sigma/8}(\ti{x}_{0})$, then $\snr{x-x_{0}}\le \snr{x-\ti{x}_{0}}+\snr{\ti{x}_{0}-x_{0}}\le \dist(x_{0},\partial \Omega)+\sigma/8\le (\sigma/4)+(\sigma/8)<\sigma$, so $\Omega\cap B_{\sigma/8}(\ti{x}_{0})\subset \Omega\cap B_{\sigma}(x_{0})$. Since \eqref{2.2} is true for boundary points we estimate
$\Gamma_{\Omega}(x_{0};\sigma)\le \sigma^{n}\snr{\Omega\cap B_{\sigma/8}(\ti{x}_{0})}^{-1}=2^{3n}\Gamma_{\Omega}(\ti{x}_{0};\sigma/8)\le 2^{3n}\ti{\Upsilon}_{\Omega}$, and \eqref{2.2} is verified for all $x_{0}\in \bar{\Omega}$ with $\Upsilon_{\Omega}:=2^{3n}\max\{\ti{\Upsilon}_{\Omega},\omega_{n}^{-1}\}$.
\end{proof}
Combining \eqref{2.2.2}-\eqref{2.2} with the relative isoperimetric inequality \cite[Theorems 5.10 and 5.11]{eg}, we obtain that if $\Omega$ is a convex domain, there is $\rr':=\min\left\{\diam(\Omega),\left(\snr{\Omega}/(2\omega_{n})\right)^{1/n}\right\}$ such that
 \eqn{2.2.1}
 $$\mathcal{H}^{n-1}(\partial B_{\sigma}(x_{0})\cap \Omega)\approx\sigma^{n-1},\qquad\quad \snr{\Omega\cap B_{\sigma}(x_{0})}\approx \snr{B_{\sigma}(x_{0})}, $$ for all $x_{0}\in \bar{\Omega}$, $\sigma\in (0,\rr']$, with constants implicit in "$\approx$" depending only on $(n,\Omega)$. It is worth stressing that if $B_{\sigma}(x_{0})\subset \mathbb{R}^{n}$ is a ball so that $0<\sigma\le \rr'$ and $\Gamma_{\Omega}(x_{0};\sigma)\le \Upsilon_{\Omega}$, then \eqref{2.2.1} holds at that radius. We further recall that if $\Omega$ is a $C^{2}$-regular convex set, $\mf{v}_{\Omega}\in C^{1}(\partial \Omega,\mathbb{R}^{n})$ is its outer unitary normal vector, $\mathbf{S}_{\Omega}$ is the second fundamental form of $\partial \Omega$, and $w\in C^{1}(\bar{\Omega},\mathbb{R}^{n})$ is a vector field whose tangential components vanish on $\partial \Omega$, it holds that
\begin{flalign}\label{3.8}
\langle w,\mf{v}_{\Omega}\rangle\diver \ w-\langle\partial_{w}w,\mf{v}_{\Omega}\rangle=-\texttt{tr}(\mathbf{S}_{\Omega})\langle w,\mf{v}_{\Omega}\rangle^{2},\qquad \qquad \texttt{tr}(\mathbf{S}_{\Omega})\le 0,
\end{flalign}
see \cite[Section 2.4]{bdms} for the general form of \eqref{3.8}. An important feature of convex sets is the possibility of approximating them from outside via smooth, convex sets. More precisely, with $\Omega\subset \mathbb{R}^{n}$ as in \eqref{om}, by \cite[Lemma 5.2]{cm2}, see also \cite[Section 4.1]{bdms}, there is a positive limiting radius $\rr''\equiv \rr''(n,\Omega)$ and a sequence $\{\Omega_{\varepsilon}\}_{\varepsilon\in (0,1]}$ of smooth, convex domains such that
\begin{flalign}
\left\{
\begin{array}{c}
\displaystyle 
\ \Omega\subset \Omega_{\varepsilon} \ \ \mbox{for all} \ \ \varepsilon\in (0,1],\qquad \quad \snr{\Omega_{\varepsilon}\setminus \Omega}\to_{\varepsilon\to 0}0,\qquad \quad \texttt{d}_{\varepsilon}:=\dist_{\mathcal{H}}(\Omega_{\varepsilon},\Omega)\to_{\varepsilon\to 0}0\\[8pt]\displaystyle
\ \sup_{\varepsilon\in (0,1]}\frac{\sigma^{n}}{\snr{\Omega_{\varepsilon}\cap B_{\sigma}(x_{0})}}\stackrel{\eqref{2.2}}{\le} \Upsilon_{\Omega} \ \ \mbox{for all} \ \ \sigma\in (0,\rr''),\ \ x_{0}\in \partial \Omega,
\end{array}
\right.\label{3.8.2}
\end{flalign}
and
\begin{flalign}\label{3.8.2.1}
\begin{array}{c}
\mbox{the Lipschitz constants of the graphs locally describing $\partial \Omega_{\varepsilon}$} \\
\mbox{are uniformly bounded in terms of those of $\Omega$},
\end{array}
\end{flalign}
where $\dist_{\mathcal{H}}(\Omega_{\varepsilon},\Omega)$ denotes the Hausdorff distance between $\Omega_{\varepsilon}$ and $\Omega$. Set $\rr_{*}:=\min\{1,\rr',\rr''\}$. From now on, any occurrence of the symbol $\rr_{*}$ throughout the paper - even if not explicitly specified - refers to the threshold radius that we have just introduced. We conclude this section by recording a Sobolev-Poincar\'e inequality for convex sets, cf. \cite[Lemma 2.3]{bdms}.
\begin{lemma}
Let $K\subset \mathbb{R}^{n}$ be a bounded, open, convex set and $1\le p\le n$. Then for any $w\in W^{1,p}(K)$ it holds
\begin{flalign}\label{pk}
\left(\mint_{K}\snr{w}^{p^{*}}\dx\right)^{\frac{1}{p^{*}}}\le \frac{c\diam (K)^{n}\snr{K}^{\frac{1}{n}}}{\snr{K}}\left(\mint_{K}\snr{Dw}^{p}\dx\right)^{\frac{1}{p}}+c\left(\mint_{K}\snr{w}^{p}\dx\right)^{\frac{1}{p}},
\end{flalign}
where $p^{*}$ has been defined in \eqref{*} and it is $c\equiv c(n,p,\nu)$.
\end{lemma}
\subsection{Structural assumptions}\label{sa} We assume that the integrand $F\colon \mathbb{R}^{N\times n}\to \mathbb{R}$ has radial (Uhlenbeck) structure, i.e., there exists $\ti{F}\colon [0,\infty)\to [0,\infty)$ such that
\begin{flalign}\label{f.0}
\begin{cases}
\ F(z)=\ti{F}(\snr{z})\\
\ \ti{F}\in C^{1}_{\loc}[0,\infty)\cap C^{2}_{\loc}(0,\infty)
\end{cases}
\end{flalign}
and the following growth/ellipticity conditions are satisfied:
\begin{flalign}\label{f.1}
\begin{cases}
\ \Lambda^{-1}\ell_{\mu}(z)^{p}\le F(z)\le \Lambda\ell_{\mu}(z)^{p}+\Lambda \ell_{\mu}(z)^{q}\\
\ \Lambda^{-1}\ell_{\mu}(z)^{p-2}\snr{\xi}^{2}\le \langle \partial^{2}F(z)\xi,\xi\rangle\\
\ \snr{\partial^{2}F(z)}\le \Lambda\ell_{\mu}(z)^{p-2}+\Lambda \ell_{\mu}(z)^{q-2},
\end{cases}
\end{flalign}
for all $z\in \mathbb{R}^{N\times n}\setminus \{0\}$, $\xi\in \mathbb{R}^{N\times n}$ and some absolute constants $\mu\in [0,1]$, $\Lambda\ge 1$. In \eqref{f.1}, exponents $(p,q)$ verify the bound
\begin{flalign}\label{pq}
1<p\le q<p+\min\left\{\frac{2p}{n-1},\frac{4(p-1)}{(n-3)}\right\}
\end{flalign}
Finally, the forcing term $f\colon \Omega\to \mathbb{R}^{N}$ in \eqref{fun} matches the borderline Lorentz condition
\begin{flalign}\label{ff}
\snr{f}\in L(n,1)(\Omega).
\end{flalign}
\begin{remark}\label{rf}
\emph{Let us point out that, up to replace it with $f\mathds{1}_{\Omega}$, we can assume with no loss of generality that $f$ is defined on the whole $\mathbb{R}^{n}$ and $f\equiv 0$ on $\mathbb{R}^{n}\setminus \Omega$}. 
\end{remark}\section{Global De Giorgi type iteration and nonlinear potentials}\label{dgp}
We start by recalling some well-known facts about Lorentz spaces. Let $B\subset \mathbb{R}^{n}$ be an open, bounded set and $f\colon B\to \mathbb{R}^{N}$ be a measurable function. Up to redefine $f$ as $\ti{f}:=f\mathds{1}_{B}$, there is no loss of generality in assuming that $f$ is defined on the whole $\mathbb{R}^{n}$. The nonincreasing rearrangement of $f$ is defined by
 $$
 f^{*}(t):=\inf\left\{\sigma\in (0,\infty)\colon \snr{\{x\in B\colon \snr{f(x)}>\sigma\}}\le t\right\}.
 $$
 Moreover, we say that $f$ belongs to the Lorentz space $L(n,1)(B)$ iff
 $$
 [f]_{n,1;B}:=\int_{0}^{\snr{B}}t^{1/n}f^{*}(t)\frac{\dt}{t}<\infty.
 $$
 We further recall that if $\supp(f)\subset B$, then
 $$
 \int_{B}\snr{f(x)}^{m}\dx=\int_{0}^{\snr{B}}(f^{*}(s))^{m}\dx\qquad \mbox{for all} \ \ m\in [1,\infty),
 $$
 which implies that
 \eqn{2.0.1}
 $$
 \sup\left\{\int_{A}\snr{f(x)}^{m}\dx\colon A\subset B, \ \snr{A}\le t\right\}\le \int_{0}^{t}(\ti{f}^{*}(s))^{m}\ds,
 $$
 see \cite[Section 2.2]{BS1} or \cite[Section 2.3]{kumi1} for more details on this matter. Next, we manipulate \cite[Lemma 3]{BS1} to derive its global version, see also \cite[Lemma 3.1]{bemi}.  
 
 \begin{lemma}\label{i++}
 Let $\Omega\subset \mathbb{R}^{n}$ be as in \eqref{om}, $B_{\rr}(x_{0})$ be a ball such that $0<\rr\le \rr_{*}$ and $\Gamma_{\Omega}(x_{0};\rr/2)\le \Upsilon_{\Omega}$, and $w\in W^{1,2}(\Omega\cap B_{\rr}(x_{0}))\cap C(\bar{\Omega}\cap B_{\rr}(x_{0}))$, $f\in L(n,1)(\Omega,\mathbb{R}^{N})$ with $\supp(f)\Subset \Omega$ be functions. If for parameters $\rr/2\le \tau_{1}<\tau_{2}\le \rr$ the nonhomogeneous Caccioppoli type inequality 
 \begin{flalign}\label{2.0}
 \int_{\Omega\cap B_{\rr}(x_{0})}\eta^{2}\snr{D(w-\kk)_{+}}^{2}\dx\le& c_{1}M_{1}^{2}\int_{\Omega\cap B_{\rr}(x_{0})}\snr{D\eta}^{2}(w-\kk)_{+}^{2}\dx\nonumber \\
 &+c_{2}M_{2}^{2}\int_{\Omega\cap B_{\rr}(x_{0})}\mathds{1}_{\{w>\kk\}}\eta^{2}\snr{f}^{2}\dx
 \end{flalign}
 is verified for all nonnegative $\eta\in C^{1}_{c}(B_{\tau_{2}}(x_{0}))$, any $\kk\ge \kk_{0}\ge 0$ and some positive, absolute constants $c_{1}, c_{2},M_{2}\ge 0$ and $M_{1}\ge 1$, then it holds that
 \begin{flalign}\label{2.1}
 \nr{w}_{L^{\infty}(\Omega\cap B_{\tau_{1}}(x_{0}))}\le& \kk_{0}+\frac{cM_{1}^{1+\max\left\{\nu,\frac{n-3}{2}\right\}}\nr{(w-\kk_{0})_{+}}_{L^{2}(\Omega\cap B_{\tau_{2}}(x_{0}))}}{\rr^{\alpha}(\tau_{2}-\tau_{1})^{\beta}}\nonumber \\
 &+\frac{cM_{1}^{\max\left\{\nu,\frac{n-3}{2}\right\}}M_{2}[f]_{n,1;B_{\rr}(x_{0})}}{\rr^{\alpha}(\tau_{2}-\tau_{1})^{\beta}},
 \end{flalign}
 for all $\nu\in (0,1/2)$, with $c\equiv c(n,\nu,\Omega,c_{1},c_{2})$ and $\alpha,\beta\equiv \alpha,\beta(n,\nu)$.
 \end{lemma}
 \begin{proof}
 The proof of \eqref{2.1} is quite involved and requires some technical work to be carried out, for this reason we shall split it into two steps eventually leading to \eqref{2.1}.
 \subsubsection*{Step 1: optimization} Let $\Omega$ and $B_{\rr}(x_{0})$ be as in the statement. We introduce the set 
 $$
 \mathcal{A}_{\sigma_{2},\sigma_{1}}:=\left\{\eta\in C^{1}_{c}(B_{\sigma_{1}}(x_{0}))\colon \eta\ge 0, \ \eta \equiv 1 \  \mbox{on}  \ B_{\sigma_{2}}(x_{0}) \right\}
 $$
 and, for $w\in L^{1}(B_{\sigma_{1}}(x_{0})\cap \Omega)$ being any function, define
 $$
 \mathcal{J}_{\sigma_{2},\sigma_{1}}(w;\Omega):=\inf_{\eta\in \mathcal{A}_{\sigma_{2},\sigma_{1}}}\int_{B_{\sigma_{1}}(x_{0})\cap \Omega}\snr{w}\snr{D\eta}^{2}\dx.
 $$
 Keeping in mind the discussion in Section \ref{cd}, a straightforward variation of \cite[Lemma 3]{BS} - just replace $w$ with $w\mathds{1}_{\Omega}$ there - shows that
 \begin{flalign}\label{2.3}
 \mathcal{J}_{\sigma_{2},\sigma_{1}}(w;\Omega)\le \frac{1}{(\sigma_{1}-\sigma_{2})^{1+\delta^{-1}}}\left(\int_{\sigma_{2}}^{\sigma_{1}}\left(\int_{\partial B_{s}(x_{0})\cap \Omega}\snr{w(x)} \d\mathcal{H}^{n-1}(x)\right)^{\delta}\ds\right)^{\delta^{-1}},
 \end{flalign}
 for all $\delta\in (0,1]$. Next, we define
 $$
 A_{\kk,t}:=B_{t}(x_{0})\cap \left\{x\in \Omega\cap B_{\rr}(x_{0})\colon w(x)>\kk\right\}\quad \mbox{for} \ \ B_{t}(x_{0})\subseteq B_{\rr}(x_{0}),
 $$
 and
 $$
 [0,\snr{\Omega\cap B_{\rr}(x_{0})}]\ni s\mapsto\omega_{f}(s):=\left(\int_{0}^{s}(\ti{f}^{*}(s))^{2}\ds\right)^{\frac{1}{2}},\qquad\quad  \ti{f}:=f\mathds{1}_{\Omega\cap B_{\rr}(x_{0})}.
 $$
 We then consider the first addendum on the right-hand side of inequality \eqref{2.0} with parameters $\tau_{1}$, $\tau_{2}$, $\rr$ as in the statement, $\sigma_{1},\sigma_{2}$ such that $\rr/2\le\tau_{1}\le\sigma_{2}<\sigma_{1}\le \tau_{2}\le  \rr$ and $0\le \kk_{1}<\kk_{2}$ and estimate:
 \begin{eqnarray}\label{2.4}
 &&\inf_{\eta\in \mathcal{A}_{\sigma_{2},\sigma_{1}}}\int_{\Omega\cap B_{\rr}(x_{0})}(w-\kk_{2})_{+}^{2}\snr{D\eta}^{2}\dx\nonumber \\
 &&\qquad \qquad \stackrel{\eqref{2.3}}{\le}\frac{1}{(\sigma_{1}-\sigma_{2})^{1+\delta^{-1}}}\left(\int_{\sigma_{2}}^{\sigma_{1}}\left(\int_{\partial B_{s}(x_{0})\cap \Omega}(w-\kk_{2})_{+}^{2} \d\mathcal{H}^{n-1}(x)\right)^{\delta}\ds\right)^{\delta^{-1}}\nonumber \\
 &&\qquad \qquad \ \le \frac{(\kk_{2}-\kk_{1})^{2-2^{*}_{\mf{s}}}}{(\sigma_{1}-\sigma_{2})^{1+\delta^{-1}}}\left(\int_{\sigma_{2}}^{\sigma_{1}}\left(\int_{\partial B_{s}(x_{0})\cap \Omega}\mathds{1}_{\{w>\kk_{2}\}}(w-\kk_{1})_{+}^{2^{*}_{\mf{s}}} \d\mathcal{H}^{n-1}(x)\right)^{\delta}\ds\right)^{\delta^{-1}},
 \end{eqnarray}
 where we used that
 \eqn{2.16}
 $$
 \kk_{2}>\kk_{1} \ \Longrightarrow \ \frac{\mathds{1}_{\{w>\kk_{2}\}}(w-\kk_{1})_{+}}{\kk_{2}-\kk_{1}}\ge 1.
 $$
 We then observe that $\Omega\cap B_{\sigma_{1}}(x_{0})$ is a Lipschitz domain, so given that $(w-\kk_{1})_{+}\in W^{1,2}(\Omega\cap B_{\sigma_{1}}(x_{0}))\cap C(\bar{\Omega}\cap B_{\sigma_{1}}(x_{0}))$ by \cite[Theorem 4.7]{eg} we can construct the Sobolev extension $\mf{w}\in C(\bar{\Omega}\cap B_{\sigma_{1}}(x_{0}))\cap W^{1,2}(\mathbb{R}^{n})$ of $(w-\kk_{1})_{+}$ such that
 \begin{flalign}\label{2.5}
 \begin{cases}
 \ \mf{w}(x)=(w(x)-\kk_{1})_{+} \ \ \mbox{for all} \ \ x\in  \Omega\cap B_{\sigma_{1}}(x_{0})\\
 \ \nr{\mf{w}}_{W^{1,2}(\mathbb{R}^{n})}\le c\rr^{-1}\nr{(w-\kk_{1})_{+}}_{W^{1,2}(\Omega\cap B_{\sigma_{1}}(x_{0}))},
 \end{cases}
 \end{flalign}
 for $c\equiv c(n,\Omega)$. Notice that here we used that $\sigma_{1}\ge \rr/2$. Keeping in mind that since $(w-\kk_{1})_{+}$ and, by construction $\mf{w}$ are continuous on $\bar{\Omega}\cap B_{\sigma_{1}}(x_{0})$ they coincide everywhere on $\Omega\cap B_{\sigma_{1}}(x_{0})$, we can complete inequality \eqref{2.4} as:
 \begin{flalign}\label{2.6}
 \inf_{\eta\in \mathcal{A}_{\sigma_{2},\sigma_{1}}}\int_{\Omega\cap B_{\rr}(x_{0})}&(w-\kk_{2})_{+}^{2}\snr{D\eta}^{2}\dx\nonumber\\
 \le& \frac{(\kk_{2}-\kk_{1})^{2-2^{*}_{\mf{s}}}}{(\sigma_{1}-\sigma_{2})^{1+\delta^{-1}}}\left(\int_{\sigma_{2}}^{\sigma_{1}}\left(\int_{\partial B_{s}(x_{0})}\snr{\mf{w}}^{2^{*}_{\mf{s}}} \d\mathcal{H}^{n-1}(x)\right)^{\delta}\ds\right)^{\delta^{-1}},
 \end{flalign}
 with $c\equiv c(n,\Omega)$. By the Sobolev embedding theorem on spheres we further obtain
 \begin{flalign}\label{2.7}
 \nr{\mf{w}}_{L^{2^{*}_{\mf{s}}}(\partial B_{s}(x_{0}))}\le c\rr^{-1}\nr{\mf{w}}_{W^{1,2}(\partial B_{s}(x_{0}))},
 \end{flalign}
 where we used again that $s\ge \sigma_{2}\ge \rr/2$ and it is $c\equiv c(n,\Omega)$, therefore \eqref{2.6} with $\delta=2/2^{*}_{\mf{s}}$ becomes
 \begin{eqnarray}\label{2.8}
 \inf_{\eta\in \mathcal{A}_{\sigma_{2},\sigma_{1}}}\int_{\Omega\cap B_{\rr}(x_{0})}(w-\kk_{2})_{+}^{2}\snr{D\eta}^{2}\dx&\stackrel{\eqref{2.7}}{\le}& \frac{c(\kk_{2}-\kk_{1})^{2-2^{*}_{\mf{s}}}}{\rr^{2^{*}_{\mf{s}}}(\sigma_{1}-\sigma_{2})^{1+2^{*}_{\mf{s}}/2}}\left(\int_{\sigma_{2}}^{\sigma_{1}}\nr{\mf{w}}_{W^{1,2}(\partial B_{s}(x_{0}))}^{2} \ds\right)^{\frac{2^{*}_{\mf{s}}}{2}}\nonumber \\
 &\le &\frac{c(\kk_{2}-\kk_{1})^{2-2^{*}_{\mf{s}}}\nr{\mf{w}}_{W^{1,2}(B_{\sigma_{1}}(x_{0})\setminus B_{\sigma_{2}}(x_{0}))}^{2^{*}_{\mf{s}}}}{\rr^{2^{*}_{\mf{s}}}(\sigma_{1}-\sigma_{2})^{1+2^{*}_{\mf{s}}/2}}\nonumber \\
&\stackrel{\eqref{2.5}}{\le} &\frac{c(\kk_{2}-\kk_{1})^{2-2^{*}_{\mf{s}}}\nr{(w-\kk_{1})_{+}}_{W^{1,2}(\Omega\cap B_{\sigma_{1}}(x_{0}))}^{2^{*}_{\mf{s}}}}{\rr^{22^{*}_{\mf{s}}}(\sigma_{1}-\sigma_{2})^{1+2^{*}_{\mf{s}}/2}},
 \end{eqnarray}
for $c\equiv c(n,\Omega)$. Since $\supp(f)\Subset \Omega$ and $\eta\in C^{1}_{c}(B_{\sigma_{1}}(x_{0}))$ we see that the second term in \eqref{2.0} can be controlled as:
\begin{eqnarray}\label{2.9}
\int_{\Omega\cap B_{\rr}(x_{0})}\mathds{1}_{\{w>\kk_{2}\}}\eta^{2}\snr{f}^{2}\dx&\le& \sup\left\{\int_{\Omega'}\snr{f}^{2}\dx\colon \snr{\Omega'}\le \snr{A_{\kk_{2},\sigma_{1}}}, \ \ \Omega'\subset B_{\rr}(x_{0})\cap \Omega\right\}\nonumber \\
&\stackrel{\eqref{2.0.1}}{\le}&\int_{0}^{\snr{A_{\kk_{2},\sigma_{1}}}}(\ti{f}^{*}(s))^{2}\ds=\omega_{f}(\snr{A_{\kk_{2},\sigma_{1}}})^{2}.
\end{eqnarray}
Merging \eqref{2.0}, \eqref{2.8} and \eqref{2.9} and recalling the definition of $\mathcal{A}_{\sigma_{2},\sigma_{1}}$ we obtain
\begin{flalign}\label{2.10}
\nr{D(w-\kk_{2})_{+}}_{L^{2}(\Omega\cap B_{\sigma_{2}}(x_{0}))}\le \frac{cM_{1}\nr{(w-\kk_{1})_{+}}_{W^{1,2}(\Omega\cap B_{\sigma_{1}}(x_{0}))}^{\frac{2^{*}_{\mf{s}}}{2}}}{\rr^{\alpha_{1}}(\sigma_{1}-\sigma_{2})^{\alpha_{2}}(\kk_{2}-\kk_{1})^{\frac{2^{*}_{\mf{s}}}{2}-1}}+c_{2}^{1/2}M_{2}\omega_{f}(\snr{A_{\kk_{2},\sigma_{1}}}),
\end{flalign}
with $c\equiv c(n,\Omega,c_{1})$, $\alpha_{1}:=2^{*}_{\mf{s}}$ and $\alpha_{2}:=(1+2^{*}_{\mf{s}}/2)/2$. We then observe that assumption $\Gamma_{\Omega}(x_{0};\rr/2)<\Upsilon_{\Omega}$ implies that $\Gamma_{\Omega}(x_{0};\sigma)\le 2^{n}\Upsilon_{\Omega}$ for all $\rr/2\le \sigma\le \rr$, so we gain:
\begin{eqnarray}\label{2.15}
\snr{A_{\kk_{2},\sigma_{1}}}&\stackrel{\eqref{2.16}}{\le}& \frac{\nr{(w-\kk_{1})_{+}}^{2^{*}}_{L^{2^{*}}(\Omega\cap B_{\sigma_{1}}(x_{0}))}}{(\kk_{2}-\kk_{1})^{2^{*}}}\stackrel{\eqref{pk}}{\le} \frac{c_{3}\nr{(w-\kk_{1})_{+}}_{W^{1,2}(\Omega\cap B_{\sigma_{1}}(x_{0}))}^{2^{*}}}{\rr^{2^{*}}(\kk_{2}-\kk_{1})^{2^{*}}},
\end{eqnarray}
with $c_{3}\equiv c_{3}(n,\Omega)$, and 
\begin{eqnarray}\label{2.17}
\nr{(w-\kk_{2})_{+}}_{L^{2}(\Omega\cap B_{\sigma_{2}}(x_{0}))}&\le& \snr{A_{\kk_{2},\sigma_{1}}}^{\frac{1}{n}}\nr{(w-\kk_{1})_{+}}_{L^{2^{*}}(\Omega\cap B_{\sigma_{1}}(x_{0}))}\nonumber \\
&\stackrel{\eqref{2.15}}{\le}&\frac{c''\nr{(w-\kk_{1})_{+}}_{W^{1,2}(\Omega\cap B_{\sigma_{1}}(x_{0}))}^{1+\frac{2^{*}}{n}}}{\rr^{\alpha_{3}}(\kk_{2}-\kk_{1})^{\frac{2^{*}}{n}}},
\end{eqnarray}
where $\alpha_{3}:=1+2^{*}/n$ and $c''\equiv c''(n,\Omega)$. We then plug \eqref{2.15}-\eqref{2.17} in \eqref{2.10} to conclude with
\begin{flalign}\label{2.11}
\nr{(w-\kk_{2})_{+}}_{W^{1,2}(\Omega\cap B_{\sigma_{2}}(x_{0}))}\le& \frac{c'M_{1}\nr{(w-\kk_{1})_{+}}_{W^{1,2}(\Omega\cap B_{\sigma_{1}}(x_{0}))}^{\frac{2^{*}_{\mf{s}}}{2}}}{\rr^{\alpha_{1}}(\sigma_{1}-\sigma_{2})^{\alpha_{2}}(\kk_{2}-\kk_{1})^{\frac{2^{*}_{\mf{s}}}{2}-1}}+\frac{c''\nr{(w-\kk_{1})_{+}}_{W^{1,2}(\Omega\cap B_{\sigma_{1}}(x_{0}))}^{1+\frac{2^{*}}{n}}}{\rr^{\alpha_{3}}(\kk_{2}-\kk_{1})^{\frac{2^{*}}{n}}}\nonumber \\
&+c_{2}^{1/2}M_{2}\omega_{f}\left(\frac{c_{3}\nr{(w-\kk_{1})_{+}}_{W^{1,2}(\Omega\cap B_{\sigma_{1}}(x_{0}))}^{2^{*}}}{\rr^{2^{*}}(\kk_{2}-\kk_{1})^{2^{*}}}\right),
\end{flalign}
where $c'\equiv c'(n,\Omega,c_{1})$, $c'',c_{3}\equiv c'',c_{3}(n,\Omega)$ and $c_{1}$, $c_{2}$, $M_{1}$, $M_{2}$ are the absolute constant in \eqref{2.0}.

\subsubsection*{Step 2: iteration} With numbers $i\in \N_{0}$, $\kk_{0}\ge 0$ and a sequence $\{\Delta_{j}\}_{j\in \N}\subset [0,\infty)$, we define
$$
\kk_{i}:=\kk_{0}+\sum_{j=1}^{i}\Delta_{j},\qquad \sigma_{i}:=\tau_{1}+2^{-i-1}(\tau_{2}-\tau_{1}),\qquad J_{i}:=\nr{(w-\kk_{i})_{+}}_{W^{1,2}(\Omega\cap B_{\sigma_{i}}(x_{0}))}.
$$
In these terms, \eqref{2.11} becomes
\begin{flalign}\label{2.12}
J_{i}\le\frac{c'2^{\alpha_{2}( i+1)}M_{1}J_{i-1}^{\frac{2^{*}_{\mf{s}}}{2}}}{\rr^{\alpha_{1}}(\tau_{2}-\tau_{1})^{\alpha_{2}}\Delta_{i}^{\frac{2^{*}_{\mf{s}}}{2}-1}}+c_{2}^{1/2}M_{2}\omega_{f}\left(\frac{c_{3}J_{i-1}^{2^{*}}}{\rr^{2^{*}}\Delta_{i}^{2^{*}}}\right)+c''\rr^{-\alpha_{3}}\left(\frac{J_{i-1}}{\Delta_{i}}\right)^{\frac{2^{*}}{n}}J_{i-1},
\end{flalign}
for all $i\in \N$. We then select $\chi\equiv \chi(n,\nu)\in (0,1/2)$ such that $(2\chi)^{-1+2^{*}_{\mf{s}}/2}=2^{-\alpha_{2}}$ and for all $i\in \N$, split $\Delta_{i}$ as $\Delta_{i}=\Delta_{i}^{1}+\Delta_{i}^{2}+\Delta_{i}^{3}$, where
$$
\Delta_{i}^{1}:=\left(\frac{c'2^{\alpha_{2}}3M_{1}\chi^{-(-1+2^{*}_{\mf{s}}/2)-1}}{\rr^{\alpha_{1}}(\tau_{2}-\tau_{1})^{\alpha_{2}}}\right)^{\frac{1}{\frac{2^{*}_{\mf{s}}}{2}-1}}J_{0}2^{-i},\qquad\quad \Delta_{i}^{3}:=\left(\frac{3c''}{\rr^{\alpha_{3}}\chi}\right)^{\frac{n}{2^{*}}}\chi^{i-1}J_{0},
$$
and $\Delta_{i}^{2}$ is the smallest value such that
$$
c_{2}^{1/2}M_{2}\omega_{f}\left(\frac{c_{3}J_{i-1}^{2^{*}}}{\rr^{2^{*}}(\Delta_{i}^{2})^{2^{*}}}\right)\le \frac{\chi^{i}}{3}J_{0} \ \Longrightarrow \ \Delta_{i}^{2}:=\left(\frac{c_{3}J_{i-1}^{2^{*}}}{\rr^{2^{*}}(\omega_{f})^{-1}\left(\chi^{i}J_{0}/(3c_{2}^{1/2}M_{2})\right)}\right)^{\frac{1}{2^{*}}},
$$
given that $s\mapsto \omega_{f}(s)$ is nondecreasing.
A straightforward computation shows that directly plugging in \eqref{2.12} the values $\Delta_{1}^{1}$-$\Delta_{1}^{3}$ we obtain $J_{1}\le \chi J_{0}$. By induction, let us assume that $J_{i-1}\le \chi^{i-1}J_{0}$ for some $i\in \N$, $i\ge 2$ and bound
\begin{eqnarray*}
J_{i}
&\le&\frac{c'2^{\alpha_{2}( i+1)}M_{1}\chi^{\frac{2^{*}_{\mf{s}}(i-1)}{2}}J_{0}^{\frac{2^{*}_{\mf{s}}}{2}}}{\rr^{\alpha_{1}}(\tau_{2}-\tau_{1})^{\alpha_{2}}(\Delta_{i}^{1})^{\frac{2^{*}_{\mf{s}}}{2}-1}}+c_{2}^{1/2}M_{2}\omega_{f}\left(\frac{c_{3}J_{i-1}^{2^{*}}}{\rr^{2^{*}}(\Delta_{i}^{2})^{2^{*}}}\right)+c''\rr^{-\alpha_{3}}\left(\frac{\chi^{i-1}J_{0}}{\Delta_{i}^{3}}\right)^{\frac{2^{*}}{n}}\chi^{i-1}J_{0}\nonumber \\
&\le&\frac{2^{\alpha_{2}i}(2\chi)^{(-1+2^{*}_{\mf{s}}/2)i}\chi^{i}J_{0}}{3}+\frac{2J_{0}\chi^{i}}{3}\le \chi^{i}J_{0},
\end{eqnarray*}
thanks to our choice of $\chi$. This allows concluding that
\eqn{2.13}
$$J_{i}\le \chi^{i}J_{0}\qquad \mbox{for all} \ \ i\in \N.$$
Moreover, by very definition it is
\begin{eqnarray*}
\sum_{i=1}^{\infty}\Delta_{i}&\le& \left(\frac{c'2^{\alpha_{2}}3M_{1}\chi^{-(-1+2^{*}_{\mf{s}}/2)-1}}{\rr^{\alpha_{1}}(\tau_{2}-\tau_{1})^{\alpha_{2}}}\right)^{\frac{1}{\frac{2^{*}_{\mf{s}}}{2}-1}}J_{0}\sum_{i=1}^{\infty}2^{-i}+\left(\frac{3c''}{\rr^{\alpha_{3}}\chi}\right)^{\frac{n}{2^{*}}}J_{0}\sum_{i=0}^{\infty}\chi^{i-1}\nonumber \\
&&+\sum_{i=0}^{\infty}\left(\frac{c_{3}J_{i-1}^{2^{*}}}{\rr^{2^{*}}(\omega_{f})^{-1}\left(\chi^{i}J_{0}/(3c_{2}^{1/2}M_{2})\right)}\right)^{\frac{1}{2^{*}}}\nonumber \\
&\stackrel{\eqref{2.13}}{\le}&\frac{cM_{1}^{\frac{2}{2^{*}_{\mf{s}}-2}}J_{0}}{\rr^{\alpha}(\tau_{2}-\tau_{1})^{\beta_{1}}}+\sum_{i=1}^{\infty}\frac{c_{3}^{\frac{1}{2^{*}}}\chi^{i-1}J_{0}}{\rr(\omega_{f})^{-1}\left(\chi^{i}J_{0}/(3c_{2}^{1/2}M_{2})\right)^{1/2^{*}}}=:\frac{cM_{1}^{\frac{2}{2^{*}_{\mf{s}}-2}}J_{0}}{\rr^{\alpha}(\tau_{2}-\tau_{1})^{\beta_{1}}}+\frac{c_{3}^{\frac{1}{2^{*}}}\mathcal{S}_{f}}{\rr},
\end{eqnarray*}
where we set $\alpha:=\max\left\{\frac{n\alpha_{3}}{2^{*}},\frac{2^{*}_{\mf{s}}\alpha_{1}}{2^{*}_{\mf{s}}-2}\right\}$, $\beta_{1}:=\frac{2\alpha_{2}}{2^{*}_{\mf{s}}-2}$ and it is $c\equiv c(n,\nu,\Omega,c_{1},c_{2})$. Term $\mathcal{S}_{f}$ can be estimated in terms of the $L(n,1)$-norm of $f$:
\begin{flalign*}
\mathcal{S}_{f}=&\frac{1}{\chi}\sum_{i=1}^{\infty}\int_{i}^{i+1}\frac{\chi^{i}J_{0}}{(\omega_{f})^{-1}\left(\chi^{i}J_{0}/(3c_{2}^{1/2}M_{2})\right)^{1/2^{*}}}\ds\le \frac{1}{\chi^{2}}\int_{1}^{\infty}\frac{\chi^{s}J_{0}}{(\omega_{f})^{-1}\left(\chi^{s}J_{0}/(3c_{2}^{1/2}M_{2})\right)^{1/2^{*}}}\ds\nonumber \\
\le&\frac{1}{\chi^{2}\snr{\log(\chi)}}\int_{0}^{\chi}\frac{tJ_{0}}{(\omega_{f})^{-1}\left(tJ_{0}/(3c_{2}^{1/2}M_{2})\right)^{1/2^{*}}}\frac{\dt}{t}\le \frac{3c_{2}^{1/2}M_{2}}{\chi^{2}\snr{\log(\chi)}}\int_{0}^{(\omega_{f})^{-1}\left(\frac{\chi J_{0}}{3c_{2}^{1/2}M_{2}}\right)}\frac{\omega_{f}'(\lambda)}{\lambda^{1/2^{*}}}\d\lambda\nonumber \\
\le&\frac{3c_{2}^{1/2}M_{2}}{2\chi^{2}\snr{\log(\chi)}}\int_{0}^{(\omega_{f})^{-1}\left(\frac{\chi J_{0}}{3c_{2}^{1/2}M_{2}}\right)}\frac{\ti{f}^{*}(\lambda)}{\lambda^{\frac{1}{2}+\frac{1}{2^{*}}}}\d\lambda\le\frac{3c_{2}^{1/2}M_{2}}{2\chi^{2}\snr{\log(\chi)}}\int_{0}^{\infty}\lambda^{1/n}\ti{f}^{*}(\lambda)\frac{\d\lambda}{\lambda}\le cM_{2}[f]_{n,1;\Omega\cap B_{\rr}(x_{0})},
\end{flalign*}
where $c\equiv c(n,\nu,c_{2})$ and we used that being $\lambda\mapsto \ti{f}^{*}(\lambda)$ nonincreasing and since $\omega_{f}'(\lambda)=(2\omega_{f}(\lambda))^{-1}(\ti{f}^{*}(\lambda))^{2}$ it is $\omega_{f}'(\lambda)\ge \lambda^{1/2}\ti{f}^{*}(\lambda)$. Combining the content of the three previous displays we obtain that $\lim_{i\to \infty}J_{i}=0$ and $\lim_{i\to \infty}\kk_{i}<\infty$, therefore we can conclude with
\begin{eqnarray}\label{2.14}
\nr{w}_{L^{\infty}(\Omega\cap B_{\tau_{1}}(x_{0}))}&\le&\kk_{0}+\sum_{j=1}^{\infty}\Delta_{j}\nonumber \\
&\le& \kk_{0}+\frac{cM_{1}^{\frac{2}{2^{*}_{\mf{s}}-2}}\nr{(w-\kk_{0})_{+}}_{W^{1,2}(\Omega\cap B_{(\tau_{1}+\tau_{2})/2}(x_{0}))}}{\rr^{\alpha}(\tau_{2}-\tau_{1})^{\beta_{1}}}+\frac{cM_{2}[f]_{n,1;\Omega\cap B_{\rr}(x_{0})}}{\rr},
\end{eqnarray}
 for $c\equiv c(n,\nu,\Omega,c_{1},c_{2})$. Another application of \eqref{2.0} on the right-hand side of \eqref{2.14} with $\eta\in C^{1}_{c}(B_{\rr}(x_{0}))$ such that $\mathds{1}_{B_{(\tau_{1}+\tau_{2})/2}(x_{0})}\le \eta\le \mathds{1}_{B_{\tau_{2}}(x_{0})}$ and $\snr{D\eta}\lesssim (\tau_{2}-\tau_{1})^{-1}$ yields \eqref{2.1} - with $\beta:=\beta_{1}+1$ - and the proof is complete.
 \end{proof}
 \section{A priori estimates}\label{apr}
 In this section we deal with suitable regularized versions of \eqref{fun} to derive $L^{\infty}$ and $W^{2,2}$-estimates for smooth solutions to the associated Euler-Lagrange system. More precisely, let 
\begin{flalign}\label{om.1}
\Omega\subset \mathbb{R}^{n}, \ \ n\ge 3, \ \ \mbox{is an open, bounded, convex domain with $C^{2}$-regular boundary}
\end{flalign}
and $F\colon \mathbb{R}^{N\times n}\to \mathbb{R}$ be an integrand satisfying
\begin{flalign}\label{3.3}
F(z)=\ti{F}(\ell_{\ti{\mu}}(z))\qquad \mbox{for some} \ \ \ti{F}\in C^{\infty}_{\loc}[0,\infty), \ \ \ti{\mu}\in (0,1].
\end{flalign}
Moreover, the following growth/ellipticity conditions are in force:
\begin{flalign}\label{3.0}
\begin{cases}
\ L^{-1}\ell_{\mu}(z)^{p}-L\mu^{p}\le F(z)\le L\ell_{\mu}(z)^{p}+L \ell_{\mu}(z)^{q}\\
\ L^{-1}\ell_{\mu}(z)^{p-2}\snr{\xi}^{2}\le \langle \partial^{2}F(z)\xi,\xi\rangle\\
\ \snr{\partial^{2}F(z)}\le L\ell_{\mu}(z)^{p-2}+L \ell_{\mu}(z)^{q-2}
\end{cases}
\end{flalign}
for all $z\in \mathbb{R}^{N\times n}$, some absolute constant $L\ge 1$ and $\mu\equiv \mu(\ti{\mu})\in (0,4]$. As direct consequences of the assumptions listed above, we find that $z\mapsto F(z)$ is strictly convex. Furthermore, setting $\ti{a}(t):=\ti{F}'(t)/t$, $a(z):=\ti{a}(\ell_{\ti{\mu}}(z))z$, as in \cite[Sections 4.1 and 12]{demi1}, we obtain
\begin{flalign}\label{3.2}
\partial a(z)=\ti{a}(\ell_{\ti{\mu}}(z))\mathds{I}_{N\times n}+\ti{a}'(\ell_{\ti{\mu}}(z))\ell_{\ti{\mu}}(z)\frac{z\otimes z}{\ell_{\ti{\mu}}(z)^{2}},
\end{flalign}
where we denoted $\mathds{I}_{N\times n}:=(\delta_{ij}\delta^{\alpha\beta})(e^{\alpha}\otimes e_{i})\otimes (e^{\beta}\otimes e_{j})$, $i,j\in \{1,\cdots,n\}$, $\alpha,\beta\in \{1,\cdots,N\}$ - here $\delta$ is the usual Kronecker's symbol - and
\begin{flalign}\label{f.2}
 \left\{
\begin{array}{c}
\displaystyle 
 \ L^{-1}\ell_{\mu}(z)^{p-2}\le\ti{a}(\ell_{\ti{\mu}}(z))+\frac{\ti{a}'(\ell_{\ti{\mu}}(z))\snr{z}^{2}}{\ell_{\ti{\mu}}(z)}\le L\ell_{\mu}(z)^{p-2}+L \ell_{\mu}(z)^{q-2}\\[12pt]\displaystyle
 \ L^{-1}\ell_{\mu}(z)^{p-2}\le \ti{a}(\ell_{\ti{\mu}}(z))\le L\ell_{\mu}(z)^{p-2}+L \ell_{\mu}(z)^{q-2}\\[12pt]\displaystyle
 \ \frac{\snr{\ti{a}'(\ell_{\ti{\mu}}(z))}\snr{z}^{2}}{\ell_{\ti{\mu}}(z)}\le L\ell_{\mu}(z)^{p-2}+L \ell_{\mu}(z)^{q-2}
 \end{array}
\right.
\end{flalign}
for all $z\in \mathbb{R}^{N\times n}$. We finally take a function $f\in C^{\infty}_{c}(\Omega,\mathbb{R}^{N})$, a ball $B_{\rr}(x_{0})\subset \mathbb{R}^{n}$ with $\rr\in (0,\rr_{*}]$ and satisfying $\Gamma_{\Omega}(x_{0};\rr/2)\le \Upsilon_{\Omega}$, for some positive, absolute constant $\Upsilon_{\Omega}\equiv \Upsilon_{\Omega}(n,\Omega)$ and consider a classical solution $v\in C^{3}(\bar{\Omega}\cap B_{\rr}(x_{0}),\mathbb{R}^{N})$ of system
\begin{flalign}\label{3.1}
\begin{cases}
\ -\diver \ a(Dv)=f\qquad &\mbox{in} \ \ \Omega\cap B_{\rr}(x_{0})\\
\ v=0\qquad &\mbox{on} \ \ \partial \Omega\cap B_{\rr}(x_{0}).
\end{cases}
\end{flalign}
\subsection{A Caccioppoli type inequality on level sets}\label{dgk}
Recalling that $\mu>0$, from $\eqref{f.2}_{2}$ we deduce that $\ti{a}(\ell_{\ti{\mu}}(z))>0$ for all $z\in \mathbb{R}^{N\times n}$, therefore the matrix
\begin{flalign*}
\gamma\equiv \left\{\gamma_{ij}\right\}_{i,j=1}^{n}:=\left\{\delta_{ij}+\frac{\ti{a}'(\ell_{\ti{\mu}}(z))\ell_{\ti{\mu}}(z)}{\ti{a}(\ell_{\ti{\mu}}(z))}\sum_{\alpha=1}^{N}\frac{z^{\alpha}_{i}z^{\alpha}_{j}}{\ell_{\ti{\mu}}(z)^{2}}\right\}_{i,j=1}^{n}
\end{flalign*}
is well definite for all $z\in \mathbb{R}^{N\times n}$, symmetric and positive:
\begin{eqnarray}\label{3.4}
\gamma\langle\lambda_{i},\lambda_{j}\rangle&=&\snr{\lambda}^{2}+\frac{\ti{a}'(\ell_{\ti{\mu}}(z))\ell_{\ti{\mu}}(z)}{\ti{a}(\ell_{\ti{\mu}}(z))}\sum_{i,j=1}^{n}\sum_{\alpha=1}^{N}\frac{z^{\alpha}_{i}\lambda_{i}z^{\alpha}_{j}\lambda_{j}}{\ell_{\ti{\mu}}(z)^{2}}\nonumber \\
&=&\snr{\lambda}^{2}+\frac{\ti{a}'(\ell_{\ti{\mu}}(z))\ell_{\ti{\mu}}(z)}{\ti{a}(\ell_{\ti{\mu}}(z))}\sum_{\alpha=1}^{N}\frac{\snr{\langle z^{\alpha},\lambda\rangle}^{2}}{\ell_{\ti{\mu}}(z)^{2}}\nonumber \\
&=&\frac{1}{\ti{a}(\ell_{\ti{\mu}}(z))}\left(\ti{a}(\ell_{\ti{\mu}}(z))\snr{\lambda}^{2}+\frac{\ti{a}'(\ell_{\ti{\mu}}(z))\snr{z}^{2}}{\ell_{\ti{\mu}}(z)}\sum_{\alpha=1}^{N}\frac{\snr{\langle z^{\alpha},\lambda\rangle}^{2}}{\snr{z}^{2}}\right)\nonumber \\
&\ge&\frac{\snr{\lambda}^{2}}{\ti{a}(\ell_{\ti{\mu}}(z))}\min\left\{\ti{a}(\ell_{\ti{\mu}}(z)),\ti{a}(\ell_{\ti{\mu}}(z))+\frac{\ti{a}'(\ell_{\ti{\mu}}(z))\snr{z}^{2}}{\ell_{\ti{\mu}}(z)}\right\}\stackrel{\eqref{f.2}_{1}}{\ge} \frac{\snr{\lambda}^{2}\ell_{\mu}(z)^{p-2}}{L\ti{a}(\ell_{\ti{\mu}}(z))}\ge 0,
\end{eqnarray}
for all $\lambda\in \mathbb{R}^{n}$, where the last term in \eqref{3.4} vanishes iff $\lambda=0$
and we can set
\begin{flalign}\label{3.5}
\gamma^{v}\equiv \left\{\gamma_{ij}^{v}\right\}_{i,j=1}^{n}:=\left\{\delta_{ij}+\frac{\ti{a}'(\ell_{\ti{\mu}}(Dv))\ell_{\ti{\mu}}(Dv)}{\ti{a}(\ell_{\ti{\mu}}(Dv))}\sum_{\alpha=1}^{N}\frac{Dv^{\alpha}_{i}Dv^{\alpha}_{j}}{\ell_{\ti{\mu}}(Dv)^{2}}\right\}_{i,j=1}^{n}.
\end{flalign}
We then fix a vector $e\in \mathbb{S}^{n-1}$, differentiate $\eqref{3.1}_{1}$ in the direction of $e$ to get
\begin{flalign}\label{3.6}
-D_{e}f^{\alpha}=\sum_{i,k=1}^{n}\sum_{\beta=1}^{N}\partial_{x_{i}}\left[\left(\ti{a}(\ell_{\ti{\mu}}(Dv))\delta_{ik}\delta^{\alpha\beta}+\ti{a}'(\ell_{\ti{\mu}}(Dv))\ell_{\ti{\mu}}(Dv)\frac{D_{i}v^{\alpha}D_{k}v^{\beta}}{\ell_{\ti{\mu}}(Dv)^{2}}\right)D_{e}D_{k}v^{\beta}\right]
\end{flalign}
and consider the second order elliptic operator defined as $$C^{2}(\bar{\Omega}\cap B_{\rr}(x_{0}))\ni w \mapsto \mathcal{L}(w):=\sum_{i,j=1}^{n}\partial_{x_{i}}\left(\gamma_{ij}^{v}D_{j}w\right),$$
where $\{\gamma_{ij}^{v}\}_{i,j=1}^{n}$ has been defined in \eqref{3.5}. Keeping in mind the definition of $\ti{a}(\cdot)$, we compute
\begin{eqnarray*}
\mathcal{L}\left(\ti{F}(\ell_{\ti{\mu}}(Dv))\right)&=&\Delta\left(\ti{F}(\ell_{\ti{\mu}}(Dv))\right)\nonumber \\
&&+\sum_{i,j,k=1}^{n}\sum_{\alpha,\beta=1}^{N}\partial_{x_{i}}\left(\ti{a}'(\ell_{\ti{\mu}}(Dv))\ell_{\ti{\mu}}(Dv)\frac{D_{k}v^{\beta}D_{j}D_{k}v^{\beta}D_{i}v^{\alpha}D_{j}v^{\alpha}}{\ell_{\ti{\mu}}(Dv)^{2}}\right)\nonumber \\
&=:&\Delta\left(\ti{F}(\ell_{\ti{\mu}}(Dv))\right)+\mathcal{I}.
\end{eqnarray*}
Expanding the expression of term $\mathcal{I}$ we obtain
\begin{eqnarray*}
\mathcal{I}&=&\sum_{i,j,k=1}^{n}\sum_{\alpha,\beta=1}^{N}\ti{a}'(\ell_{\ti{\mu}}(Dv))\ell_{\ti{\mu}}(Dv)\frac{D_{k}v^{\beta}D_{j}D_{k}v^{\beta}D_{i}v^{\alpha}D_{i}D_{j}v^{\alpha}}{\ell_{\ti{\mu}}(Dv)^{2}}\nonumber \\
&&+\sum_{j=1}^{n}\sum_{\alpha=1}^{N}D_{j}v^{\alpha}\sum_{i,k=1}^{n}\sum_{\beta=1}^{N}\partial_{x_{i}}\left(\ti{a}'(\ell_{\ti{\mu}}(Dv))\ell_{\ti{\mu}}(Dv)\frac{D_{k}v^{\beta}D_{i}v^{\alpha}D_{j}D_{k}v^{\beta}}{\ell_{\ti{\mu}}(Dv)^{2}}\right)\nonumber \\
&\stackrel{\eqref{3.6}}{=}&\sum_{i,j,k=1}^{n}\sum_{\alpha,\beta=1}^{N}\ti{a}'(\ell_{\ti{\mu}}(Dv))\ell_{\ti{\mu}}(Dv)\frac{D_{k}v^{\beta}D_{j}D_{k}v^{\beta}D_{i}v^{\alpha}D_{i}D_{j}v^{\alpha}}{\ell_{\ti{\mu}}(Dv)^{2}}\nonumber \\
&&-\sum_{j=1}^{n}\sum_{\alpha=1}^{N}D_{j}v^{\alpha}D_{j}f^{\alpha}-\sum_{i,j=1}^{n}\sum_{\alpha=1}^{N}D_{j}v^{\alpha}\partial_{x_{i}}\left(\ti{a}(\ell_{\ti{\mu}}(Dv))D_{i}D_{j}v^{\alpha}\right)\nonumber \\
&=&\sum_{i,j,k=1}^{n}\sum_{\alpha,\beta=1}^{N}\ti{a}'(\ell_{\ti{\mu}}(Dv))\ell_{\ti{\mu}}(Dv)\frac{D_{k}v^{\beta}D_{j}D_{k}v^{\beta}D_{i}v^{\alpha}D_{i}D_{j}v^{\alpha}}{\ell_{\ti{\mu}}(Dv)^{2}}\nonumber \\
&&-\sum_{j=1}^{n}\sum_{\alpha=1}^{N}D_{j}v^{\alpha}D_{j}f^{\alpha}-\sum_{i,j=1}^{n}\sum_{\alpha=1}^{N}\partial_{x_{i}}\left(\ti{a}(\ell_{\ti{\mu}}(Dv))D_{i}D_{j}v^{\alpha}D_{j}v^{\alpha}\right)\nonumber \\
&&+\sum_{i,j=1}^{n}\sum_{\alpha=1}^{N}\ti{a}(\ell_{\ti{\mu}}(Dv))D_{i}D_{j}v^{\alpha}D_{i}D_{j}v^{\alpha}\nonumber \\
&\stackrel{\eqref{3.2}}{=}&\langle\partial a(Dv)D^{2}v,D^{2}v\rangle-\langle Dv,Df\rangle-\Delta\left(\ti{F}(\ell_{\ti{\mu}}(Dv))\right).
\end{eqnarray*}
Merging the content of the two above displays we obtain
\begin{flalign}\label{3.7}
 \mathcal{L}\left(\ti{F}(\ell_{\ti{\mu}}(Dv))\right)=\langle\partial a(Dv)D^{2}v,D^{2}v\rangle-\langle Dv,Df\rangle.
\end{flalign}
Now, let $H\in \textnormal{Lip}_{\loc}[0,\infty)$ be a nonnegative, nondecreasing function such that 
\eqn{3.7.4}
$$H(t)>0\Longrightarrow \ t>0 \ \Longrightarrow \ H'(t)>0,$$ $0\le \eta(\cdot)\in C^{1}_{c}(B_{\rr}(x_{0}))$ be any map and observe that by coarea formula it is $\snr{D\snr{Dv}}=0$ whenever $\snr{Dv}=0$. Moreover, the classical chain rule for Lipschitz functions applies to the composition $H(\snr{Dv})$ on $\Omega\cap B_{\rr}(x_{0})$, so we multiply the identity in \eqref{3.7} by $\eta^{2}H(\snr{Dv})$ to get on the left-hand side
\begin{eqnarray*}
\eta^{2}H(\snr{Dv})\mathcal{L}\left(\ti{F}(\ell_{\ti{\mu}}(Dv))\right)&=&H(\snr{Dv})\sum_{i,j=1}^{n}\partial_{x_{i}}\left(\eta^{2}\gamma_{ij}^{v}\partial_{x_{j}}\left(\ti{F}(\ell_{\ti{\mu}}(Dv))\right)\right)\nonumber \\
&&-2\eta H(\snr{Dv})\sum_{i,j=1}^{n}\gamma_{ij}^{v}D_{i}\eta \partial_{x_{j}}\left(\ti{F}(\ell_{\ti{\mu}}(Dv))\right)\nonumber \\
&=:&\mathcal{J}-2\eta H(\snr{Dv})\ti{a}(\ell_{\ti{\mu}}(Dv))\snr{Dv}\sum_{i,j=1}^{n} \gamma^{v}_{ij}D_{i}\eta D_{j}\snr{Dv}.
\end{eqnarray*}
Let us have a look to term $\mathcal{J}$. We have:
\begin{eqnarray*}
\mathcal{J}&=&\sum_{i,j=1}^{n}\partial_{x_{i}}\left(\eta^{2}H(\snr{Dv})\gamma_{ij}^{v}\partial_{x_{j}}\left(\ti{F}(\ell_{\ti{\mu}}(Dv))\right)\right)\nonumber \\
&&-\sum_{i,j=1}^{n}\eta^{2}H'(\snr{Dv})\snr{Dv}\ti{a}(\ell_{\ti{\mu}}(Dv))\gamma_{ij}^{v}D_{i}\snr{Dv}D_{j}\snr{Dv}=:\mathcal{J}_{1}+\mathcal{J}_{2}.
\end{eqnarray*}
Plugging the content of the last two displays in \eqref{3.7} and rearranging terms, we obtain
\begin{eqnarray}\label{3.7.1}
\eta^{2}H(\snr{Dv})\langle \partial a(Dv)D^{2}v,D^{2}v\rangle-\mathcal{J}_{2}&=&\eta^{2}H(\snr{Dv})\langle Dv,Df\rangle+\mathcal{J}_{1}\nonumber \\
&&-2\eta H(\snr{Dv})\ti{a}(\ell_{\ti{\mu}}(Dv))\snr{Dv}\sum_{i,j=1}^{n} \gamma^{v}_{ij}D_{i}\eta D_{j}\snr{Dv}.
\end{eqnarray}
We then integrate the content of the above display over $\Omega\cap B_{\rr}(x_{0})$ to get
\begin{flalign}\label{3.9.1}
\int_{\Omega\cap B_{\rr}(x_{0})}&\eta^{2}H(\snr{Dv})\langle\partial a(Dv)D^{2}v,D^{2}v\rangle \dx-\int_{\Omega\cap B_{\rr}(x_{0})}\mathcal{J}_{2}\dx\nonumber \\
=&\int_{\Omega\cap B_{\rr}(x_{0})}\eta^{2}H(\snr{Dv})\langle Dv,Df\rangle \dx+\int_{\Omega\cap B_{\rr}(x_{0})}\mathcal{J}_{1} \dx\nonumber \\
&-2\sum_{i,j=1}^{n}\int_{\Omega\cap B_{\rr}(x_{0})}\eta H(\snr{Dv})\ti{a}(\ell_{\ti{\mu}}(Dv))\snr{Dv}\gamma_{ij}^{v}D_{i}\eta D_{j}\snr{Dv} \dx\nonumber \\
=:&\mbox{(I)}+\mbox{(II)}+\mbox{(III)}.
\end{flalign}
Recalling that $\supp(f)\Subset \Omega$ and $\supp(\eta)\Subset B_{\rr}(x_{0})$, we integrate by parts to get:
\begin{eqnarray*}
\mbox{(I)}&=&-\int_{\Omega\cap B_{\rr}(x_{0})} \diver(\eta^{2}H(\snr{Dv})Dv)f \dx\nonumber \\
&\le&\sigma\int_{\Omega\cap B_{\rr}(x_{0})}\eta^{2}\snr{DH(\snr{Dv})}^{2} \ \dx+\sigma\int_{\Omega\cap B_{\rr}(x_{0})}\eta^{2}H(\snr{Dv})\ell_{\mu}(Dv)^{p-2}\snr{D^{2}v}^{2} \ \dx\nonumber \\
&&+\frac{2}{\sigma}\int_{\Omega\cap B_{\rr}(x_{0})\cap\{H(\snr{Dv})>0\}}\eta^{2}\snr{f}^{2}\left(\snr{Dv}^{2}+\frac{H(\snr{Dv})}{\ell_{\mu}(Dv)^{p-2}}\right)\dx\nonumber \\
&&+4\int_{\Omega\cap B_{\rr}(x_{0})}\snr{D\eta}^{2}H(\snr{Dv})^{2} \dx,
\end{eqnarray*}
with $\sigma\in (0,1)$ to be fixed suitably small in a few lines. Concerning term $\mbox{(II)}$, we recall that by $\eqref{3.1}_{2}$ the tangential components of $Dv^{\alpha}$ vanish on $\Omega\cap B_{\rr}(x_{0})$, so we can apply \eqref{3.8} with $w=Dv^{\alpha}$ to deduce the validity of
\begin{flalign}\label{3.8.1}
 \Delta v^{\alpha}\partial_{\mf{v}_{\Omega}}v^{\alpha}-\sum_{i,k=1}^{n}D_{i}D_{k}v^{\alpha}D_{k}v^{\alpha}(\mf{v}_{\Omega})_{i}=-\texttt{tr}(\mathbf{S}_{\Omega})(\partial_{\mf{v}_{\Omega}}v^{\alpha})^{2},\qquad \qquad \texttt{tr}(\mathbf{S}_{\Omega})\le 0,
\end{flalign}
for all $\alpha\in \{1,\cdots,N\}$ and then estimate:
\begin{eqnarray*}
\mbox{(II)}&=&\sum_{i,j=1}^{n}\int_{\partial \Omega\cap B_{\rr}(x_{0})}\eta^{2}H(\snr{Dv})\gamma^{v}_{ij}(\mf{v}_{\Omega})_{i}\partial_{x_{j}}\left(\ti{F}(\ell_{\ti{\mu}}(Dv))\right) \d\mathcal{H}^{n-1}\nonumber \\
&=&\sum_{j,k=1}^{n}\sum_{\alpha,\beta=1}^{N}\int_{\partial \Omega\cap B_{\rr}(x_{0})}\eta^{2}H(\snr{Dv})\ti{a}'(\ell_{\ti{\mu}}(Dv))\ell_{\ti{\mu}}(Dv)\frac{D_{j}v^{\alpha}D_{k}v^{\beta}}{\ell_{\ti{\mu}}(Dv)^{2}}D_{j}D_{k}v^{\beta}\partial_{\mf{v}_{\Omega}}v^{\alpha}\d\mathcal{H}^{n-1}\nonumber \\
&&+\sum_{\alpha=1}^{N}\int_{\partial \Omega\cap B_{\rr}(x_{0})}\eta^{2}H(\snr{Dv})\ti{a}(\ell_{\ti{\mu}}(Dv))\left(\sum_{i,k=1}^{n}D_{k}v^{\alpha}D_{i}D_{k}v^{\alpha}(\mf{v}_{\Omega})_{i}\right)\d\mathcal{H}^{n-1}\nonumber \\
&\stackrel{\eqref{3.8.1}_{1}}{=}&\sum_{\alpha=1}^{N}\int_{\partial \Omega\cap B_{\rr}(x_{0})}\eta^{2}H(\snr{Dv})\partial_{\mf{v}_{\Omega}}v^{\alpha}\left(\ti{a}'(\ell_{\ti{\mu}}(Dv))\ell_{\ti{\mu}}(Dv)\sum_{\beta=1}^{N}\sum_{j,k=1}^{n}\frac{D_{j}v^{\alpha}D_{k}v^{\beta}}{\ell_{\ti{\mu}}(Dv)^{2}}D_{j}D_{k}v^{\beta}\right)\d\mathcal{H}^{n-1}\nonumber \\
&&+\sum_{\alpha=1}^{N}\int_{\partial\Omega\cap B_{\rr}(x_{0})}\eta^{2}H(\snr{Dv})\ti{a}(\ell_{\ti{\mu}}(Dv))\left(\Delta v^{\alpha}\partial_{\mf{v}_{\Omega}}v^{\alpha}+\texttt{tr}(\mathbf{S}_{\Omega})(\partial_{\mf{v}_{\Omega}}v^{\alpha})^{2}\right)\d\mathcal{H}^{n-1}\nonumber \\
&\stackrel{\eqref{3.1}_{1}}{=}&-\sum_{\alpha=1}^{N}\int_{\partial\Omega\cap B_{\rr}(x_{0})}\eta^{2}H(\snr{Dv})\partial_{\mf{v}_{\Omega}}v^{\alpha}\left(\ti{a}(\ell_{\ti{\mu}}(Dv))\sum_{\beta=1}^{N}\sum_{j,k=1}^{n}\delta_{jk}\delta^{\alpha\beta}D_{j}D_{k}v^{\beta}+f^{\alpha}\right)\d\mathcal{H}^{n-1}\nonumber \\
&&+\sum_{\alpha=1}^{N}\int_{\partial\Omega\cap B_{\rr}(x_{0})}\eta^{2}H(\snr{Dv})\ti{a}(\ell_{\ti{\mu}}(Dv))\left(\sum_{\beta=1}^{N}\sum_{j,k=1}^{n}\delta_{jk}\delta^{\alpha\beta} D_{j}D_{k}v^{\beta}\partial_{\mf{v}_{\Omega}}v^{\alpha}\right)\d\mathcal{H}^{n-1}\nonumber \\
&&+\sum_{\alpha=1}^{N}\int_{\partial\Omega\cap B_{\rr}(x_{0})}\eta^{2}H(\snr{Dv})\ti{a}(\ell_{\ti{\mu}}(Dv))\texttt{tr}(\mathbf{S}_{\Omega})(\partial_{\mf{v}_{\Omega}}v^{\alpha})^{2}\d\mathcal{H}^{n-1}\nonumber \\
&\stackrel{\eqref{3.8.1}_{2}}{\le}&-\sum_{\alpha=1}^{N}\int_{\partial \Omega\cap B_{\rr}(x_{0})}\eta^{2}H(\snr{Dv})\partial_{\mf{v}_{\Omega}}v^{\alpha}f^{\alpha}\d\mathcal{H}^{n-1}=0,
\end{eqnarray*}
given that $\supp(f)\Subset \Omega$. We finally take care of term $\mbox{(III)}$. Recalling \eqref{3.7.4}, we use Cauchy Schwarz and Young's inequalities to estimate
\begin{eqnarray*}
\mbox{(III)}&=&-2\sum_{i,j=1}^{n}\int_{\Omega\cap B_{\rr}(x_{0})\cap \{H(\snr{Dv})>0\}}\eta H(\snr{Dv})\ti{a}(\ell_{\ti{\mu}}(Dv))\snr{Dv}\gamma_{ij}^{v}D_{i}\eta D_{j}\snr{Dv} \dx\nonumber \\
&\stackrel{\eqref{3.4},\eqref{3.7.4}}{\le}&2\left(\sum_{i,j=1}^{n}\int_{\Omega\cap B_{\rr}(x_{0})}\eta^{2}H'(\snr{Dv})\snr{Dv}\ti{a}(\ell_{\ti{\mu}}(Dv))\gamma_{ij}^{v}D_{i}\snr{Dv}D_{j}\snr{Dv}\dx\right)^{1/2}\nonumber \\
&&\cdot\left(\sum_{i,j=1}^{n}\int_{\Omega\cap B_{\rr}(x_{0}))\cap \{H(\snr{Dv})>0\}}\frac{H(\snr{Dv})^{2}}{H'(\snr{Dv})}\ti{a}(\ell_{\ti{\mu}}(Dv))\snr{Dv}\gamma_{ij}^{v}D_{i}\eta D_{j}\eta\dx\right)^{1/2}\nonumber \\
&\le&-\frac{1}{4}\mathcal{J}_{2}+4\int_{\Omega\cap B_{\rr}(x_{0}))\cap \{H(\snr{Dv})>0\}}\frac{H(\snr{Dv})^{2}\snr{Dv}\ti{a}(\ell_{\ti{\mu}}(Dv))\snr{D\eta}^{2}}{H'(\snr{Dv})}\dx\nonumber \\
&&+4\int_{\Omega\cap B_{\rr}(x_{0}))\cap \{H(\snr{Dv})>0\}}\frac{H(\snr{Dv})^{2}\snr{Dv}}{H'(\snr{Dv})}\left(\frac{\snr{\ti{a}'(\ell_{\ti{\mu}}(Dv))}\snr{Dv}^{2}}{\ell_{\ti{\mu}}(Dv)}\right)\sum_{\alpha=1}^{N}\frac{\langle Dv^{\alpha},D\eta\rangle^{2}}{\snr{Dv}^{2}}\dx\nonumber \\
&\stackrel{\eqref{f.2}_{2,3}}{\le}&-\frac{1}{4}\mathcal{J}_{2}+8L\int_{\Omega\cap B_{\rr}(x_{0}))\cap \{H(\snr{Dv})>0\}}\frac{H(\snr{Dv})^{2}\snr{Dv}\snr{D\eta}^{2}}{H'(\snr{Dv})}\left(\ell_{\mu}(Dv)^{p-2}+\ell_{\mu}(Dv)^{q-2}\right)\dx.
\end{eqnarray*}
Combining the estimates for terms $\mbox{(I)}$-$\mbox{(III)}$ we obtain
\begin{flalign}\label{3.9}
\int_{\Omega\cap B_{\rr}(x_{0})}&\eta^{2}H(\snr{Dv})\langle\partial a(Dv)D^{2}v,D^{2}v\rangle \dx\nonumber \\
&+\frac{3}{4}\int_{\Omega\cap B_{\rr}(x_{0})}\eta^{2}H'(\snr{Dv})\snr{Dv}\ti{a}(\ell_{\ti{\mu}}(Dv))\gamma_{ij}^{v}D_{i}\snr{Dv}D_{j}\snr{Dv}\dx\nonumber \\
\le &\sigma\int_{\Omega\cap B_{\rr}(x_{0})}\eta^{2}\snr{DH(\snr{Dv})}^{2} \ \dx+\sigma\int_{\Omega\cap B_{\rr}(x_{0})}\eta^{2}H(\snr{Dv})\ell_{\mu}(Dv)^{p-2}\snr{D^{2}v}^{2} \ \dx\nonumber \\
&+\frac{4}{\sigma}\int_{\Omega\cap B_{\rr}(x_{0})\cap\{H(\snr{Dv})>0\}}\eta^{2}\snr{f}^{2}\left(\snr{Dv}^{2}+\frac{H(\snr{Dv})}{\ell_{\mu}(Dv)^{p-2}}\right)+\snr{D\eta}^{2}H(\snr{Dv})^{2}\dx\nonumber \\
&+8L\int_{\Omega\cap B_{\rr}(x_{0}))\cap \{H(\snr{Dv})>0\}}\frac{H(\snr{Dv})^{2}\snr{Dv}\snr{D\eta}^{2}}{H'(\snr{Dv})}\left(\ell_{\mu}(Dv)^{p-2}+\ell_{\mu}(Dv)^{q-2}\right)\dx
\end{flalign}
for $\sigma\in (0,1)$ still to be fixed. To conclude, we only need to make a proper choice of the test function $H(\cdot)$. With $\kk\ge 0$ being any number and $t\ge 0$, we define
$$
H(t)\equiv H_{\kk}(t):=\left(h(t)-\kk\right)_{+},\qquad \qquad h(t):=\int_{0}^{t}\ell_{\mu}(s)^{p-2}s\ds.
$$
In particular, we have
\begin{flalign}\label{3.10.1}
 \left\{
\begin{array}{c}
\displaystyle 
 \ H_{\kk}(\snr{z})>0 \ \Longrightarrow \ \snr{z}>0,\qquad \qquad H_{\kk}(\snr{z})>0 \ \Longleftrightarrow \ h(\snr{z})>\kk  \\[10pt]\displaystyle
 \ \frac{H_{\kk}(\snr{z})}{\ell_{\mu}(z)^{p-2}}\le \snr{z}^{2},\qquad \qquad  H_{\kk}(\snr{z})>0 \ \Longrightarrow \ H'_{\kk}(\snr{z})>0,
 \end{array}
\right.
\end{flalign}
for all $z\in \mathbb{R}^{N\times n}$, where we used that $s\mapsto \ell_{\mu}(s)^{p-2}s$ is increasing. Moreover, by the chain rule we have
\begin{flalign}\label{3.10}
DH_{\kk}(\snr{Dv})=H_{\kk}'(\snr{Dv})D\snr{Dv}=\mathds{1}_{\{h(\snr{Dv})>\kk\}}\ell_{\mu}(Dv)^{p-2}\snr{Dv}D\snr{Dv}.
\end{flalign}
We then jump back to \eqref{3.9} with the above choice of $H(\cdot)$, apply $\eqref{3.0}_{2}$, \eqref{3.4}, $\eqref{3.10.1}$ and \eqref{3.10} to get
\begin{flalign*}
\frac{1}{L}\int_{\Omega\cap B_{\rr}(x_{0})}&\eta^{2}H_{\kk}(\snr{Dv})\ell_{\mu}(Dv)^{p-2}\snr{D^{2}v}^{2}  \dx+\frac{3}{4L}\int_{\Omega\cap B_{\rr}(x_{0})}\eta^{2}\snr{DH_{\kk}(\snr{Dv})}^{2}\dx\nonumber \\
\le &\int_{\Omega\cap B_{\rr}(x_{0})}\eta^{2}H_{\kk}(\snr{Dv})\langle\partial a(Dv)D^{2}v,D^{2}v\rangle \dx\nonumber \\
&+\frac{3}{4}\int_{\Omega\cap B_{\rr}(x_{0})}\eta^{2}H_{\kk}'(\snr{Dv})\snr{Dv}\ti{a}(\ell_{\ti{\mu}}(Dv))\gamma_{ij}^{v}D_{i}\snr{Dv}D_{j}\snr{Dv}\dx\nonumber \\
\le &\sigma\int_{\Omega\cap B_{\rr}(x_{0})}\eta^{2}\snr{DH_{\kk}(\snr{Dv})}^{2} \ \dx+\sigma\int_{\Omega\cap B_{\rr}(x_{0})}\eta^{2}H_{\kk}(\snr{Dv})\ell_{\mu}(Dv)^{p-2}\snr{D^{2}v}^{2} \ \dx\nonumber \\
&+\frac{8}{\sigma}\int_{\Omega\cap B_{\rr}(x_{0})\cap\{H_{\kk}(\snr{Dv})>0\}}\eta^{2}\snr{f}^{2}\snr{Dv}^{2}\dx+\frac{4}{\sigma}\int_{\Omega\cap B_{\rr}(x_{0})}\snr{D\eta}^{2}H_{\kk}(\snr{Dv})^{2} \dx\nonumber \\
&+8L\int_{\Omega\cap B_{\rr}(x_{0}))}H_{\kk}(\snr{Dv})^{2}\snr{D\eta}^{2}\left(1+\ell_{\mu}(Dv)^{q-p}\right)\dx.
\end{flalign*}
In the previous display, we fix $\sigma:=(2L)^{-1}$, use \eqref{3.10.1} and reabsorbe terms to conclude with
\begin{flalign}\label{3.11}
 \int_{\Omega\cap B_{\rr}(x_{0})}\eta^{2}\snr{D(h(\snr{Dv})-\kk)_{+}}^{2}\dx\le&c\int_{\Omega\cap B_{\rr}(x_{0})}(h(\snr{Dv})-\kk)_{+}^{2}\snr{D\eta}^{2}\left(1+\ell_{\mu}(Dv)^{q-p}\right)\dx\nonumber \\
 &+c\int_{\Omega\cap B_{\rr}(x_{0})}\mathds{1}_{ \{h(\snr{Dv})>\kk\}}\eta^{2}\snr{f}^{2}\snr{Dv}^{2} \ \dx,
\end{flalign}
with $c\equiv c(L)$.
\subsection{Basic Caccioppoli type inequalities}
To derive rather elementary Caccioppoli type inequalities to control second derivatives, we distinguish the two cases $p\ge 2$ and $1<p<2$, in each of them we shall make a proper choice of test function $H(\cdot)$.
\subsubsection*{Case $1<p\le 2$}
Observe that
\begin{flalign*}
&\left\{
\begin{array}{c}
\displaystyle 
\ \sum_{j,k=1}^{n}\sum_{\beta=1}^{N}\delta_{ij}D_{j}D_{k}v^{\beta}D_{k}v^{\beta}=\sum_{j,k=1}^{n}\sum_{\alpha,\beta=1}^{N}\delta_{ik}\delta^{\alpha\beta}D_{k}D_{j}v^{\beta}D_{j}v^{\alpha} \\[16pt]\displaystyle
\ \sum_{j,k=1}^{n}\sum_{\alpha,\beta=1}^{N}D_{i}v^{\alpha}D_{j}v^{\alpha}D_{j}D_{k}v^{\beta}D_{k}v^{\beta}=\sum_{j,k=1}^{n}\sum_{\alpha,\beta=1}^{N}D_{i}v^{\alpha}D_{k}v^{\beta}D_{j}D_{k}v^{\beta}D_{j}v^{\alpha}
\end{array}
\right.\vspace{16pt}\nonumber \\
&\qquad \qquad \quad \Longrightarrow \ \sum_{j=1}^{n}\gamma_{ij}^{v}\partial_{x_{j}}\left(\ti{F}(\ell_{\ti{\mu}}(Dv))\right)=\sum_{j,k=1}^{n}\sum_{\alpha,\beta=1}^{N}\partial_{z^{\alpha}_{i}}a^{\beta}_{k}(Dv)D_{j}D_{k}v^{\beta}D_{j}v^{\alpha},
\end{flalign*}
therefore choosing $H(\cdot)\equiv 1$, multiplying \eqref{3.7} by $\eta^{2}$ and recalling \eqref{3.7.1} (with $H(\cdot)\equiv 1$) we obtain
\begin{eqnarray*}
\eta^{2}\langle\partial a(Dv)D^{2}v,D^{2}v\rangle&=&\eta^{2}\langle Dv,Df\rangle+\mathcal{I}-2\eta\sum_{i,j,k=1}^{n}\sum_{\alpha,\beta=1}^{N}\partial_{z^{\alpha}_{i}}a^{\beta}_{k}(Dv)D_{j}D_{k}v^{\beta}D_{j}v^{\alpha}D_{i}\eta.
\end{eqnarray*}
We then integrate the identity in the above display on $\Omega\cap B_{\rr}(x_{0})$ to get
\begin{flalign}\label{x.1}
\int_{\Omega\cap B_{\rr}(x_{0})}&\eta^{2}\langle\partial a(Dv)D^{2}v,D^{2}v\rangle\dx=-\int_{\partial \Omega\cap B_{\rr}(x_{0})}\diver(\eta^{2}Dv)f\dx+\int_{\Omega\cap B_{\rr}(x_{0})}\mathcal{I}\dx\nonumber \\
&-2\sum_{i,j,k=1}^{n}\sum_{\alpha,\beta=1}^{N}\int_{\Omega\cap B_{\rr}(x_{0})}\partial_{z^{\alpha}_{i}}a^{\beta}_{k}(Dv)D_{j}D_{k}v^{\beta}D_{j}v^{\alpha}D_{i}\eta\dx\nonumber \\
\le &\sigma\int_{\Omega\cap B_{\rr}(x_{0})}\eta^{2}\ell_{\mu}(Dv)^{p-2}\snr{D^{2}v}^{2}\dx+c\int_{\Omega\cap B_{\rr}(x_{0})}\left(\ell_{\mu}(Dv)^{p}+\ell_{\mu}(Dv)^{q}\right)\snr{D\eta}^{2}\dx
\nonumber \\
&+\frac{1}{4}\int_{\Omega\cap B_{\rr}(x_{0})}\eta^{2}\langle\partial a(Dv)D^{2}v,D^{2}v\rangle\dx+\frac{c}{\sigma}\int_{\Omega\cap B_{\rr}(x_{0})}\eta^{2}\ell_{\mu}(Dv)^{2-p}\snr{f}^{2}\dx,
\end{flalign}
where $c\equiv c(n,N,L)$ and we used estimates analogous to those employed for bounding $\mbox{(I)}$-$\mbox{(II)}$ (with $H(\cdot)\equiv 1$) above to deal with the first two integrals in \eqref{x.1}, while, concerning the third one, we applied Cauchy-Schwarz and Young's inequalities and $\eqref{3.0}_{3}$. Finally, combining $\eqref{dvp}_{1}$, \eqref{3.0}$_{2}$ and choosing $\sigma>0$ sufficiently small in \eqref{x.1} to reabsorbe terms we conclude with
\begin{eqnarray}\label{3.12}
\int_{\Omega\cap B_{\rr}(x_{0})}\eta^{2}\snr{DV_{\mu}(Dv)}^{2}\dx
&\le& c\int_{\Omega\cap B_{\rr}(x_{0})}\eta^{2}\ell_{\mu}(Dv)^{2-p}\snr{f}^{2}\dx\nonumber \nonumber\\
&&+c\int_{\Omega\cap B_{\rr}(x_{0})}\snr{D\eta}^{2}\ell_{\mu}(Dv)^{p}\left(1+\ell_{\mu}(Dv)^{q-p}\right)\dx,
\end{eqnarray}
with $c\equiv c(n,N,L,p)$.
\subsubsection*{Case $p> 2$} Now we let $H(t):=(\mu^{2}+t^{2})^{\frac{p-2}{2}}$ in \eqref{3.9.1} and observe that it is smooth, positive and nondecreasing - here it is $p> 2$ and $\mu>0$, cf. \eqref{3.3}. We notice that \eqref{3.7.4} is used only for bounding term $\mbox{(III)}$ in Section \ref{dgk}; in particular it is not required to derive \eqref{3.9.1}, so we slightly modify the estimate for term $\mbox{(I)}$ in Section \ref{dgk} as:
\begin{eqnarray*}
\mbox{(I)}&\le&\sigma\int_{\Omega\cap B_{\rr}(x_{0})}\eta^{2}\snr{Dv}^{2}\snr{D(\ell_{\mu}(Dv)^{p-2})}^{2} \ \dx+\sigma\int_{\Omega\cap B_{\rr}(x_{0})}\eta^{2}\ell_{\mu}(Dv)^{2(p-2)}\snr{D^{2}v}^{2} \ \dx\nonumber \\
&&+\frac{4}{\sigma}\int_{\Omega\cap B_{\rr}(x_{0})}\eta^{2}\snr{f}^{2}\dx+4\int_{\Omega\cap B_{\rr}(x_{0})}\snr{D\eta}^{2}\ell_{\mu}(Dv)^{2p-2} \dx,
\end{eqnarray*}
for some $\sigma\in (0,1)$ to be fixed in a few lines. Next, we notice that the precise form of $H(\cdot)\ge 0$ plays no role in the estimation of term $\mbox{(II)}$, so also now it is $\mbox{(II)}\le 0$. The bound for terms $\mbox{(III)}$ can be derived exactly as done in Section \ref{dgk} observing that $H'(t)>0$ when $t>0$ and $\snr{D\snr{Dv}}=0$ on $\left\{x\in \Omega\cap B_{\rr}(x_{0})\colon \snr{Dv(x)}=0\right\}$ by coarea formula. Plugging these variations in \eqref{3.9}, we obtain:
\begin{flalign*}
\int_{\Omega\cap B_{\rr}(x_{0})}&\eta^{2}\ell_{\mu}(Dv)^{p-2}\langle\partial a(Dv)D^{2}v,D^{2}v\rangle\dx\nonumber \\
&+\frac{3(p-2)}{4}\int_{\Omega\cap B_{\rr}(x_{0})}\eta^{2}\ell_{\mu}(Dv)^{p-4}\snr{Dv}^{2}\ti{a}(\ell_{\ti{\mu}}(Dv))\gamma^{v}_{ij}D_{i}\snr{Dv}D_{j}\snr{Dv}\dx\nonumber \\
\le &\sigma \int_{\Omega\cap B_{\rr}(x_{0})}\eta^{2}\ell_{\mu}(Dv)^{2}\snr{D(\ell_{\mu}(Dv)^{p-2})}^{2} \dx +\sigma \int_{\Omega\cap B_{\rr}(x_{0})}\eta^{2}\ell_{\mu}(Dv)^{2(p-2)}\snr{D^{2}v}^{2}\dx\nonumber \\
&+\frac{4}{\sigma}\int_{\Omega\cap B_{\rr}(x_{0})}\eta^{2}\snr{f}^{2}\dx+4\int_{\Omega\cap B_{\rr}(x_{0})}\snr{D\eta}^{2}\ell_{\mu}(Dv)^{2p-2}\dx\nonumber \\
&+\frac{8L}{p-2}\int_{\Omega\cap B_{\rr}(x_{0})}\snr{D\eta}^{2}\ell_{\mu}(Dv)^{2p-2}\left(1+\ell_{\mu}(Dv)^{q-p}\right)\dx.
\end{flalign*}
The two terms on the left-hand side of the previous inequality can be estimated from below via \eqref{3.0}$_{2}$ and \eqref{3.4}:
\begin{flalign*}
\int_{\Omega\cap B_{\rr}(x_{0})}&\eta^{2}\ell_{\mu}(Dv)^{p-2}\langle\partial a(Dv)D^{2}v,D^{2}v\rangle\dx\nonumber \\
&+\frac{3(p-2)}{4}\int_{\Omega\cap B_{\rr}(x_{0})}\eta^{2}\ell_{\mu}(Dv)^{p-4}\snr{Dv}^{2}\ti{a}(\ell_{\ti{\mu}}(Dv))\gamma^{v}_{ij}D_{i}\snr{Dv}D_{j}\snr{Dv}\dx\nonumber \\
\ge &\frac{1}{L}\int_{\Omega\cap B_{\rr}(x_{0})}\eta^{2}\ell_{\mu}(Dv)^{2(p-2)}\snr{D^{2}v}^{2}\dx\nonumber \\
&+\frac{3(p-2)}{4L}\int_{\Omega\cap B_{\rr}(x_{0})}\eta^{2}\ell_{\mu}(Dv)^{2p-6}\snr{Dv}^{2}\snr{D\snr{Dv}}^{2}\dx\nonumber \\
\ge &\frac{1}{L}\int_{\Omega\cap B_{\rr}(x_{0})}\eta^{2}\ell_{\mu}(Dv)^{2(p-2)}\snr{D^{2}v}^{2}\dx\nonumber \\
&+\frac{3}{4(p-2)L}\int_{\Omega\cap B_{\rr}(x_{0})}\eta^{2}\ell_{\mu}(Dv)^{2}\snr{D(\ell_{\mu}(Dv)^{p-2})}^{2}\dx.
\end{flalign*}
Merging the two previous displays, choosing $\sigma>0$ sufficiently small and applying \eqref{dvp}$_{2}$, we obtain
\begin{flalign}\label{3.13}
\int_{\Omega\cap B_{\rr}(x_{0})}\eta^{2}\snr{DW_{\mu}(Dv)}^{2}\dx\le c\int_{\Omega\cap B_{\rr}(x_{0})}\left(\eta^{2}\snr{f}^{2}+\snr{D\eta}^{2}\ell_{\mu}(Dv)^{2p-2}\left(1+\ell_{\mu}(Dv)^{q-p}\right)\right)\dx,
\end{flalign}
with $c\equiv c(n,N,L,p)$.
\subsection{$L^{\infty}$-bound}\label{linf} With $\rr\in (0,\rr_{*}]$ and parameters $\rr/2\le\tau_{1}<\tau_{2}\le \rr$, we define $M:=\nr{Dv}_{L^{\infty}(\Omega\cap B_{\tau_{2}}(x_{0}))}$, observe that \eqref{3.11} holds in particular for all nonnegative $\eta\in C^{1}_{c}(B_{\tau_{2}}(x_{0}))$ and rearrange it as 
\begin{eqnarray}\label{3.14}
 \int_{\Omega\cap B_{\rr}(x_{0})}\eta^{2}\snr{D(h(\snr{Dv})-\kk)_{+}}^{2}\dx &\le&cM^{q-p}\int_{\Omega\cap B_{\rr}(x_{0})}(h(\snr{Dv})-\kk)_{+}\snr{D\eta}^{2}\dx\nonumber \\
 &&+cM^{2}\int_{\Omega\cap B_{\rr}(x_{0})}\mathds{1}_{ \{h(\snr{Dv})>\kk\}}\eta^{2}\snr{f}^{2}\dx,
\end{eqnarray}
where we assumed with no loss of generality that $M\ge 1$ - the opposite inequality would result in an immediate proof of \eqref{lip}. Since \eqref{3.14} is verified for all nonnegative $\eta\in C^{1}_{c}(B_{\tau_{2}}(x_{0}))$ and $h(\snr{Dv})\in W^{1,2}(\Omega\cap B_{\rr}(x_{0}))\cap C(\bar{\Omega}\cap B_{\rr}(x_{0}))$ - recall that $v\in C^{3}(\bar{\Omega}\cap B_{\rr}(x_{0}),\mathbb{R}^{N})$ - we can apply Lemma \ref{i++} with $\kk_{0}=0$, $M_{1}:=M^{\frac{q-p}{2}}$ and $M_{2}:=M$ to get 
\begin{flalign}\label{3.15}
 \nr{h(\snr{Dv})}_{L^{\infty}(\Omega\cap B_{\tau_{1}}(x_{0}))}\le& \frac{cM^{\frac{q-p}{2}\left(1+\max\left\{\nu,\frac{n-3}{2}\right\}\right)}}{\rr^{\alpha}(\tau_{2}-\tau_{1})^{\beta}}\left(\int_{\Omega\cap B_{\tau_{2}}(x_{0})}h(\snr{Dv})^{2} \dx\right)^{\frac{1}{2}}\nonumber \\
 &+\frac{cM^{\frac{q-p}{2}\max\left\{\nu,\frac{n-3}{2}\right\}+1}[f]_{n,1;B_{\rr}(x_{0})}}{\rr^{\alpha}(\tau_{2}-\tau_{1})^{\beta}},
 \end{flalign}
for $c\equiv c(n,\Omega,\nu,L,p)$. Recalling that $h(t)=p^{-1}(\mu^{2}+t^{2})^{p/2}-p^{-1}\mu^{p}$, we have that $$\frac{2^{\frac{p}{2}}t^{p}}{p}+\frac{2^{\frac{p}{2}}\mu^{p}}{p}\ge h(t)\ge \frac{t^{p}}{p}-\frac{\mu^{p}}{p},$$
therefore we can rewrite \eqref{3.15} as
\begin{flalign}\label{3.16}
 \nr{Dv}_{L^{\infty}(\Omega\cap B_{\tau_{1}}(x_{0}))}\le&\frac{c\nr{Dv}_{L^{\infty}(\Omega\cap B_{\tau_{2}}(x_{0}))}^{\frac{1}{2}+\frac{q-p}{2p}\left(1+\max\left\{\nu,\frac{n-3}{2}\right\}\right)}}{\rr^{\alpha/p}(\tau_{2}-\tau_{1})^{\beta/p}}\left(\int_{\Omega\cap B_{\tau_{2}}(x_{0})}h(\snr{Dv})\dx\right)^{\frac{1}{2p}}\nonumber \\
 &+\frac{c\nr{Dv}_{L^{\infty}(\Omega\cap B_{\tau_{2}}(x_{0}))}^{\frac{1}{p}+\frac{q-p}{2p}\max\left\{\nu,\frac{n-3}{2}\right\}}[f]_{n,1;\Omega\cap B_{\rr}(x_{0})}^{\frac{1}{p}}}{\rr^{\alpha/p}(\tau_{2}-\tau_{1})^{\beta/p}}+c,
\end{flalign}
with $c\equiv c(n,\Omega,\nu,L,p)$. Now notice that the bound on exponents $(p,q)$ in \eqref{pq} guarantees that
$$
\frac{1}{2}+\frac{q-p}{2p}\left(1+\max\left\{\nu,\frac{n-3}{2}\right\}\right)<1\qquad \mbox{and}\qquad \frac{1}{p}+\frac{q-p}{2p}\max\left\{\nu,\frac{n-3}{2}\right\}<1,
$$
provided we choose $\nu\equiv \nu(p,q)\in (0,1/2)$ sufficiently small when $n=3$, therefore we apply Young inequality with conjugate exponents
$$
(a_{1},a_{2}):=\left(\frac{2p}{q+m_{\nu}(q-p)},\frac{2p}{p-(1+m_{\nu})(q-p)}\right),\quad (a_{3},a_{4}):=\left(\frac{2p}{2+m_{\nu}(q-p)},\frac{2p}{2(p-1)-m_{\nu}(q-p)}\right)
$$
where it is $m_{\nu}:=\max\left\{\nu,\frac{n-3}{2}\right\}$, to conclude with
\begin{eqnarray*}
 \nr{Dv}_{L^{\infty}(\Omega\cap B_{\tau_{1}}(x_{0}))}&\le&\frac{1}{4} \nr{Dv}_{L^{\infty}(\Omega\cap B_{\tau_{2}}(x_{0}))}+\frac{c}{\rr^{\alpha a_{2}/p}(\tau_{2}-\tau_{1})^{\beta a_{2}/p}}\left(\int_{\Omega\cap B_{\rr}(x_{0})}h(\snr{Dv})\dx\right)^{\frac{a_{2}}{2p}}\nonumber \\
 &&+\frac{c[f]_{n,1;\Omega\cap B_{\rr}(x_{0})}^{\frac{a_{4}}{p}}}{\rr^{\alpha a_{4}/p}(\tau_{2}-\tau_{1})^{\beta a_{4}/p}}+c,
\end{eqnarray*}
for $c\equiv c(n,\Omega,\nu,L,p,q)$. Finally, via Lemma \ref{l5} we obtain
\begin{flalign}\label{lip}
\nr{Dv}_{L^{\infty}(\Omega\cap B_{\rr/2}(x_{0}))}\le c\rr^{-b_{1}}\nr{Dv}_{L^{p}(\Omega\cap B_{\rr}(x_{0}))}^{d_{1}}+c\rr^{-b_{2}}[f]_{n,1;\Omega\cap B_{\rr}(x_{0})}^{d_{2}}+c,
\end{flalign}
with $c\equiv c(n,\Omega,L,p,q)$, $b_{1},b_{2}\equiv b_{1},b_{2}(n,p,q)$ and $d_{1},d_{2}\equiv b_{1},b_{2}(n,p,q)$ - since the size of $\nu$ is ultimately influenced only by $(p,q)$, we incorporated any dependency on $\nu$ into a dependency on $(p,q)$.
\begin{remark}\label{r4.1}
\emph{Notice that \eqref{lip} can be derived on balls $B_{\sigma}(\bar{x})$, $\bar{x}\in \bar{\Omega}\cap B_{\rr}(x_{0})$, $\sigma\in (0,\rr-\snr{x_{0}-\bar{x}})$ - recall \eqref{2.2} - in which case $\Omega\cap B_{\sigma}(\bar{x})\subset \Omega\cap B_{\rr}(x_{0})$, i.e.:
\begin{flalign}\label{lip.2}
 \nr{Dv}_{L^{\infty}(\Omega\cap B_{\sigma/2}(\bar{x}))}\le c\sigma^{-b_{1}}\nr{Dv}_{L^{p}(\Omega\cap B_{\sigma}(\bar{x}))}^{d_{1}}+c\sigma^{-b_{2}}[f]_{n,1;\Omega\cap B_{\sigma}(\bar{x})}^{d_{2}}+c,
\end{flalign}
with $c\equiv c(n,\Omega,L,p,q)$, $b_{1},b_{2}\equiv b_{1},b_{2}(n,p,q)$ and $d_{1},d_{2}\equiv d_{1},d_{2}(n,p,q)$. In fact, if $\dist(\bar{x},\partial \Omega)\ge\sigma$, then $B_{3\sigma/4}(\bar{x})\Subset \Omega$, so a Lipschitz estimate analogous to \eqref{lip} follows from \cite[Theorem 1]{BS1}\footnote{In \cite[Theorem 1]{BS1} the result is given in the scalar case $N=1$, but adopting the very same strategy presented there and the approximation procedure of \cite[Section 5]{bemi}, Lipschitz bounds follow verbatim.}. On the other hand, if $\dist(\bar{x},\partial \Omega)< \sigma$, it is $\partial \Omega\cap B_{\sigma}(\bar{x})\subset \partial \Omega\cap B_{\rr}(x_{0})$ so system \eqref{3.1} holds on $\Omega\cap B_{\sigma}(\bar{x})$ and \eqref{lip.2} on $\Omega\cap B_{\sigma}(\bar{x})$ is recovered exactly as done for \eqref{lip} in Sections \ref{dgk} and \ref{linf}.}
\end{remark}

 To conclude, let us consider concentric balls $B_{\sigma_{2}}(x_{0})\subset B_{\sigma_{1}}(x_{0})\subseteq B_{\rr}(x_{0})$. We let $\mathcal{K}\subset \mathbb{N}$ be a discrete set of indices, set $r:=(\sigma_{1}-\sigma_{2})/4$, $\mf{Z}:=\mathds{1}_{\Omega}Dv$ and, keeping \eqref{2.2} in mind, we cover $B_{\sigma_{2}}(x_{0})$ with a collection $\{B_{r/2}(x_{i})\}_{i\in \mathcal{K}}$ made by $\snr{\mathcal{K}}=c(n)(\sigma_{1}-\sigma_{2})^{-n}$ balls, with $\{x_{i}\}_{i\in \mathcal{K}}\subset B_{(3\sigma_{2}+\sigma_{1})/4}(x_{0})$ - we stress that this covering can be chosen in such a way that the finite intersection property is satisfied, in the sense that each doubled ball $B_{2r}(x_{i})$ intersects at most $8^{n}$ of other doubled balls from the same family. Recalling Remarks \ref{rf} and \ref{r4.1}, for arbitrary $t>1$ we have
\begin{eqnarray*}
\left(\mint_{\Omega\cap B_{\sigma_{2}}(x_{0})}\snr{Dv}^{t}\dx\right)^{1/t}&=&\snr{\Omega\cap B_{\sigma_{2}(x_{0})}}^{-1/t}\nr{\mf{Z}}_{L^{t}(B_{\sigma_{2}}(x_{0}))}\stackrel{\eqref{2.2}}{\le} \left(\frac{\Upsilon_{\Omega}}{\sigma_{2}^{n}}\right)^{1/t}\sum_{i\in\mathcal{K}}\nr{\mf{Z}}_{L^{t}(B_{r/2}(x_{i}))}\nonumber \\
&=&\left(\frac{\Upsilon_{\Omega}}{\sigma_{2}^{n}}\right)^{1/t}\sum_{i\in\mathcal{K}}\nr{Dv}_{L^{t}(\Omega\cap B_{r/2}(x_{i}))}\nonumber \\
&\le& \left(\frac{\Upsilon_{\Omega}}{\sigma_{2}^{n}}\right)^{1/t}\sum_{i\in \mathcal{K}}\snr{\Omega\cap B_{r/2}(x_{i})}^{1/t}\nr{Dv}_{L^{\infty}(\Omega\cap B_{r/2}(x_{i}))}\nonumber \\
&\stackrel{\eqref{2.2}}{\le}&\left(\frac{r\omega_{n}^{1/n}\Upsilon_{\Omega}^{1/n}}{2\sigma_{2}}\right)^{n/t}\sum_{i\in \mathcal{K}}\nr{Dv}_{L^{\infty}(\Omega\cap B_{r/2}(x_{i}))}\nonumber \\
&\stackrel{\eqref{lip.2}}{\le}&\frac{c}{(\sigma_{1}-\sigma_{2})^{b_{1}}}\left(\frac{r\omega_{n}^{1/n}\Upsilon_{\Omega}^{1/n}}{2\sigma_{2}}\right)^{n/t}\sum_{i\in \mathcal{K}}\left(\nr{Dv}_{L^{p}(\Omega\cap B_{r}(x_{i}))}^{d_{1}}+1\right)\nonumber \\
&&+\frac{c}{(\sigma_{1}-\sigma_{2})^{b_{2}}}\left(\frac{r\omega_{n}^{1/n}\Upsilon_{\Omega}^{1/n}}{2\sigma_{2}}\right)^{n/t}\sum_{i\in \mathcal{K}}[f]_{n,1;\Omega\cap B_{r}(x_{i})}^{d_{2}}\nonumber \\
&=&\frac{c}{(\sigma_{1}-\sigma_{2})^{b_{1}}}\left(\frac{r\omega_{n}^{1/n}\Upsilon_{\Omega}^{1/n}}{2\sigma_{2}}\right)^{n/t}\sum_{i\in \mathcal{K}}\left(\nr{\mf{Z}}_{L^{p}( B_{r}(x_{i}))}^{d_{1}}+1\right)\nonumber \\
&&+\frac{c}{(\sigma_{1}-\sigma_{2})^{b_{2}}}\left(\frac{r\omega_{n}^{1/n}\Upsilon_{\Omega}^{1/n}}{2\sigma_{2}}\right)^{n/t}\sum_{i\in \mathcal{K}}[f]_{n,1; B_{r}(x_{i})}^{d_{2}}\nonumber \\
&\le&\frac{c}{(\sigma_{1}-\sigma_{2})^{b_{1}+n}}\max\left\{1,\left(\frac{r\omega_{n}^{1/n}\Upsilon_{\Omega}^{1/n}}{2\sigma_{2}}\right)\right\}^{n/t}\left(\nr{Dv}_{L^{p}(\Omega\cap B_{\sigma_{1}}(x_{0}))}^{d_{1}}+1\right)\nonumber \\
&&+\frac{c}{(\sigma_{1}-\sigma_{2})^{b_{2}}}\max\left\{1,\left(\frac{r\omega_{n}^{1/n}\Upsilon_{\Omega}^{1/n}}{2\sigma_{2}}\right)\right\}^{n/t}[f]_{n,1;\Omega\cap B_{\sigma_{1}}(x_{0})}^{d_{2}},
\end{eqnarray*}
with $c\equiv c(n,\Omega,L,p,q)$, $b_{1},b_{2}\equiv b_{1},b_{2}(n,p,q)$, $d_{1},d_{2}\equiv d_{1},d_{2}(n,p,q)$. At this stage, we observe that none of the bounding parameters appearing above depends in an increasing fashion on $t$, therefore we can send $t\to \infty$ in the previous display to get \eqref{lip.1} below. What we proved so far can be summarized in the following proposition.
\begin{proposition}\label{p4.1}
Under assumptions \eqref{pq} and \eqref{om.1}-\eqref{3.0}, let $x_{0}\in \bar{\Omega}$ be any point and $v\in C^{3}(\bar{\Omega}\cap B_{\rr}(x_{0}),\mathbb{R}^{N})$, $0<\rr\le \rr_{*}$, be a classical solution to system \eqref{3.1}. Then, whenever $B_{\sigma_{2}}(x_{0})\subset B_{\sigma_{1}}(x_{0})\subseteq B_{\rr}(x_{0})$ are concentric balls, 
\begin{flalign}\label{lip.1}
\nr{Dv}_{L^{\infty}(\Omega\cap B_{\sigma_{2}}(x_{0}))}\le\frac{c}{(\sigma_{1}-\sigma_{2})^{a_{1}}}\left(\nr{Dv}_{L^{p}(\Omega\cap B_{\sigma_{1}}(x_{0}))}^{d_{1}}+1\right)+\frac{c}{(\sigma_{1}-\sigma_{2})^{a_{2}}}[f]_{n,1;\Omega\cap B_{\sigma_{1}}(x_{0})}^{d_{2}},
\end{flalign}
holds true with $a_{1},a_{2}\equiv a_{1},a_{2}(n,p,q)$ and $c\equiv c(n,\Omega,L,p,q)$.
\end{proposition}
\subsection{Higher differentiability}\label{hd}
Here we combine estimates \eqref{3.12}-\eqref{3.13} with the $L^{\infty}$-bound in \eqref{lip.1} to obtain higher differentiability results for suitable nonlinear vector fields depending on the gradient of solutions to \eqref{3.1}. Let $0<r_{2}<r_{1}\le \rr$ be radii and, for simplicity, let us rename the right-hand side of \eqref{lip.1} with the choice $\sigma_{1}\equiv(r_{2}+r_{1})/2$ and $\sigma_{2}\equiv r_{2}$ as
$$
\mathcal{M}:=(r_{1}-r_{2})^{-a_{1}}\left(\nr{Dv}_{L^{p}(\Omega\cap B_{(r_{2}+r_{1})/2}(x_{0}))}^{d_{1}}+1\right)+(r_{1}-r_{2})^{-a_{2}}[f]_{n,1;\Omega\cap B_{(r_{1}+r_{2})/2}(x_{0})}^{d_{2}}
$$
\subsubsection*{Case $1<p\le 2$} We observe that \eqref{3.12} holds for all nonnegative $\eta\in C^{1}_{c}(B_{\rr}(x_{0}))$, therefore we can pick any of them such that $\mathds{1}_{B_{r_{2}}(x_{0})}\le \eta\le \mathds{1}_{B_{(r_{1}+r_{2})/2}(x_{0})}$ and $\snr{D\eta}\lesssim (r_{1}-r_{2})^{-1}$ to get
\begin{eqnarray}\label{d2-}
\int_{\Omega\cap B_{r_{2}}(x_{0})}\snr{DV_{\mu}(Dv)}^{2}\dx&\le&  c\left(1+\nr{Dv}_{L^{\infty}(\Omega\cap B_{(r_{1}+r_{2})/2}(x_{0}))}^{2-p}\right)\int_{\Omega\cap B_{(r_{1}+r_{2})/2}(x_{0})}\snr{f}^{2}\dx\nonumber \\
&&+\frac{c\left(1+\nr{Dv}_{L^{\infty}(\Omega\cap B_{(r_{1}+r_{2})/2}(x_{0}))}^{q-p}\right)}{(r_{1}-r_{2})^{2}}\int_{\Omega\cap B_{(r_{1}+r_{2})/2}(x_{0})}\ell_{\mu}(Dv)^{p}\dx\nonumber \\
&\stackrel{\eqref{lip.1}}{\le}&c\mathcal{M}^{2-p}\int_{\Omega\cap B_{r_{1}}(x_{0})}\snr{f}^{2}\dx+\frac{c\mathcal{M}^{q-p}}{(r_{1}-r_{2})^{2}}\int_{\Omega\cap B_{r_{1}}(x_{0})}\ell_{\mu}(Dv)^{p}\dx,
\end{eqnarray}
for $c\equiv c(n,N,\Omega,L,p,q)$. 

\subsubsection*{Case $p> 2$} Being \eqref{3.13} valid for all nonnegative $\eta\in C^{1}_{c}(B_{\rr}(x_{0}))$, as done before we select $\eta\in C^{1}_{c}(B_{r_{1}}(x_{0}))$ so that $\mathds{1}_{B_{r_{2}}(x_{0})}\le \eta\le \mathds{1}_{B_{(r_{1}+r_{2})/2}(x_{0})}$ and $\snr{D\eta}\lesssim (r_{1}-r_{2})^{-1}$ and plug it into \eqref{3.13} for obtaining
\begin{eqnarray}\label{d2+}
\int_{\Omega\cap B_{r_{2}}(x_{0})}\snr{DW_{\mu}(Dv)}^{2}\dx&\le& c\int_{\Omega\cap B_{(r_{1}+r_{2})/2}(x_{0})}\eta^{2}\snr{f}^{2}\dx\nonumber \\
&&+\frac{c}{(r_{1}-r_{2})^{2}}\left(1+\nr{Dv}_{L^{\infty}(\Omega\cap B_{(r_{1}+r_{2})/2}(x_{0}))}^{q-2}\right)\int_{\Omega\cap B_{r_{1}}(x_{0})}\ell_{\mu}(Dv)^{p}\dx\nonumber \\
&\stackrel{\eqref{lip.1}}{\le}&c\int_{\Omega\cap B_{r_{1}}(x_{0})}\snr{f}^{2}\dx+\frac{c\mathcal{M}^{q-2}}{(r_{1}-r_{2})^{2}}\int_{\Omega\cap B_{r_{1}}(x_{0})}\ell_{\mu}(Dv)^{p}\dx,
\end{eqnarray}
with $c\equiv c(n,N,\Omega,L,p,q)$. Recalling the explicit expression of $\mathcal{M}$, a straightforward manipulation of \eqref{d2-}-\eqref{d2+} yields \eqref{hd1}-\eqref{hd2}.
\begin{proposition}\label{p4.2}
Under the same assumptions of Proposition \ref{p4.1}, let $v\in C^{3}(\bar{\Omega}\cap B_{\rr}(x_{0}),\mathbb{R}^{N})$ be a classical solution to system \eqref{3.1} and $B_{r_{2}}(x_{0})\subset B_{r_{1}}(x_{0})\subseteq B_{\rr}(x_{0})$ be concentric balls.
\begin{itemize}
    \item If $1<p\le 2$, it is $V_{p}(Dv)\in W^{1,2}(\Omega\cap B_{r_{2}}(x_{0}),\mathbb{R}^{N\times n})$ and
 \begin{flalign}\label{hd1}
\nr{DV_{\mu}(Dv)}_{L^{2}(\Omega\cap B_{r_{2}}(x_{0}))}\le c(r_{1}-r_{2})^{-a_{1}}\nr{\ell_{\mu}(Dv)}_{L^{p}(\Omega\cap B_{r_{1}}(x_{0}))}^{d_{1}}+c(r_{1}-r_{2})^{-a_{2}}[f]_{n,1;\Omega\cap B_{r_{1}}(x_{0})}^{d_{2}}
\end{flalign}
is verified for $a_{1},a_{2},d_{1},d_{2}\equiv a_{1},a_{2},d_{1},d_{2}(n,p,q)$ and $c\equiv c(n,N,\Omega,L,p,q)$.
\item If $p>2$, it is $W_{p}(Dv)\in W^{1,2}(\Omega\cap B_{r_{2}}(x_{0}),\mathbb{R}^{N\times n})$ and
\begin{flalign}\label{hd2}
\nr{DW_{\mu}(Dv)}_{L^{2}(\Omega\cap B_{r_{2}}(x_{0}))}\le c(r_{1}-r_{2})^{-a_{1}}\nr{\ell_{\mu}(Dv)}_{L^{p}(\Omega\cap B_{r_{1}}(x_{0}))}^{d_{1}}+c(r_{1}-r_{2})^{-a_{2}}[f]_{n,1;\Omega\cap B_{r_{1}}(x_{0})}^{d_{2}}
\end{flalign}
holds with $a_{1},a_{2},d_{1},d_{2}\equiv a_{1},a_{2},d_{1},d_{2}(n,p,q)$ and $c\equiv c(n,N,\Omega,L,p,q)$.
\end{itemize}
\end{proposition}

\section{Approximation scheme}\label{as}
The ultimate goal of this section is the proof of Theorems \ref{t1}-\ref{t2}, which essentially consists in transferring the a priori estimates \eqref{lip.1} and \eqref{hd1}-\eqref{hd2} to minimizers of \eqref{fun}. A quick inspection of the arguments leading to Propositions \ref{p4.1}-\ref{p4.2} shows that they crucially rely on the possibility of working with classical solutions to system \eqref{3.1} on a $C^{2}$-regular, convex domain $\Omega$ and the smoothness of $\partial \Omega$ never enters in a quantitative form. This means that we need to construct highly regular functionals approximating \eqref{fun} governing suitable Dirichlet problems defined on smooth convex domains properly approximating $\Omega$.

\subsection{Approximation of the domain} \label{ad} Let $\Omega\subset \mathbb{R}^{n}$ be as in \eqref{om}. With $\varepsilon\in (0,1]$, we apply the approximation procedure described in Section \ref{cd} to derive a sequence of smooth, convex domains $\{\Omega_{\varepsilon}\}_{\varepsilon\in (0,1]}$ satisfying \eqref{3.8.2}-\eqref{3.8.2.1}. Since we are going to apply on $\Omega_{\varepsilon}$ the whole machinery built in Sections \ref{dgp}-\ref{apr}, the validity of $\eqref{3.8.2.1}$ yields that all the constants appearing in Lemma \ref{i++} and Propositions \ref{p4.1}-\ref{p4.2} are stable as $\varepsilon\to 0$ and ultimately depend only on $\Omega$.
\subsection{Approximating Dirichlet problems}\label{ri} 

With $F\colon \mathbb{R}^{N\times n}\to \mathbb{R}$ as in \eqref{f.0}-\eqref{f.1}, $f$ satisfying in \eqref{ff} and $\delta,\varepsilon\in (0,1]$, we extend (without relabelling) $\ti{F}(\cdot)$ by even reflection making it defined on the whole $\mathbb{R}$, i.e. $\ti{F}(t)=\ti{F}(-t)$ when $t<0$ - it is $\ti{F}'(0)=0$, - consider a symmetric mollifier $\psi\in C^{\infty}_{c}(-1,1)$ such that $\nr{\psi}_{L^{1}(\mathbb{R})}=1$, $(-3/4,3/4)\subset \supp(\psi)$, $\psi_{\delta}(t):=\delta^{-1}\psi(t/\delta)$, set $\mu_{\delta}:=\mu+\delta$ and define 
\begin{flalign*}
&\ti{F}_{\delta}(t):=\int_{-1}^{1}\ti{F}(t+s\delta)\psi(s)\ds,\qquad F_{\delta}(z):=\ti{F}_{\delta}(\ell_{\delta}(z)),\\
& \ti{a}_{\delta}(t):=t^{-1}\ti{F}_{\delta}'(t) \ \ \mbox{for} \ \ t\not =0,\qquad\ \ F_{\delta\varepsilon}(z):=F_{\delta}(z)+\texttt{d}_{\varepsilon}^{\frac{2n}{n-1}}\ell_{\mu_{\delta}}(z)^{q}.
\end{flalign*}
By construction $F_{\delta\varepsilon}\in C^{\infty}_{\loc}(\mathbb{R}^{N\times n})$ and, by \eqref{f.1} and \cite[Lemma 12.2]{demi1} we immediately have that
\begin{flalign}\label{5.1}
 \begin{cases}
 \ L^{-1}\ell_{\mu_{\delta}}(z)^{p}+\texttt{d}_{\varepsilon}^{\frac{2n}{n-1}} \ell_{\mu_{\delta}}(z)^{q}-L\mu_{\delta}^{p}\le F_{\delta\varepsilon}(z)\le L\ell_{\mu_{\delta}}(z)^{p}+L\ell_{\mu_{\delta}}(z)^{q}\\
 \ L^{-1}\ell_{\mu_{\delta}}(z)^{p-2}\snr{\xi}^{2}+L^{-1}\texttt{d}_{\varepsilon}^{\frac{2n}{n-1}} \ell_{\mu_{\delta}}(z)^{q-2}\snr{\xi}^{2}\le \langle \partial^{2}F_{\delta\varepsilon}(z)\xi,\xi\rangle\\
 \ \snr{\partial^{2}F_{\delta\varepsilon}(z)}\le L\ell_{\mu_{\delta}}(z)^{p-2}+L\ell_{\mu_{\delta}}(z)^{q-2},
 \end{cases}
\end{flalign}
for all $z,\xi\in \mathbb{R}^{N\times n}$, with $L\equiv L(n,N,\Lambda,p,q)\ge 1$ and
\begin{flalign}\label{5.2}
 F_{\delta}(z)\to_{\delta\to 0} F(z) \ \ \mbox{uniformly on compact subsets of }\mathbb{R}^{N\times n}.
\end{flalign}
In particular, setting $$\ti{a}_{\delta\varepsilon}(t):=\ti{a}_{\delta}(t)+q\texttt{d}_{\varepsilon}^{\frac{2n}{n-1}}(\mu_{\delta}^{2}+t^{2})^{\frac{q-2}{2}} \ \ \mbox{for }t\not=0\qquad \mbox{and}\qquad  a_{\delta\varepsilon}(z):=\ti{a}_{\delta\varepsilon}(\ell_{\delta}(z))z,$$
it holds that
\begin{flalign}\label{f.2.1}
 \left\{
\begin{array}{c}
\displaystyle 
 \ L^{-1}\ell_{\mu_{\delta}}(z)^{p-2}+L^{-1}\texttt{d}_{\varepsilon}^{\frac{2n}{n-1}}\ell_{\mu_{\delta}}(z)^{q-2}\le\ti{a}_{\delta\varepsilon}(\ell_{\delta}(z))+\frac{\ti{a}_{\delta\varepsilon}'(\ell_{\delta}(z))\snr{z}^{2}}{\ell_{\delta}(z)}\le L\ell_{\mu_{\delta}}(z)^{p-2}+L \ell_{\mu_{\delta}}(z)^{q-2}\\[12pt]\displaystyle
 \ L^{-1}\ell_{\mu_{\delta}}(z)^{p-2}+L^{-1}\texttt{d}_{\varepsilon}^{\frac{2n}{n-1}}\ell_{\mu_{\delta}}(z)^{q-2}\le \ti{a}_{\delta\varepsilon}(\ell_{\ti{\mu}}(z))\le L\ell_{\mu_{\delta}}(z)^{p-2}+L \ell_{\mu_{\delta}}(z)^{q-2}\\[12pt]\displaystyle
 \ \frac{\snr{\ti{a}_{\delta\varepsilon}'(\ell_{\delta}(z))}\snr{z}^{2}}{\ell_{\delta}(z)}\le L\ell_{\mu_{\delta}}(z)^{p-2}+L \ell_{\mu_{\delta}}(z)^{q-2},
 \end{array}
\right.
\end{flalign}
for all $z\in \mathbb{R}^{N\times n}$ and $L\equiv L(n,N,\Lambda,p,q)$, and, for any $M>0$ there is a constant $A_{M}\equiv A_{M}(n,N,\Lambda,p,q,M)>0$ such that
\begin{flalign}\label{f.2.1.1}
0< t\le M \ \Longrightarrow \ \delta\snr{\ti{a}_{\delta\varepsilon}''(t)}t^{2}\le A_{M}\texttt{d}_{\varepsilon}^{-\frac{2n}{n-1}}\ti{a}_{\delta\varepsilon}(t).
\end{flalign}
The bounds in $\eqref{f.2.1}$ can be derived as those in \eqref{f.2} with $\delta$ instead of $\ti{\mu}$ and $\mu_{\delta}$ replacing $\mu$, so we only need to quickly justify $\eqref{f.2.1.1}$. From $\eqref{5.1}_{3}$, \eqref{f.0} and $\eqref{f.1}_{3}$ we deduce that $\snr{\ti{F}'_{\delta}(t)}+\snr{F''_{\delta}(t)}(\mu_{\delta}^{2}+t^{2})^{1/2}\le L(\mu_{\delta}^{2}+t^{2})^{(p-1)/2}+L(\mu_{\delta}^{2}+t^{2})^{(q-1)/2}$, so recalling the very definition of $\ti{a}_{\delta\varepsilon}(\cdot)$, a quick computation shows that $\ti{a}_{\delta\varepsilon}''(t)=t^{-1}(\ti{F}''*\psi_{\delta}')(t)-2t^{-2}\ti{F}_{\delta}''(t)+2t^{-3}\ti{F}_{\delta}'(t)+q(q-2)\texttt{d}_{\varepsilon}^{\frac{2n}{n-1}}(\mu_{\delta}+t^{2})^{(q-4)/2}[1+(q-4)t^{2}/(\mu_{\delta}^{2}+t^{2})]$, therefore if $M>0$ is any constant and $0< t\le M$, \eqref{f.2.1.1} follows by combining the previous considerations with \eqref{f.2.1}$_{2}$. Next, let $u\in W^{1,p}(\Omega,\mathbb{R}^{N})$ be a minimizer of \eqref{fun}, $B_{2\rr}(x_{0})\subset \mathbb{R}^{n}$ be a ball centered in $x_{0}\in \partial \Omega$ with $0<\rr\le \rr_{*}$ and assume that $u=0$ on $\partial \Omega\cap B_{2\rr}(x_{0})$ in the sense of traces. Since we are only interested in the behavior of $u$ on the intersection $\Omega\cap B_{\rr}(x_{0})$, we can assume that $u\in W^{1,p}_{0}(\Omega,\mathbb{R}^{N})$ and extend it as $u\equiv 0$ in $\mathbb{R}^{n}\setminus \Omega$. We further introduce a nonnegative, radially symmetric mollifier of $\mathbb{R}^{n}$, $\phi\in C^{\infty}_{c}(B_{1}(0))$, $\nr{\phi}_{L^{1}(\mathbb{R}^{n})}=1$, $\phi_{\varepsilon}(x):=(\texttt{d}_{\varepsilon}/4)^{-n}\phi(4x/\texttt{d}_{\varepsilon})$ and let $\ti{u}_{\varepsilon}:=u*\phi_{\texttt{d}_{\varepsilon}/4}$, $f_{\varepsilon}:=f*\phi_{\texttt{d}_{\varepsilon}/4}$. Recalling Remark \ref{rf} and that $u=0$ in the sense of traces of $\partial \Omega\cap B_{2\rr}(x_{0})$, we can conclude that
\eqn{5.3}
$$\left.\ti{u}_{\varepsilon}\right|_{\partial \Omega_{\varepsilon}\cap B_{\rr}(x_{0})}=0,\qquad\qquad \qquad  f_{\varepsilon}\in C^{\infty}_{c}(\Omega_{\varepsilon},\mathbb{R}^{N}).$$
We then consider the approximating Dirichlet problems
\eqn{pde}
$$
\ti{u}_{\varepsilon}+W^{1,q}_{0}(\Omega_{\varepsilon}\cap B_{\rr}(x_{0}),\mathbb{R}^{N})\ni w\mapsto  \min_{\ti{u}_{\varepsilon}+W^{1,q}_{0}(\Omega_{\varepsilon}\cap B_{\rr}(x_{0}),\mathbb{R}^{N})}\mathcal{F}_{\delta\varepsilon}(w;\Omega_{\varepsilon}\cap B_{\rr}(x_{0})),
$$
where we set
$$
\mathcal{F}_{\delta\varepsilon}(w;\Omega_{\varepsilon}\cap B_{\rr}(x_{0})):=\int_{\Omega_{\varepsilon}\cap B_{\rr}(x_{0})}\left[F_{\delta\varepsilon}(Dw)-f_{\varepsilon}\cdot w\right]\dx,
$$
which is well defined for all $w\in W^{1,q}(\Omega_{\varepsilon}\cap B_{\rr}(x_{0}),\mathbb{R}^{N})$. Standard direct methods, convexity and minimality arguments yield that problem \eqref{pde} admits a unique solution $u_{\delta\varepsilon}\in \ti{u}_{\varepsilon}+W^{1,q}_{0}(\Omega_{\varepsilon}\cap B_{\rr}(x_{0}),\mathbb{R}^{N})$ solving the Euler-Lagrange system
\eqn{ela}
$$
\int_{\Omega_{\varepsilon}\cap B_{\rr}(x_{0})}\left[\langle a_{\delta\varepsilon}(Du_{\varepsilon}),D\varphi\rangle-f_{\varepsilon}\cdot\varphi\right]\dx=0\qquad \mbox{for all} \ \ \varphi\in W^{1,q}_{0}(\Omega_{\varepsilon}\cap B_{\rr}(x_{0}),\mathbb{R}^{N});
$$
so recalling \eqref{5.3}$_{1}$, $u_{\delta\varepsilon}$ is a weak solution to the nonhomogeneous system
\eqn{5.5}
$$
\begin{cases}
\ -\diver \ a_{\delta\varepsilon}(Du_{\delta\varepsilon})=f_{\varepsilon}\quad &\mbox{in} \ \ \Omega_{\varepsilon}\cap B_{\rr}(x_{0})\\
\ u_{\delta\varepsilon}=0\quad &\mbox{on} \ \ \partial \Omega_{\varepsilon}\cap B_{\rr}(x_{0}).
\end{cases}
$$ 
To assure the applicability of Propositions \ref{p4.1}-\ref{p4.2} we need to show that 
\eqn{5.7}
$$u_{\delta\varepsilon}\in C^{3}(\bar{\Omega}_{\varepsilon}\cap B_{\rr}(x_{0}),\mathbb{R}^{N}),$$ in other words, that $u_{\delta\varepsilon}$ is a classical, smooth solution to system \eqref{5.5}. 
\subsection{Flattening of the boundary, reflection and reduction to the interior}
The regularity information in \eqref{5.7} can be retrieved by first constructing a suitable diffeomorphism to flatten the boundary of $\Omega_{\varepsilon}$, and then, via reflection, reducing to a local problem for which a complete regularity theory is available. This and classical Schauder type arguments eventually lead to \eqref{5.7}. 
\subsubsection*{Flattening of the boundary} Owing to the structure conditions of the approximating integrands, the flattening of the boundary procedure of \cite[Section 5.1]{bdms} applies verbatim. In fact, there is $\sigma>0$ such that on the tubular neighborhood $U_{\sigma}:=\left\{x\in \mathbb{R}^{n}\colon \dist(x,\partial \Omega_{\varepsilon})<\sigma\right\}$ the nearest point retraction $\pi\colon U_{\sigma}\to \partial \Omega_{\varepsilon}$ exists and is smooth. Up to some rigid motions of $\mathbb{R}^{n}$, we can assume that around the origin, $\partial \Omega_{\varepsilon}$ is the graph of some $\phi\in C^{\infty}(B_{\tau}^{n-1}(0))$, where $B_{\tau}^{n-1}(0)$ denotes the open ball of $\mathbb{R}^{n-1}$ and produce a regular parametrization $\Phi\colon B_{\tau}^{n-1}(0)\times (-\sigma,\sigma)\to \mathbb{R}^{n}$ of the neighborhood $N_{0}:=\left\{x\in U_{\sigma}\colon \pi(x)\in \textnormal{graph}(\phi)\right\}$ of zero. The map $\Phi$ is invertible and its regular inverse $\Psi\colon N_{0}\to B_{\tau}^{n-1}(0)\times (-\sigma,\sigma)$ defines the local coordinates that we need. Precisely, for $z,z_{1},z_{2}\in \mathbb{R}^{N\times n}$, we set 
\begin{flalign*}
&Q(x):=\langle D\Psi(x)^{t}, D\Psi(x)\rangle, \qquad Q_{x}(z_{1},z_{2}):=\sum_{\alpha=1}^{N}\sum_{i,j=1}^{n}Q_{ij}(x)(z_{1})^{\alpha}_{i}(z_{2})^{\alpha}_{j},\qquad Q_{\Phi}:=Q\circ \Phi,\nonumber \\
&\qquad \qquad \qquad \quad  \snr{z}_{Q_{x}}:=\sqrt{Q(x)\langle z,z\rangle},\qquad\qquad  \hat{\ell}_{s}(z):=\sqrt{s^{2}+\snr{z}_{Q_\Phi}^{2}},
\end{flalign*}
and, with $u_{\delta\varepsilon}$ being the solution to \eqref{pde} and $\varphi\in W^{1,q}_{0}(\Omega_{\varepsilon}\cap B_{\rr}(x_{0}))$ being any function, define
\begin{flalign*}
&B_{\tau}^{n-1}(0)\times (-\sigma,\sigma)\ni y\mapsto \hat{u}(y):=u_{\delta\varepsilon}(\Phi(y)) \ \Longleftrightarrow \ u_{\delta\varepsilon}(x)=\hat{u}(\Psi(x))\nonumber \\
&B_{\tau}^{n-1}(0)\times (-\sigma,\sigma)\ni y\mapsto \hat{\varphi}(y):=\varphi(\Phi(y)) \ \Longleftrightarrow \ \varphi(x)=\hat{\varphi}(\Psi(x)).
\end{flalign*}
Notice that since $u_{\delta\varepsilon}\in W^{1,q}(\Omega_{\varepsilon}\cap B_{\rr}(x_{0}),\mathbb{R}^{N})$ verifies $\left.u_{\delta\varepsilon}\right|_{\partial \Omega_{\varepsilon}\cap B_{\rr}(x_{0})}=0$, then
$$
 \hat{u}\in W^{1,q}(B_{\tau}^{n-1}(0)\times (-\sigma,0),\mathbb{R}^{N})  \ \ \mbox{with}\ \  \hat{u}=0 \ \ \mbox{on} \ \ B_{\tau}^{n-1}(0)\times \{0\}.
$$
Proceeding exactly as in \cite[Section 5.1]{bdms} we observe that we can choose $\sigma$ sufficiently small - e.g. $0<\sigma<\texttt{d}_{\varepsilon}/8$ - so that $f_{\varepsilon}\equiv 0$ on $U_{\sigma}$ and system \eqref{ela} can be turned into
\eqn{5.9}
$$
\int_{B_{\tau}^{n-1}(0)\times (-\sigma,0)}\ti{a}_{\delta\varepsilon}(\hat{\ell}_{\delta}(D\hat{u}))Q_{\Phi}(D\hat{u},D\hat{\varphi})\mathcal{J}_{\Phi}\dy=0,
$$
holding for all $\hat{\varphi}\in W^{1,q}_{0}(B_{\tau}^{n-1}(0)\times (-\sigma,0),\mathbb{R}^{N})$. \subsubsection*{Reflection and reduction to local problem}
Here we apply the content of \cite[Sections 5.2-5.3]{bdms}: we extend $Q_{\Phi}$ and $\mathcal{J}_{\Phi}$ to $B_{\tau}^{n-1}(0)\times (0,\sigma)$ by even reflection - observe that both the extensions are Lipschitz continuous on $B_{\tau}^{n-1}(0)\times (-\sigma,\sigma)$. In particular, since $\mathcal{J}_{\Phi}$ is smooth on $B_{\tau}^{n-1}(0)\times (-\sigma,0]$, its extension is smooth on $B_{\tau}^{n-1}(0)\times [0,\sigma)$ and all the horizontal derivatives of the extended Jacobian are continuous across $B_{\tau}^{n-1}(0)\times \{0\}$. Moreover, up to further reduce $\tau$ or $\sigma$, we have
\eqn{5.11}
$$
 \left\{
\begin{array}{c}
\displaystyle 
 \ \snr{z}_{Q_{\Phi(y)}}\approx \snr{z}\quad \mbox{for all} \ \ z\in \mathbb{R}^{N\times n} \ \ \mbox{and any} \ \ y\in B_{\tau}^{n-1}(0)\times (-\sigma,\sigma)\\[10pt]\displaystyle
 \ \frac{1}{2}\le \mathcal{J}_{\Phi}(y)\le 2\quad \mbox{for any} \ \ y\in B_{\tau}^{n-1}(0)\times (-\sigma,\sigma),
 \end{array}
\right.
$$
with constants implicit in "$\approx$" depending at most on $n$, and we can find an absolute, finite constant $\hat{L}>0$ bounding the Lipschitz constants of $Q_{\Phi}$ and of $\mathcal{J}_{\Phi}$ on $B_{\tau}^{n-1}(0)\times (-\sigma,\sigma)$. We further extend $\hat{u}$ by odd reflection, pick a test function $\hat{\varphi}\in W^{1,q}_{0}(B_{\tau}^{n-1}(0)\times (-\sigma,\sigma),\mathbb{R}^{N})$, split it as the sum of even part and odd part, i.e. $\hat{\varphi}=\hat{\varphi}_{e}+\hat{\varphi}_{o}$ and manipulate \eqref{5.9} as done in \cite[Section 5.2]{bdms} to get
\begin{flalign}\label{5.12}
\int_{B_{\tau}^{n-1}(0)\times (-\sigma,\sigma)}&\ti{a}_{\delta\varepsilon}(\hat{\ell}_{\delta}(D\hat{u}))Q_{\Phi}(D\hat{u},D\hat{\varphi})\mathcal{J}_{\Phi}\dy\nonumber \\
=&2\int_{B_{\tau}^{n-1}(0)\times (-\sigma,0)}\ti{a}_{\delta\varepsilon}(\hat{\ell}_{\delta}(D\hat{u}))Q_{\Phi}(D\hat{u},D\hat{\varphi}_{o})\mathcal{J}_{\Phi}\dy=0,
\end{flalign}
where we used that $\hat{\varphi}_{0}\equiv 0$ on $B_{\tau}^{n-1}(0)\times \{0\}$, and this makes it an admissible test function in \eqref{5.9}. Combining \eqref{f.2.1} and \eqref{5.11}, we see that the nonlinear map $\left(B_{\tau}^{n-1}(0)\times (-\sigma,\sigma)\right)\times \mathbb{R}^{N\times n}\mapsto \ti{a}_{\delta\varepsilon}(\hat{\ell}_{\delta}(z))\mathcal{J}_{\Phi}(y)$ fits into those covered by the main theorem in \cite{tk}\footnote{In \cite{tk}, $C^{1}$-regularity is assumed with respect to the $y$-variable and \eqref{f.2.1.1} holds for all $t\ge 0$. However, as revealed by a quick inspection of the proof, Lipschitz continuity in the space depending coefficients suffices and \eqref{f.2.1.1} is not necessary to obtain gradient boundedness. This means that the validity of \eqref{f.2.1.1} only on closed subintervals of $[0,\infty)$ is enough to prove gradient H\"older continuity.}, so $\hat{u}\in C^{1,\beta_{0}}_{\loc}(B_{\tau}^{n-1}(0)\times (-\sigma,\sigma),\mathbb{R}^{N})\cap W^{2,2}_{\loc}(B_{\tau}^{n-1}(0)\times (-\sigma,\sigma),\mathbb{R}^{N})$, which in particular yields the $C^{1,\beta_{0}}$-regularity of $u_{\delta\varepsilon}$ in a neighborhood of $\partial \Omega_{\varepsilon}\cap B_{\rr}(x_{0})$. The interior regularity follows by analogous means as the right-hand side term $f_{\varepsilon}$ is bounded. We can then conclude that $u_{\delta\varepsilon}\in C^{1,\beta_{0}}(\bar{\Omega}_{\varepsilon}\cap B_{\rr}(x_{0}),\mathbb{R}^{N})$ for some $\beta_{0}\in (0,1)$. We can then differentiate system \eqref{5.12} in direction $e_{i}$, $i\in \{1,\cdots,n-1\}$ to prove that $\ti{v}^{\beta}:=\partial_{y_{i}}\hat{u}^{\beta}$ is a weak solution to the linear elliptic system in divergence form
$$
-\sum_{\beta=1}^{N}\sum_{j,k=1}^{n}\partial_{y_{j}}\left(A^{\alpha\beta}_{jk}D_{k}\ti{v}^{\beta}\right)=\sum_{j=1}^{n}\partial_{y_{j}}g^{\alpha}_{j},
$$
where it is
$$
A^{\alpha\beta}_{jk}:=\left[\ti{a}_{\delta\varepsilon}(\hat{\ell}_{\delta}(D\hat{u}))(Q_{\Phi})_{jk}\delta^{\alpha\beta}+\frac{\ti{a}_{\delta\varepsilon}'(\hat{\ell}_{\delta}(D\hat{u}))}{\hat{\ell}_{\delta}(D\hat{u})}\sum_{l,m=1}^{n}(Q_{\Phi})_{jl}(Q_{\Phi})_{mk}D_{l}\hat{u}^{\alpha}D_{m}\hat{u}^{\beta}\right]\mathcal{J}_{\Phi}
$$
and
\begin{eqnarray*}
g^{\alpha}_{j}&:=&\sum_{k,l,m=1}^{n}\sum_{\gamma=1}^{N}\frac{\ti{a}'_{\delta\varepsilon}(\hat{\ell}_{\delta}(D\hat{u}))}{2\hat{\ell}_{\delta}(D\hat{u})}\partial_{y_{i}}(Q_{\Phi})_{lm}D_{l}\hat{u}^{\gamma}D_{m}\hat{u}^{\gamma}(Q_{\Phi})_{jk}D_{k}\hat{u}^{\alpha}\mathcal{J}_{\Phi}\nonumber \\
&&+\sum_{k=1}^{n}\ti{a}_{\delta\varepsilon}(\hat{\ell}_{\delta}(D\hat{u}))\partial_{y_{i}}\left((Q_{\Phi})_{jk}\mathcal{J}_{\Phi}\right)D_{k}\hat{u}^{\alpha}.
\end{eqnarray*}
Notice that differentiation in the $e_{n}$ direction is forbidden as $\partial_{y_{n}}\mathcal{J}_{\Phi}$ and $\partial_{y_{n}}Q_{\Phi}$ may be discontinuous across $B_{\tau}^{n-1}(0)\times\{0\}$, while the horizontal derivatives $\partial_{y_{i}}$, $i\in \{1,\cdots,n\}$ are Lipschitz continuous on $B_{\tau}^{n-1}(0)\times (-\sigma,\sigma)$. On the other hand, functions $A^{\alpha\beta}_{jk}$ and $g^{\alpha}_{j}$, $j,k\in \{1,\cdots,n\}$, $\alpha,\beta\in \{1,\cdots,N\}$ are H\"older continuous, therefore by standard regularity theory on linear, elliptic systems in divergence form \cite[Theorem 5.19]{giamar} we obtain that $\ti{v}$ is $C^{1,\beta_{0}}$-regular, so $D^{2}_{jk}\hat{u}$ is H\"older continuous on $B_{\tau}^{n-1}(0)\times (-\sigma,\sigma)$ for all $j,k$ such that $j+k<2n$. We then jump back to $B_{\tau}^{n-1}(0)\times (-\sigma,0)$, rewrite system \eqref{5.12} in nondivergence form and solve it with respect to $D^{2}_{n}\hat{u}$; we obtain:
\eqn{5.13}
$$
\sum_{\beta=1}^{N}B^{\alpha\beta}D^{2}_{n}\hat{u}^{\beta}=G^{\alpha},
$$
with 
$$
B^{\alpha\beta}:=\left[\ti{a}_{\delta\varepsilon}(\hat{\ell}_{\delta}(D\hat{u}))(Q_{\Phi})_{nn}\delta^{\alpha\beta}+\frac{\ti{a}_{\delta\varepsilon}'(\hat{\ell}_{\delta}(D\hat{u}))}{\hat{\ell}_{\delta}(D\hat{u})}\sum_{jl=1}^{n}(Q_{\Phi})_{nl}(Q_{\Phi})_{jn}D_{l}\hat{u}^{\beta}D_{j}\hat{u}^{\alpha}\right]\mathcal{J}_{\Phi}
$$
and
\begin{eqnarray*}
G^{\alpha}&:=&-\sum_{j=1}^{n}\sum_{k=1}^{n-1}\partial_{y_{k}}\left(\ti{a}_{\delta\varepsilon}(\hat{\ell}_{\delta}(D\hat{u}))(Q_{\Phi})_{jk}D_{j}\hat{u}^{\alpha}\mathcal{J}_{\Phi}\right)\nonumber \\
&&-\sum_{j=1}^{n}\ti{a}_{\delta\varepsilon}(\hat{\ell}_{\delta}(D\hat{u}))D_{j}\hat{u}^{\alpha}\partial_{y_{n}}\left((Q_{\Phi})_{jn}\mathcal{J}_{\Phi}\right)-\sum_{j=1}^{n-1}\ti{a}_{\delta\varepsilon}(\hat{\ell}_{\delta}(D\hat{u}))(Q_{\Phi})_{jn}D^{2}_{j n}\hat{u}^{\alpha}\mathcal{J}_{\Phi}\nonumber \\
&&-\sum_{\beta=1}^{N}\sum_{j,l,m=1}^{n}\frac{\ti{a}_{\delta\varepsilon}'(\hat{\ell}_{\delta}(D\hat{u}))}{2\hat{\ell}_{\delta}(D\hat{u})}(Q_{\Phi})_{jn}\partial_{y_{n}}(Q_{\Phi})_{lm}D_{m}\hat{u}^{\beta}D_{l}\hat{u}^{\beta}D_{j}\hat{u}^{\alpha}\mathcal{J}_{\Phi}\nonumber \\
&&-\sum_{\beta=1}^{N}\sum_{j,l=1}^{n}\sum_{m=1}^{n-1}\frac{\ti{a}_{\delta\varepsilon}'(\hat{\ell}_{\delta}(D\hat{u}))}{\hat{\ell}_{\delta}(D\hat{u})}(Q_{\Phi})_{jn}(Q_{\Phi})_{lm}D^{2}_{nm}\hat{u}^{\beta}D_{l}\hat{u}^{\beta}D_{j}\hat{u}^{\alpha}\mathcal{J}_{\Phi}.
\end{eqnarray*}
Thanks to what we proved so far, the matrix $B$ is H\"older continuous and positive definite and the vector field $G$ is H\"older continuous on the closure of $B_{\tau}^{n-1}(0)\times (-\sigma,0)$, so the linear system \eqref{5.13} can be solved and $D^{2}_{n}\hat{u}$ is H\"older continuous up to $B_{\tau}^{n-1}(0)\times \{0\}$ as well. To summarize, we have just proven that $\hat{u}\in C^{2,\beta_{0}}$-regular on $B_{\tau}^{n-1}(0)\times (-\sigma,0]$ for some $\alpha\in (0,1)$. Transforming back on $\Omega_{\varepsilon}\cap B_{\rr}(x_{0})$ we deduce that $u_{\delta\varepsilon}\in C^{2}(\bar{\Omega}_{\varepsilon}\cap B_{\rr}(x_{0}),\mathbb{R}^{N})$, so in particular $u_{\delta\varepsilon}$ is a classical solution to system \eqref{3.1}. At this stage, classical Schauder type arguments \cite[Section 5.4.5]{giamar} yield that $u_{\delta\varepsilon}\in C^{3}(\bar{\Omega}_{\varepsilon}\cap B_{\rr}(x_{0}),\mathbb{R}^{N})$, and Propositions \ref{p4.1}-\ref{p4.2} are applicable.
\subsection{Convergence and proof of Theorem \ref{t1}} Before proceeding with the analysis of convergence for the sequence of approximating minimizers $\{u_{\delta\varepsilon}\}$, let us point out some useful facts. First, within the setting designed in Section \ref{ri}, standard properties of mollifiers and the fact that we extended $u\equiv 0$ in $\mathbb{R}^{n}\setminus \Omega$ yield that
$$
\nr{D\ti{u}_{\varepsilon}}_{L^{\infty}(\Omega_{\varepsilon}\cap B_{\rr}(x_{0}))}\le \left(\omega_{n}^{1/n}\texttt{d}_{\varepsilon}/4\right)^{-n/p}\nr{Du}_{L^{p}(\Omega\cap B_{\rr+\texttt{d}_{\varepsilon}/4}(x_{0}))},
$$
so we observe:
\begin{eqnarray}\label{5.14}
\texttt{d}_{\varepsilon}^{\frac{2n}{n-1}}\nr{\ell_{\mu_{\delta}}(D\ti{u}_{\varepsilon})}_{L^{q}(\Omega_{\varepsilon}\cap B_{\rr}(x_{0}))}^{q}&\le&2^{q+2n(q-p)/p}\omega_{n}^{(p-q)/p}\texttt{d}_{\varepsilon}^{\frac{2n}{n-1}-\frac{n(q-p)}{p}}\nr{Du}_{L^{p}(\Omega\cap B_{\rr+\texttt{d}_{\varepsilon}/4}(x_{0}))}^{q}\nonumber \\
&&+2^{q+2}\texttt{d}_{\varepsilon}^{\frac{2n}{n-1}}\snr{\Omega_{\varepsilon}\cap B_{\rr+\texttt{d}_{\varepsilon}/4}(x_{0})}\stackrel{\eqref{pq},\eqref{3.8.2}_{1}}{=}\texttt{o}(\varepsilon).
\end{eqnarray}
Moreover, recalling \eqref{3.8.2}$_{1}$, that $\ti{u}_{\varepsilon}\in W^{1,q}(\Omega_{\varepsilon}\cap B_{\rr}(x_{0}),\mathbb{R}^{N})$, Remark \ref{rf}, \cite[Lemma 2.1]{bdms} and \cite[(12.5)$_{2}$ of Lemma 12.2]{demi1}, we have
\begin{flalign}\label{5.15}
 \left\{
\begin{array}{c}
\displaystyle 
\ \nr{F_{\delta}(D\ti{u}_{\varepsilon})-F(D\ti{u}_{\varepsilon})}_{L^{1}(\Omega_{\varepsilon}\cap B_{\rr}(x_{0}))}=\texttt{o}_{\varepsilon}(\delta),\qquad \qquad \nr{F(D\ti{u}_{\varepsilon})-F(Du)}_{L^{1}(\Omega_{\varepsilon}\cap B_{\rr}(x_{0}))}=\texttt{o}(\varepsilon)\\[10pt]\displaystyle
 \ \nr{f_{\varepsilon}-f}_{L^{n}(\Omega_{\varepsilon}\cap B_{\rr}(x_{0}))}=\texttt{o}(\varepsilon),\qquad\qquad \nr{f_{\varepsilon}}_{L^{n}(\Omega_{\varepsilon}\cap B_{\rr}(x_{0}))}\le \nr{f}_{L^{n}(\Omega\cap B_{\rr+\texttt{d}_{\varepsilon}/4}(x_{0}))}\\[10pt]\displaystyle
 \ \nr{\ti{u}_{\varepsilon}}_{W^{1,p}(\Omega_{\varepsilon}\cap B_{\rr}(x_{0}))}\le \nr{u}_{W^{1,p}(\Omega\cap B_{\rr+\texttt{d}_{\varepsilon}/4}(x_{0}))},\qquad \qquad [f_{\varepsilon}]_{n,1;\Omega_{\varepsilon}\cap B_{\rr}(x_{0})}\le [f]_{n,1;\Omega\cap B_{\rr+\texttt{d}_{\varepsilon}/4}(x_{0})}.
 \end{array}
\right.
\end{flalign}
Next, we fix $\varepsilon>0$ and use H\"older, Young and Sobolev-Poincar\'e inequalities, \eqref{2.2}, \eqref{pk}, $\eqref{5.1}_{1}$, \eqref{5.14}-\eqref{5.15}, the minimality of $u_{\delta\varepsilon}$ in Dirichlet class $\ti{u}_{\varepsilon}+W^{1,q}_{0}(\Omega_{\varepsilon}\cap B_{\rr}(x_{0}),\mathbb{R}^{N})$ and the absolute continuity of the Lebesgue integral to control
\begin{flalign*}
&\frac{1}{L}\nr{\ell_{\mu_{\delta}}(Du_{\delta\varepsilon})}_{L^{p}(\Omega_{\varepsilon}\cap B_{\rr}(x_{0}))}^{p}+\texttt{d}_{\varepsilon}^{\frac{2n}{n-1}}\nr{\ell_{\mu_{\delta}}(Du_{\delta\varepsilon})}_{L^{q}(\Omega_{\varepsilon}\cap B_{\rr}(x_{0}))}^{q}\nonumber \\
&\qquad \qquad \qquad \le\mathcal{F}_{\delta\varepsilon}(u_{\delta\varepsilon};\Omega_{\varepsilon}\cap B_{\rr}(x_{0}))+c\nr{f}_{L^{n}(\Omega\cap B_{\rr+\texttt{d}_{\varepsilon}/4}(x_{0}))}\nr{u}_{W^{1,p}(\Omega\cap B_{\rr+\texttt{d}_{\varepsilon}/4}(x_{0}))}\nonumber \\
&\qquad \qquad \qquad \qquad +c\nr{f}_{L^{n}(\Omega\cap B_{\rr+\texttt{d}_{\varepsilon}/4}(x_{0}))}\nr{Du_{\delta\varepsilon}}_{L^{p}(\Omega_{\varepsilon}\cap B_{\rr}(x_{0}))}+c\nonumber \\
&\qquad\qquad\qquad   \le\mathcal{F}(\ti{u}_{\varepsilon};\Omega_{\varepsilon}\cap B_{\rr}(x_{0}))+c\nr{f_{\varepsilon}-f}_{L^{n}(\Omega_{\varepsilon}\cap B_{\rr}(x_{0}))}\nr{\ti{u}}_{W^{1,p}(\Omega\cap B_{\rr+\texttt{d}_{\varepsilon}/4}(x_{0}))}+c\nonumber \\
&\qquad \qquad \qquad \qquad +\nr{F_{\delta}(D\ti{u}_{\varepsilon})-F(D\ti{u}_{\varepsilon})}_{L^{1}(\Omega_{\varepsilon}\cap B_{\rr}(x_{0}))}+\frac{1}{2L}\nr{\ell_{\mu_{\delta}}(Du_{\delta\varepsilon})}_{L^{p}(\Omega_{\varepsilon}\cap B_{\rr}(x_{0}))}^{p}+\texttt{o}(\varepsilon)\nonumber \\
&\qquad \qquad \qquad\qquad +c\nr{u}_{W^{1,p}(\Omega\cap B_{\rr+\texttt{d}_{\varepsilon}/4}(x_{0}))}^{p}+c\nr{f}_{L^{n}(\Omega\cap B_{\rr+\texttt{d}_{\varepsilon}/4}(x_{0}))}^{\frac{p}{p-1}}\nonumber \\
&\qquad \qquad\qquad  \le\frac{1}{2L}\nr{\ell_{\mu_{\delta}}(Du_{\delta\varepsilon})}_{L^{p}(\Omega_{\varepsilon}\cap B_{\rr}(x_{0}))}^{p}+\mathcal{F}(u;\Omega\cap B_{\rr}(x_{0}))+\texttt{o}_{\varepsilon}(\delta)+\texttt{o}(\varepsilon)\nonumber \\
&\qquad \qquad \qquad \qquad +c\nr{u}_{W^{1,p}(\Omega\cap B_{\rr}(x_{0}))}^{p}+c\nr{f}_{L^{n}(\Omega\cap B_{\rr}(x_{0}))}^{\frac{p}{p-1}}+c,
 \end{flalign*}
 for $c\equiv c(n,N,\Omega,\Lambda,p,q)$. Reabsorbing terms in the previous display, we conclude with
 \begin{flalign}\label{5.17}
 &\nr{\ell_{\mu_{\delta}}(Du_{\delta\varepsilon})}_{L^{p}(\Omega_{\varepsilon}\cap B_{\rr}(x_{0}))}^{p}+\texttt{d}_{\varepsilon}^{\frac{2n}{n-1}}\nr{\ell_{\mu_{\delta}}(Du_{\delta\varepsilon})}_{L^{q}(\Omega_{\varepsilon}\cap B_{\rr}(x_{0}))}^{q}\nonumber \\
 &\qquad \qquad \qquad \le c\mathcal{F}(u;\Omega\cap B_{\rr}(x_{0}))+c\nr{u}_{W^{1,p}(\Omega\cap B_{\rr}(x_{0}))}^{p}+c\nr{f}_{L^{n}(\Omega\cap B_{\rr}(x_{0}))}^{\frac{p}{p-1}}+c+\texttt{o}_{\varepsilon}(\delta)+\texttt{o}(\varepsilon)\nonumber \\
 &\qquad \qquad \qquad=:\mf{F}+\texttt{o}_{\varepsilon}(\delta)+\texttt{o}(\varepsilon),
 \end{flalign}
 with $c\equiv c(n,N,\Omega,\Lambda,p,q)$, which means that for fixed $\varepsilon$, sequence $\{u_{\delta\varepsilon}\}$ is uniformly bounded in $W^{1,q}(\Omega_{\varepsilon}\cap B_{\rr}(x_{0}),\mathbb{R}^{N})$, therefore, up to extract a (nonrelabelled) subsequence, we deduce that 
 \eqn{5.18}
$$
 u_{\delta\varepsilon}\rightharpoonup_{\delta\to 0} u_{\varepsilon}\ \  \mbox{weakly in} \ \ W^{1,q}(\Omega_{\varepsilon}\cap B_{\rr}(x_{0}),\mathbb{R}^{N})\qquad \mbox{and}\qquad \left.u_{\varepsilon}\right|_{\partial(\Omega_{\varepsilon}\cap B_{\rr}(x_{0}))}=\left.\ti{u}_{\varepsilon}\right|_{\partial(\Omega_{\varepsilon}\cap B_{\rr}(x_{0}))}.
 $$
We pick any $0< \sigma<\rr$, by \eqref{lip.1} with $\sigma_{2}\equiv \sigma$ and $\sigma_{1}\equiv \rr$, \eqref{5.17} and $\eqref{5.15}_{3}$ it is
\begin{eqnarray}\label{5.19}
\nr{Du_{\delta\varepsilon}}_{L^{\infty}(\Omega_{\varepsilon}\cap B_{\sigma}(x_{0}))}&\le& \frac{c}{(\rr-\sigma)^{a_{1}}}\left(\nr{Du_{\delta\varepsilon}}_{L^{p}(\Omega_{\varepsilon}\cap B_{\rr}(x_{0}))}^{d_{1}}+1\right)+\frac{c}{(\rr-\sigma)^{a_{2}}}[f_{\varepsilon}]_{n,1;\Omega_{\varepsilon}\cap B_{\rr}(x_{0})}^{d_{2}}\nonumber \\
&\le&\frac{c}{(\rr-\sigma)^{a_{1}}}\left(\mf{F}+\texttt{o}_{\varepsilon}(\delta)+\texttt{o}(\varepsilon)\right)^{d_{1}}+\frac{c}{(\rr-\sigma)^{a_{2}}}[f]_{n,1;\Omega\cap B_{\rr+\texttt{d}_{\varepsilon}/4}(x_{0})}^{d_{2}}\nonumber \\
&=:&c\mathcal{M}_{\delta\varepsilon}(\sigma,\rr),
\end{eqnarray}
which means that, up to subsequences,
\begin{flalign}\label{5.20}
u_{\delta\varepsilon}\rightharpoonup^{*} u_{\varepsilon} \ \ \mbox{weakly$^{*}$ in} \ \ W^{1,\infty}(\Omega_{\varepsilon}\cap B_{\sigma}(x_{0}),\mathbb{R}^{N}),\qquad \qquad \left.u_{\varepsilon}\right|_{\partial(\Omega_{\varepsilon}\cap B_{\rr}(x_{0}))}=\left.\ti{u}_{\varepsilon}\right|_{\partial(\Omega_{\varepsilon}\cap B_{\rr}(x_{0}))},
\end{flalign}
and, by weak$^{*}$-lower semicontinuity, we can send $\delta\to 0$ in \eqref{5.19} to get
\begin{flalign}\label{5.19.1}
\nr{Du_{\varepsilon}}_{L^{\infty}(\Omega_{\varepsilon}\cap B_{\sigma}(x_{0}))}\le\frac{c}{(\rr-\sigma)^{a_{1}}}\left(\mf{F}+\texttt{o}(\varepsilon)\right)^{d_{1}}+\frac{c}{(\rr-\sigma)^{a_{2}}}[f]_{n,1;\Omega\cap B_{\rr+\texttt{d}_{\varepsilon}/4}(x_{0})}^{d_{2}}=:c\mathcal{M}_{\varepsilon}(\sigma,\rr).
\end{flalign}
By \eqref{5.2} and \eqref{5.19} we deduce that $\nr{F_{\delta}(Du_{\delta\varepsilon})-F(Du_{\delta\varepsilon})}_{L^{1}(\Omega_{\varepsilon}\cap B_{\sigma}(x_{0}))}=\texttt{o}_{\varepsilon}(\delta)$,
therefore, by weak lower semicontinuity we have
\begin{eqnarray}\label{5.21.2}
\mathcal{F}(u_{\varepsilon};\Omega_{\varepsilon}\cap B_{\sigma}(x_{0}))&\stackrel{\eqref{5.20}_{1}}{\le}&\liminf_{\delta\to 0}\left\{\int_{\Omega_{\varepsilon}\cap B_{\sigma}(x_{0})}\left[F_{\delta}(Du_{\delta\varepsilon})-f_{\varepsilon}\cdot u_{\delta\varepsilon}\right]\dx+\texttt{o}_{\varepsilon}(\delta)\right\}\nonumber \\
&\stackrel{\eqref{5.1}_{1}}{\le} &\limsup_{\delta\to 0}\left\{\int_{\Omega_{\varepsilon}\cap B_{\rr}(x_{0})}\left[F_{\delta}(Du_{\delta\varepsilon})-f_{\varepsilon}\cdot u_{\delta\varepsilon}\right]\dx\right\}\nonumber \\
&&+L\mu^{p}\snr{\Omega_{\varepsilon}\cap (B_{\rr}(x_{0})\setminus B_{\sigma}(x_{0}))}\nonumber \\
&&+\limsup_{\delta\to 0}\int_{\Omega_{\varepsilon}\cap (B_{\rr}(x_{0})\setminus B_{\sigma}(x_{0}))}f_{\varepsilon}\cdot u_{\delta\varepsilon}\dx\nonumber \\
&\stackrel{\eqref{3.8.2}_{1}}{\le}&\limsup_{\delta\to 0}\mathcal{F}_{\delta\varepsilon}(u_{\delta\varepsilon};\Omega_{\varepsilon}\cap B_{\rr}(x_{0}))+\texttt{o}(\varepsilon;\rr-\sigma)\nonumber \\
&&+c\nr{f_{\varepsilon}}_{L^{n}(\Omega_{\varepsilon}\cap (B_{\rr}(x_{0})\setminus B_{\sigma}(x_{0})))}\nr{u}_{W^{1,p}(\Omega\cap B_{\rr+\texttt{d}_{\varepsilon}/4}(x_{0})}\nonumber \\
&&+c\limsup_{\delta\to 0}\nr{f_{\varepsilon}}_{L^{n}(\Omega_{\varepsilon}\cap (B_{\rr}(x_{0})\setminus B_{\sigma}(x_{0})))}\nr{Du_{\delta\varepsilon}-D\ti{u}_{\varepsilon}}_{L^{p}(\Omega_{\varepsilon}\cap B_{\rr}(x_{0}))}\nonumber \\
&\stackrel{\eqref{5.17}}{\le}&\limsup_{\delta\to 0}\mathcal{F}_{\delta\varepsilon}(\ti{u}_{\varepsilon};\Omega_{\varepsilon}\cap B_{\rr}(x_{0}))+\texttt{o}(\varepsilon;\rr-\sigma)\nonumber \\
&&+c\nr{f_{\varepsilon}}_{L^{n}(\Omega_{\varepsilon}\cap (B_{\rr}(x_{0})\setminus B_{\sigma}(x_{0})))}\nr{u}_{W^{1,p}(\Omega\cap B_{\rr+\texttt{d}_{\varepsilon}/4}(x_{0})}\nonumber \\
&&+c\nr{f_{\varepsilon}}_{L^{n}(\Omega_{\varepsilon}\cap (B_{\rr}(x_{0})\setminus B_{\sigma}(x_{0})))}\left(\mf{F}+\texttt{o}(\varepsilon)\right)\nonumber \\
&\stackrel{\eqref{5.14},\eqref{5.15}_{1,2}}{\le}&\texttt{o}(\varepsilon)+\texttt{o}(\varepsilon;\rr-\sigma)+\mathcal{F}(u;\Omega\cap B_{\rr}(x_{0})).
\end{eqnarray}
At this stage, we send $\sigma\to \rr$ in the previous display to conclude with
\begin{flalign}\label{5.21.1}
\mathcal{F}(u_{\varepsilon};\Omega_{\varepsilon}\cap B_{\rr}(x_{0}))\le \mathcal{F}(u;\Omega\cap B_{\rr}(x_{0}))+\texttt{o}(\varepsilon)
\end{flalign}
which in turn yields
\begin{flalign}\label{5.21}
\int_{\Omega\cap B_{\rr}(x_{0})}\ell_{\mu}(Du_{\varepsilon})^{p}\dx&+\int_{\Omega\cap B_{\rr}(x_{0})}F(Du_{\varepsilon})\dx\nonumber \\
&\le c\mathcal{F}(u;\Omega\cap B_{\rr}(x_{0}))+c\nr{f}_{L^{n}(\Omega\cap B_{\rr}(x_{0}))}^{\frac{p}{p-1}}+c\nr{u}_{W^{1,p}(\Omega\cap B_{\rr+\texttt{d}_{\varepsilon}/4}(x_{0}))}^{p}+\texttt{o}(\varepsilon),
\end{flalign}
where we also used $\eqref{5.20}_{2}$, \eqref{f.1}$_{1}$, Sobolev-Poincar\'e and Young inequalities and it is $c\equiv c(n,\Omega,L,p)$. From \eqref{3.8.2}$_{1}$, \eqref{5.21}, \eqref{5.19.1} and the weak continuity of trace operator, we have (up to subsequences) that,
\begin{flalign}\label{5.22}
\begin{cases}
\ u_{\varepsilon}\rightharpoonup \ti{u} \ \ \mbox{weakly in} \ \ W^{1,p}(\Omega\cap B_{\rr}(x_{0}))\\
\ u_{\varepsilon}\rightharpoonup^{*} \ti{u} \ \ \mbox{weakly in} \ \ W^{1,\infty}(\Omega\cap B_{\sigma}(x_{0})) \ \ \mbox{for all} \ \ \sigma\in (0,\rr)\\
\ \left.\ti{u}\right|_{\partial(\Omega\cap B_{\rr}(x_{0}))}=\left.u\right|_{\partial(\Omega\cap B_{\rr}(x_{0}))}.
\end{cases}
\end{flalign}
By \eqref{5.22}$_{2}$ we can pass to the limit as $\varepsilon\to 0$ in \eqref{5.19.1} to get that
\begin{flalign}\label{5.23}
\nr{D\ti{u}}_{L^{\infty}(\Omega\cap B_{\sigma}(x_{0}))}\le c(\rr-\sigma)^{-a_{1}}\mf{F}^{d_{1}}+c(\rr-\sigma)^{-a_{2}}[f]_{n,1;\Omega\cap B_{\rr}(x_{0})}^{d_{2}}=:\mathcal{M}(\sigma,\rr),
\end{flalign}
for all $\sigma\in (0,\rr)$. Moreover, combining $\eqref{5.22}_{1}$ with weak lower semicontinuity in \eqref{5.21.1} we obtain
\begin{flalign}\label{5.26}
\mathcal{F}(\ti{u};\Omega\cap B_{\rr}(x_{0}))\le \liminf_{\varepsilon\to 0}\left\{\mathcal{F}(u;\Omega\cap B_{\rr}(x_{0}))+\texttt{o}(\varepsilon)\right\}\le \mathcal{F}(u;\Omega\cap B_{\rr}(x_{0})).
\end{flalign}
The strict convexity of $F(\cdot)$ implied by $\eqref{f.1}_{2}$, the linearity of term $w\mapsto \int f\cdot w\dx$, \eqref{5.26} and \eqref{5.22}$_{3}$ yield that $\ti{u}=u$ a.e. on $\Omega\cap B_{\rr}(x_{0})$ and \eqref{lipfin} follows from \eqref{5.23}. 
\subsection{Uniform $L^{\infty}$-bounds and proof of Theorem \ref{t2}} Thanks to the construction laid down in Section \ref{ri}, we see that the results in Proposition \ref{p4.2} apply, so $u_{\delta\varepsilon}\in W^{2,2}(\Omega_{\varepsilon}\cap B_{\sigma}(x_{0}),\mathbb{R}^{N})$ for all $0<\sigma<\rr$. Coherently with the content of Section \ref{hd}, we shall separately treat the singular and the degenerate case.
\subsubsection*{Case $1<p\le 2$} Recalling $\eqref{dvp}_{1}$, 
we fix $\varepsilon>0$, $0<\sigma<\rr$ and combine \eqref{3.8.2}$_{1}$, $\eqref{5.15}_{3}$, \eqref{5.17}, \eqref{5.19} and \eqref{hd1} to get
\begin{flalign}\label{5.26.1}
&\nr{D^{2}u_{\delta\varepsilon}}_{L^{2}(\Omega\cap B_{\sigma}(x_{0}))}+\nr{DV_{\mu_{\delta}}(Du_{\delta\varepsilon})}_{L^{2}(\Omega\cap B_{\sigma}(x_{0}))}\nonumber \\
&\qquad \qquad \qquad \qquad \le c\left(1+\nr{Du_{\delta\varepsilon}}_{L^{\infty}(\Omega_{\varepsilon}\cap B_{\rr}(x_{0}))}^{(2-p)/2}\right)\nr{DV_{\mu_{\delta}}(Du_{\delta\varepsilon})}_{L^{2}(\Omega_{\varepsilon}\cap B_{\sigma}(x_{0}))}\nonumber \\
&\qquad \qquad \qquad\qquad  \le c(\rr-\sigma)^{-a}\left(\mf{F}^{d_{1}}+\texttt{o}_{\varepsilon}(\delta)+\texttt{o}(\varepsilon)\right)\left([f]_{n,1;\Omega\cap B_{\rr+\texttt{d}_{\varepsilon}/4}(x_{0})}^{d_{2}}+1\right),
\end{flalign}
for $a,d_{1},d_{1}\equiv a,d_{1},d_{2}(n,p,q)$ and exponents $d_{1},d_{2}$ are not necessarily the same as those appearing in \eqref{hd1} and \eqref{5.19} and will possibly vary from line to line and it is $c\equiv c(n,N,\Omega,\Lambda,p,q)$. Estimates \eqref{5.26.1}, \eqref{5.20} and Rellich-Kondrachev theorem render that $u_{\delta\varepsilon}\rightharpoonup_{\delta\to 0} u_{\varepsilon}$ weakly in $W^{2,2}(\Omega\cap B_{\sigma}(x_{0}),\mathbb{R}^{N})$, $Du_{\delta\varepsilon}\to_{\delta\to 0} Du_{\varepsilon}$ strongly in $L^{2}(\Omega\cap B_{\sigma}(x_{0}),\mathbb{R}^{N\times n})$ and, via a variant of the dominated convergence theorem, also that $V_{\mu_{\delta}}(Du_{\delta\varepsilon})\to_{\delta\to 0}V_{\mu}(Du_{\varepsilon})$ strongly in $L^{2}(\Omega\cap B_{\sigma}(x_{0}),\mathbb{R}^{N\times n})$. This means that we can send $\delta\to 0$ in \eqref{5.26.1} and use weak lower semicontinuity to get
\begin{flalign}\label{5.27}
&\nr{D^{2}u_{\varepsilon}}_{L^{2}(\Omega\cap B_{\sigma}(x_{0}))}+\nr{DV_{\mu}(Du_{\varepsilon})}_{L^{2}(\Omega\cap B_{\sigma}(x_{0}))}\nonumber \\
&\qquad \qquad \qquad\qquad  \le c(\rr-\sigma)^{-a}\left(\mf{F}^{d_{1}}+\texttt{o}(\varepsilon)\right)\left([f]_{n,1;\Omega\cap B_{\rr+\texttt{d}_{\varepsilon}/4}(x_{0})}^{d_{2}}+1\right),
\end{flalign}
which, together with \eqref{5.22} in turn implies that $u_{\varepsilon}\rightharpoonup_{\varepsilon\to 0} u$ weakly in $W^{2,2}(\Omega\cap B_{\sigma}(x_{0}),\mathbb{R}^{N})$, $Du_{\varepsilon}\to_{\varepsilon\to 0} Du$ strongly in $L^{2}(\Omega\cap B_{\sigma}(x_{0}),\mathbb{R}^{N\times n})$ and that $V_{\mu}(Du_{\varepsilon})\to_{\varepsilon\to 0}V_{\mu}(Du)$ strongly in $L^{2}(\Omega\cap B_{\sigma}(x_{0}),\mathbb{R}^{N\times n})$, so we can finally send $\varepsilon\to 0$ in \eqref{5.27} to conclude with \eqref{t2.1}.

\subsubsection*{Case $p>2$} 
As done for the singular case, we fix $\varepsilon>0$, let $0<\sigma<\rr$ and observe that by \eqref{5.1} and \eqref{5.7}, inequalities \eqref{hd2} with $r_{2}:=\sigma$ and $r_{1}=\rr$, \eqref{5.15}$_{3}$ and \eqref{5.17} apply so
\begin{flalign}\label{p.1}
\nr{DW_{\mu_{\delta}}(Du_{\delta\varepsilon})}_{L^{2}(\Omega\cap B_{\sigma}(x_{0}))}\le \frac{c}{(\rr-\sigma)^{a_{1}}}\left(\mf{F}^{d_{1}}+\texttt{o}_{\delta}(\varepsilon)+\texttt{o}(\varepsilon)\right)+\frac{c}{(\rr-\sigma)^{a_{2}}}[f]_{n,1;\Omega\cap B_{\rr+\texttt{d}_{\varepsilon}/4}(x_{0})}^{d_{2}}
\end{flalign}
holds with $c\equiv c(n,N,\Omega,\Lambda,p,q)$, $a_{1},a_{2}\equiv a_{1},a_{2}(n,p,q)$ and $d_{1},d_{2}\equiv d_{1},d_{2}(n,p,q)$. By Rellich-Kondrachov theorem we deduce that $W_{\mu_{\delta}}(Du_{\delta\varepsilon})\rightharpoonup_{\delta\to 0} W_{\varepsilon}$ weakly in $W^{1,2}(\Omega\cap B_{\sigma}(x_{0}),\mathbb{R}^{N\times n})$ and, as $\varepsilon\to 0$ in \eqref{p.1}, by weak lower semicontinuity we also have
\begin{flalign}\label{p.2}
\nr{DW_{\varepsilon}}_{L^{2}(\Omega\cap B_{\sigma}(x_{0}))}\le \frac{c}{(\rr-\sigma)^{a_{1}}}\left(\mf{F}^{d_{1}}+\texttt{o}(\varepsilon)\right)+\frac{c}{(\rr-\sigma)^{a_{2}}}[f]_{n,1;\Omega\cap B_{\rr+\texttt{d}_{\varepsilon}/4}(x_{0})}^{d_{2}},
\end{flalign}
which in turn implies that $W_{\varepsilon}\rightharpoonup W$ weakly in $W^{1,2}(\Omega\cap B_{\sigma}(x_{0}),\mathbb{R}^{N\times n})$ and, again by weak lower semicontinuity,
\begin{flalign}\label{p.3}
\nr{DW}_{L^{2}(\Omega\cap B_{\sigma}(x_{0}))}\le \frac{c\mf{F}^{d_{1}}}{(\rr-\sigma)^{a_{1}}}+\frac{c}{(\rr-\sigma)^{a_{2}}}[f]_{n,1;\Omega\cap B_{\rr}(x_{0})}^{d_{2}}
\end{flalign}
is verified for $c\equiv c(n,N,\Omega,\Lambda,p,q)$, $a_{1},a_{2}\equiv a_{1},a_{2}(n,p,q)$ and $d_{1},d_{2}\equiv d_{1},d_{2}(n,p,q)$. As a consequence of \eqref{p.1}, \eqref{p.2} and of Rellich-Kondrachov theorem we have in particular that
\begin{flalign}\label{p.0}
W_{\delta\varepsilon}(Du_{\delta\varepsilon})\to_{\delta\to 0} W_{\varepsilon}, \ \ W_{\varepsilon}\to_{\varepsilon\to 0} W\qquad \mbox{strongly in} \ \ L^{2}(\Omega\cap B_{\sigma}(x_{0}),\mathbb{R}^{N\times n}). 
\end{flalign}
Before proceeding further, let us extract from display \eqref{5.21.2} a couple of useful inequalities:
\begin{eqnarray}\label{y.0}
\mathcal{F}(u_{\varepsilon};\Omega_{\varepsilon}\cap B_{\sigma}(x_{0}))&\le&\texttt{o}(\varepsilon)+\texttt{o}(\varepsilon;\rr-\sigma)+\liminf_{\delta\to 0}\int_{\Omega_{\varepsilon}\cap B_{\rr}(x_{0})}\left[F_{\delta}(Du_{\delta\varepsilon})-f_{\varepsilon}\cdot u_{\delta\varepsilon}\right]\dx\nonumber \\
&\le&\texttt{o}(\varepsilon)+\texttt{o}(\varepsilon;\rr-\sigma)+\limsup_{\delta\to 0}\int_{\Omega_{\varepsilon}\cap B_{\rr}(x_{0})}\left[F_{\delta}(Du_{\delta\varepsilon})-f_{\varepsilon}\cdot u_{\delta\varepsilon}\right]\dx\nonumber \\
&\le&\texttt{o}(\varepsilon)+\texttt{o}(\varepsilon;\rr-\sigma)+\mathcal{F}(u;\Omega\cap B_{\rr}(x_{0}))
\end{eqnarray}
and
\begin{eqnarray}\label{y.1}
\mathcal{F}(u_{\varepsilon};\Omega_{\varepsilon}\cap B_{\sigma}(x_{0}))&\le&\texttt{o}(\varepsilon)+\texttt{o}(\varepsilon;\rr-\sigma)+\liminf_{\delta\to 0}\mathcal{F}_{\delta\varepsilon}(u_{\delta\varepsilon};\Omega_{\varepsilon}\cap B_{\rr}(x_{0}))\nonumber \\
&\le&\texttt{o}(\varepsilon)+\texttt{o}(\varepsilon;\rr-\sigma)+\limsup_{\delta\to 0}\mathcal{F}_{\delta\varepsilon}(u_{\delta\varepsilon};\Omega_{\varepsilon}\cap B_{\rr}(x_{0}))\nonumber\\
&\le&\texttt{o}(\varepsilon)+\texttt{o}(\varepsilon;\rr-\sigma)+\mathcal{F}(u;\Omega\cap B_{\rr}(x_{0})),
\end{eqnarray}
proving that for each fixed $\varepsilon>0$, the $\limsup$/$\liminf$ displayed above remain bounded. Combining the content of the two previous displays with \eqref{3.8.2}$_{1}$ and using \eqref{5.22}$_{1,2}$, weak lower semicontinuity and the absolute continuity of Lebesgue integral we obtain
\begin{eqnarray*}
 \mathcal{F}(u;\Omega\cap B_{\rr}(x_{0}))&\le&\mathcal{F}(u;\Omega\cap B_{\sigma}(x_{0}))+\mathcal{F}(u;\Omega\cap B_{\rr}(x_{0})\setminus B_{\sigma}(x_{0}))\nonumber \\
 &\le&\texttt{o}(\rr-\sigma)+\liminf_{\varepsilon\to 0}\mathcal{F}(u_{\varepsilon};\Omega_{\varepsilon}\cap B_{\sigma}(x_{0}))\nonumber \\
  &\le&\texttt{o}(\rr-\sigma)+\liminf_{\varepsilon\to 0}\liminf_{\delta\to 0}\int_{\Omega_{\varepsilon}\cap B_{\rr}(x_{0})}\left[F_{\delta}(Du_{\delta\varepsilon})-f_{\varepsilon}\cdot u_{\delta\varepsilon}\right]\dx\nonumber \\
 &\le&\texttt{o}(\rr-\sigma)+\limsup_{\varepsilon\to 0}\limsup_{\delta\to 0}\int_{\Omega_{\varepsilon}\cap B_{\rr}(x_{0})}\left[F_{\delta}(Du_{\delta\varepsilon})-f_{\varepsilon}\cdot u_{\delta\varepsilon}\right]\dx\nonumber \\
 &\le& \texttt{o}(\rr-\sigma)+\mathcal{F}(u;\Omega\cap B_{\rr}(x_{0}))
\end{eqnarray*}
and
\begin{eqnarray*}
 \mathcal{F}(u;\Omega\cap B_{\rr}(x_{0}))&\le&\mathcal{F}(u;\Omega\cap B_{\sigma}(x_{0}))+\mathcal{F}(u;\Omega\cap B_{\rr}(x_{0})\setminus B_{\sigma}(x_{0}))\nonumber \\
 &\le&\texttt{o}(\rr-\sigma)+\liminf_{\varepsilon\to 0}\mathcal{F}(u_{\varepsilon};\Omega_{\varepsilon}\cap B_{\sigma}(x_{0}))\nonumber \\
 &\le&\texttt{o}(\rr-\sigma)+\liminf_{\varepsilon\to 0}\liminf_{\delta\to 0}\mathcal{F}_{\delta\varepsilon}(u_{\delta\varepsilon};\Omega_{\varepsilon}\cap B_{\rr}(x_{0}))\nonumber \\
 &\le&\texttt{o}(\rr-\sigma)+\limsup_{\varepsilon\to 0}\limsup_{\delta\to 0}\mathcal{F}_{\delta\varepsilon}(u_{\delta\varepsilon};\Omega_{\varepsilon}\cap B_{\rr}(x_{0}))\nonumber \\
 &\le& \texttt{o}(\rr-\sigma)+\mathcal{F}(u;\Omega\cap B_{\rr}(x_{0})).
\end{eqnarray*}
Sending $\sigma\to \rr$ in the two above inequalities, we obtain
\begin{flalign}\label{5.29}
\limsup_{\varepsilon\to 0}\limsup_{\delta\to 0}&\int_{\Omega_{\varepsilon}\cap B_{\rr}(x_{0})}\left[F_{\delta}(Du_{\delta\varepsilon})-f_{\varepsilon}\cdot u_{\delta\varepsilon}\right]\dx=\limsup_{\varepsilon\to 0}\limsup_{\delta\to 0}\mathcal{F}_{\delta\varepsilon}(u_{\delta\varepsilon};\Omega_{\varepsilon}\cap B_{\rr}(x_{0}))\nonumber \\
&=\liminf_{\varepsilon\to 0}\liminf_{\delta\to 0}\int_{\Omega_{\varepsilon}\cap B_{\rr}(x_{0})}\left[F_{\delta}(Du_{\delta\varepsilon})-f_{\varepsilon}\cdot u_{\delta\varepsilon}\right]\dx=\liminf_{\varepsilon\to 0}\liminf_{\delta\to 0}\mathcal{F}_{\delta\varepsilon}(u_{\delta\varepsilon};\Omega_{\varepsilon}\cap B_{\rr}(x_{0}))\nonumber \\
&=\mathcal{F}(u;\Omega\cap B_{\rr}(x_{0})).
\end{flalign}
We then observe that \eqref{y.0}-\eqref{y.1} and \eqref{5.29} allow 
using well-known properties of the limit superior, that yield:
\begin{flalign}
&\limsup_{\varepsilon\to 0}\limsup_{\delta\to 0}\texttt{d}_{\varepsilon}^{\frac{2n}{n-1}}\int_{\Omega_{\varepsilon}\cap B_{\rr}(x_{0})}\ell_{\mu_{\delta}}(Du_{\delta\varepsilon})^{q}\dx\nonumber \\
&\qquad \qquad =\limsup_{\varepsilon\to 0}\limsup_{\delta\to 0}\left\{\mathcal{F}_{\delta\varepsilon}(Du_{\delta\varepsilon};\Omega_{\varepsilon}\cap B_{\rr}(x_{0}))-\int_{\Omega_{\varepsilon}\cap B_{\rr}(x_{0})}\left[F_{\delta}(Du_{\delta\varepsilon})-f_{\varepsilon}\cdot u_{\delta\varepsilon}\right]\dx\right\}\nonumber \\
&\qquad \qquad \le\limsup_{\varepsilon\to 0}\limsup_{\delta\to 0}\mathcal{F}_{\delta\varepsilon}(Du_{\delta\varepsilon};\Omega_{\varepsilon}\cap B_{\rr}(x_{0}))\nonumber \\
&\qquad \qquad \qquad -\liminf_{\varepsilon\to 0}\liminf_{\delta\to 0}\int_{\Omega_{\varepsilon}\cap B_{\rr}(x_{0})}\left[F_{\delta}(Du_{\delta\varepsilon})-f_{\varepsilon}\cdot u_{\delta\varepsilon}\right]\dx\stackrel{\eqref{5.29}}{=}0.\label{5.30}
\end{flalign}
Moreover, as a direct consequence of $\eqref{3.8.2}_{1}$, of the weak convergences in \eqref{5.20}, \eqref{5.22} and Rellich–Kondrachov theorem, and of $\eqref{5.15}_{2}$ we have
\begin{eqnarray}\label{5.31}
\int_{\Omega_{\varepsilon}\cap B_{\rr}(x_{0})}\snr{f_{\varepsilon}}\snr{\ti{u}_{\varepsilon}-u_{\delta\varepsilon}}\dx
=\texttt{o}(\varepsilon)+\texttt{o}_{\varepsilon}(\delta).
\end{eqnarray}
By means of \eqref{5.31} and \eqref{5.14} it is
\begin{eqnarray*}
\left| \ \int_{\Omega_{\varepsilon}\cap B_{\rr}(x_{0})}\left[F_{\delta}(Du_{\delta\varepsilon})-F_{\delta}(D\ti{u}_{\varepsilon})\right]\dx \ \right|&\le& \texttt{o}(\varepsilon)+\texttt{o}_{\varepsilon}(\delta)+\texttt{d}_{\varepsilon}^{\frac{2n}{n-1}}\int_{\Omega_{\varepsilon}\cap B_{\rr}(x_{0})}\ell_{\mu_{\delta}}(Du_{\delta\varepsilon})^{q}\dx\nonumber\\
&&+\mathcal{F}_{\delta\varepsilon}(\ti{u}_{\varepsilon};\Omega_{\varepsilon}\cap B_{\rr}(x_{0}))-\mathcal{F}_{\delta\varepsilon}(u_{\delta\varepsilon};\Omega_{\varepsilon}\cap B_{\rr}(x_{0})),
\end{eqnarray*}
where we also used the minimality of $u_{\delta\varepsilon}$ in the Dirichlet class $\ti{u}_{\varepsilon}+W^{1,q}_{0}(\Omega_{\varepsilon}\cap B_{\rr}(x_{0}),\mathbb{R}^{N})$, so we get
\begin{flalign*}
&\limsup_{\delta\to 0}\left| \ \int_{\Omega_{\varepsilon}\cap B_{\rr}(x_{0})}\left[F_{\delta}(Du_{\delta\varepsilon})-F_{\delta}(D\ti{u}_{\varepsilon})\right]\dx \ \right|\le \texttt{o}(\varepsilon)+\limsup_{\delta\to 0}\texttt{d}_{\varepsilon}^{\frac{2n}{n-1}}\int_{\Omega_{\varepsilon}\cap B_{\rr}(x_{0})}\ell_{\mu_{\delta}}(Du_{\delta\varepsilon})^{q}\dx\nonumber \\
&\qquad \qquad \qquad \qquad \qquad \qquad  +\limsup_{\delta\to 0}\left(\mathcal{F}_{\delta\varepsilon}(\ti{u}_{\varepsilon};\Omega_{\varepsilon}\cap B_{\rr}(x_{0}))-\mathcal{F}_{\delta\varepsilon}(u_{\delta\varepsilon};\Omega_{\varepsilon}\cap B_{\rr}(x_{0}))\right),
\end{flalign*}
which, together with \eqref{5.30}, \eqref{5.29}, \eqref{5.14} and \eqref{5.15}$_{1}$ eventually yield
\begin{flalign}\label{5.32}
  &\limsup_{\varepsilon\to 0}\limsup_{\delta\to 0}\left| \ \int_{\Omega_{\varepsilon}\cap B_{\rr}(x_{0})}\left[F_{\delta}(Du_{\delta\varepsilon})-F_{\delta}(D\ti{u}_{\varepsilon})\right]\dx \ \right|\nonumber \\
  &\qquad \qquad \qquad \le \limsup_{\varepsilon\to 0}\limsup_{\delta\to 0}\left(\mathcal{F}_{\delta\varepsilon}(\ti{u}_{\varepsilon};\Omega_{\varepsilon}\cap B_{\rr}(x_{0}))-\mathcal{F}_{\delta\varepsilon}(u_{\delta\varepsilon};\Omega_{\varepsilon}\cap B_{\rr}(x_{0}))\right)=0.
\end{flalign}
All the previous considerations hold for all $0<\sigma<\rr$. Next, we notice that by $\eqref{3.8.2}_{1}$ there is no loss of generality in assuming $\rr>2\texttt{d}_{\varepsilon}$, so now we let $0<\sigma<\rr-\texttt{d}_{\varepsilon}/2$ and recall that $u_{\delta\varepsilon}$ solves the integral identity \eqref{ela} and combine standard convexity/monotonicity arguments, \eqref{dvp.1}$_{1}$, $\eqref{5.1}_{2}$, \eqref{5.14} and \eqref{5.31} to get
\begin{flalign}\label{5.35}
&\mathcal{V}^{2}:=\int_{\Omega_{\varepsilon}\cap B_{\rr}(x_{0})}\snr{V_{\mu_{\delta}}(Du_{\delta\varepsilon})-V_{\mu_{\delta}}(D\ti{u}_{\varepsilon})}^{2}\dx\nonumber \\
&\quad \quad \quad \le c\int_{\Omega_{\varepsilon}\cap B_{\rr}(x_{0})}F_{\delta}(D\ti{u}_{\varepsilon})-F_{\delta}(Du_{\delta\varepsilon})-\langle \partial F_{\delta}(Du_{\delta\varepsilon}),D\ti{u}_{\varepsilon}-Du_{\delta\varepsilon}\rangle\dx\nonumber \\
&\quad \quad \quad =c\int_{\Omega_{\varepsilon}\cap B_{\rr}(x_{0})}F_{\delta}(D\ti{u}_{\varepsilon})-F_{\delta}(Du_{\delta\varepsilon})-f_{\varepsilon}\cdot (\ti{u}_{\varepsilon}-u_{\delta\varepsilon})\dx\nonumber \\
  &\quad \qquad \quad \quad +c\texttt{d}_{\varepsilon}^{\frac{2n}{n-1}}\int_{\Omega_{\varepsilon}\cap B_{\rr}(x_{0})}\ell_{\mu_{\delta}}(Du_{\delta\varepsilon})^{q-2}\langle Du_{\delta\varepsilon},D\ti{u}_{\varepsilon}-Du_{\delta\varepsilon}\rangle\dx\nonumber \\
  &\quad \quad \quad \le c\left| \ \int_{\Omega_{\varepsilon}\cap B_{\rr}(x_{0})}\left[F_{\delta}(D\ti{u}_{\varepsilon})-F_{\delta}(Du_{\delta\varepsilon})\right]\dx \ \right|+c\texttt{d}_{\varepsilon}^{\frac{2n}{n-1}}\int_{\Omega_{\varepsilon}\cap B_{\rr}(x_{0})}\ell_{\mu_{\delta}}(Du_{\delta\varepsilon})^{q}\dx+\texttt{o}_{\varepsilon}(\delta)+\texttt{o}(\varepsilon),
\end{flalign}
with $c\equiv c(n,N,\Omega,\Lambda,p,q)$. Next, using triangle inequality, \eqref{dvp.1}$_{2}$, \eqref{5.35}, \eqref{5.19} and \eqref{5.23} we estimate
\begin{flalign*}
&\int_{\Omega\cap B_{\sigma}(x_{0})}\snr{W_{\mu_{\delta}}(Du_{\delta\varepsilon})-W_{\mu}(Du)}^{2}\dx\le c\int_{\Omega\cap B_{\sigma}(x_{0})}\snr{W_{\mu_{\delta}}(Du_{\delta\varepsilon})-W_{\mu_{\delta}}(D\ti{u}_{\varepsilon})}^{2}\dx\nonumber \\
&\qquad \qquad \qquad \qquad +c\int_{\Omega\cap B_{\sigma}(x_{0})}\snr{W_{\mu_{\delta}}(D\ti{u}_{\varepsilon})-W_{\mu}(D\ti{u}_{\varepsilon})}^{2}\dx+c\int_{\Omega\cap B_{\sigma}(x_{0})}\snr{W_{\mu}(Du)-W_{\mu}(D\ti{u}_{\varepsilon})}^{2}\dx\nonumber \\
&\qquad \qquad \qquad \le c\left(1+\nr{Du_{\delta\varepsilon}}_{L^{\infty}(\Omega\cap B_{\sigma}(x_{0}))}^{p-2}+\nr{Du}_{L^{\infty}(\Omega\cap B_{\sigma+\texttt{d}_{\varepsilon}/4}(x_{0}))}^{p-2}\right)\mathcal{V}^{2}+ \texttt{o}_{\varepsilon}(\delta)\nonumber \\
&\qquad \qquad \qquad \qquad +\texttt{o}(\varepsilon)\left(1+\nr{Du}_{L^{\infty}(\Omega\cap B_{\sigma+\texttt{d}_{\varepsilon}/4}(x_{0}))}^{p-2}\right)\nonumber \\
&\qquad \qquad \qquad \le c\mathcal{M}_{\delta\varepsilon}(\sigma+\texttt{d}_{\varepsilon}/4,\rr)^{p-2}\left| \ \int_{\Omega_{\varepsilon}\cap B_{\rr}(x_{0})}\left[F_{\delta}(D\ti{u}_{\varepsilon})-F_{\delta}(Du_{\delta\varepsilon})\right]\dx \ \right|\nonumber \\
&\qquad \qquad \qquad \qquad +c\mathcal{M}_{\delta\varepsilon}(\sigma+\texttt{d}_{\varepsilon}/4,\rr)^{p-2}\texttt{d}_{\varepsilon}^{\frac{2n}{n-1}}\int_{\Omega_{\varepsilon}\cap B_{\rr}(x_{0})}\ell_{\mu_{\delta}}(Du_{\delta\varepsilon})^{q}\dx\nonumber \\
&\qquad \qquad \qquad \qquad +c\mathcal{M}_{\delta\varepsilon}(\sigma+\texttt{d}_{\varepsilon}/4,\rr)^{p-2}\left(\texttt{o}_{\varepsilon}(\delta)+\texttt{o}(\varepsilon)\right),
\end{flalign*}
for $c\equiv c(n,N,\Omega,\Lambda,p,q)$ and $\mathcal{M}_{\delta\varepsilon}(\sigma+\texttt{d}_{\varepsilon}/4,\rr)$ being defined in \eqref{5.19}. We then send $\delta\to 0$ above, exploit Fatou Lemma and the strong $L^{2}$-convergence in \eqref{p.0} to get
\begin{flalign*}
  \int_{\Omega\cap B_{\sigma}(x_{0})}\snr{W_{\varepsilon}-W_{\mu}(Du)}^{2}\dx\le& \limsup_{\delta\to 0}\int_{\Omega\cap B_{\sigma}(x_{0})}\snr{W_{\mu_{\delta}}(Du_{\delta\varepsilon})-W_{\mu}(Du)}^{2}\dx\nonumber \\
  \le &c\mathcal{M}_{\varepsilon}(\sigma+\texttt{d}_{\varepsilon}/4,\rr)^{p-2}\limsup_{\delta\to 0}\left|\ \int_{\Omega_{\varepsilon}\cap B_{\rr}(x_{0})}\left[F_{\delta}(D\ti{u}_{\varepsilon})-F_{\delta}(Du_{\delta\varepsilon})\right]\dx\ \right|\nonumber \\
  &+c\mathcal{M}_{\varepsilon}(\sigma+\texttt{d}_{\varepsilon}/4,\rr)^{p-2}\limsup_{\delta\to 0}\texttt{d}_{\varepsilon}^{\frac{2n}{n-1}}\int_{\Omega_{\varepsilon}\cap B_{\rr}(x_{0})}\ell_{\mu_{\delta}}(Du_{\delta\varepsilon})^{q}\dx\nonumber \\
  &+\texttt{o}(\varepsilon)\mathcal{M}_{\varepsilon}(\sigma+\texttt{d}_{\varepsilon}/4,\rr)^{p-2},
\end{flalign*}
for $c\equiv c(n,N,\Omega,\Lambda,p,q)$ and this time $\mathcal{M}_{\varepsilon}(\sigma+\texttt{d}_{\varepsilon}/4,\rr)$ is defined in \eqref{5.19.1}. Keeping this and \eqref{p.0}$_{2}$ in mind, we can apply Fatou lemma, pass to the limit as $\varepsilon\to 0$ in the previous display and use \eqref{5.30}-\eqref{5.32} to deduce that
\begin{flalign*}
\int_{\Omega\cap B_{\sigma}(x_{0})}\snr{W-W_{\mu}(Du)}^{2}\dx\le& \limsup_{\varepsilon\to 0}\int_{\Omega\cap B_{\sigma}(x_{0})}\snr{W_{\varepsilon}-W_{\mu}(Du)}^{2}\dx\nonumber \\
\le &c\mathcal{M}(\sigma,\rr)^{p-2}\limsup_{\varepsilon\to 0}\limsup_{\delta\to 0}\left|\ \int_{\Omega_{\varepsilon}\cap B_{\rr}(x_{0})}\left[F_{\delta}(D\ti{u}_{\varepsilon})-F_{\delta}(Du_{\delta\varepsilon})\right]\dx\ \right|\nonumber \\
&+c\mathcal{M}(\sigma,\rr)^{p-2}\limsup_{\varepsilon\to 0}\limsup_{\delta\to 0}\texttt{d}_{\varepsilon}^{\frac{2n}{n-1}}\int_{\Omega_{\varepsilon}\cap B_{\rr}(x_{0})}\ell_{\mu_{\delta}}(Du_{\delta\varepsilon})^{q}\dx=0,
\end{flalign*}
therefore $W=W_{\mu}(Du)$ a.e. in $\Omega\cap B_{\sigma}(x_{0})$ for all $\sigma\in (0,\rr)$ and \eqref{p.3} yields \eqref{t2.2}. Finally, from \eqref{dvp}$_{2}$ we obtain that if $\mu>0$ in \eqref{f.1}, then \eqref{t2.2} also implies that $u\in W^{2,2}(\Omega\cap B_{\sigma}(x_{0}),\mathbb{R}^{N})$ for all $0<\sigma<\rr$, with bounding constants in \eqref{t2.3} depending also on $\mu$. The proof is complete.
  
\end{document}